\documentclass[a4paper]{article}
\usepackage[dvipdfmx]{graphicx}
\usepackage{amsmath,amsthm,txfonts,pxfonts,mathrsfs,comment,mathtools}
\usepackage{float}
\newtheorem{defi}{Definition}
\newtheorem{prp}[defi]{Proposition}
\newtheorem{thm}[defi]{Theorem}
\newtheorem{cor}[defi]{Corollary}
\newtheorem{lmm}[defi]{Lemma}
\newtheorem*{exap}{Example}
\newtheorem*{rmk}{Remark}

\newtheorem{con}{Condition}

\title{Error Distribution for One-Dimensional Stochastic Differential Equation Driven By Fractional Brownian Motion}
\author{UEDA Kento}

\newcommand{\bitime}[1]{\tau_{#1}^{(m)}}
\newcommand{\bitimelast}[1]{\bitime{#1}\wedge t}
\newcommand{\timeunit}{(t-\bitime{r})}
\newcommand{\fBminc}[1]{\delta B_{#1}}
\newcommand{\fBmitin}{B}
\newcommand{\fBmitinunit}{\fBmitin_{\bitime{r},t}}
\newcommand{\fBmunit}{\fBminc{\bitime{r},t}}
\newcommand{\fBmdotunit}{\fBminc{\cdot_1,\cdot_2}}
\newcommand{\timeinunit}{\bitime{r}<t\leq \bitime{r+1}}
\newcommand{\nusol}[1]{\hat{Y}_{#1}^{(m)}}
\newcommand{\nusoldis}[1]{\nusol{\bitime{#1}}}
\newcommand{\exsoldis}[1]{\exsol{\bitime{#1}}}
\newcommand{\exsol}[1]{Y_{#1}}
\newcommand{\Jacobi}[1]{J_{#1}}
\newcommand{\polpar}{\rho}
\newcommand{\polsol}[2]{Y_{#1}^{(m,#2)}}
\newcommand{\polsolstd}[1]{\polsol{#1}{\polpar}}
\newcommand{\polsolstddis}[1]{\polsol{\bitime{#1}}{\polpar}}
\newcommand{\polJacobi}[2]{M_{#1}^{(m,#2)}}
\newcommand{\polJacobistd}[1]{M_{#1}^{(m,\polpar)}}
\newcommand{\error}[1]{\nusol{#1}-\exsol{#1}}
\newcommand{\Jacobiunit}[1]{\mu_{#1}^{(m,\rho)}}
\newcommand{\moment}[1]{\mu_{#1}}
\newcommand{\xunit}{\delta x_{s,t}}
\begin{document}
\maketitle

\begin{abstract}
This paper deals with asymptotic errors, limit theorems for errors between numerical and exact solutions of stochastic differential equation (SDE) driven by one-dimensional fractional Brownian motion (fBm).
The Euler-Maruyama, higher-order Milstein, and Crank-Nicolson schemes are among the most studied numerical schemes for SDE (fSDE) driven by fBm.
Most previous studies of asymptotic errors have derived specific asymptotic errors for these schemes as main theorems or their corollary.
Even in the one-dimensional case, the asymptotic error was not determined for the Milstein or the Crank-Nicolson method when the Hurst exponent is less than or equal to $1/3$ with a drift term.
We obtained a new evaluation method for convergence and asymptotic errors.
This evaluation method improves the conditions under which we can prove convergence of the numerical scheme and obtain the asymptotic error under the same conditions.
We have completely determined the asymptotic error of the Milstein method for arbitrary orders.
In addition, we have newly determined the asymptotic error of the Crank-Nicolson method for $1/4<H\leq 1/3$.
\end{abstract}

\newpage
%この論文で扱うこと

%Error distribution can be considered in multi-dimensional cases, but we consider it only in one-dimensional cases in this paper.
\section{Introduction}
The fBm is a central Gaussian process with Hurst parameter $H\in (0,1)$ and characterized by the following covariance:
\begin{equation}
\text{Cov}^B(s,t) = E[B_sB_t] = \frac{1}{2}(|s|^{2H}+|t|^{2H}-|t-s|^{2H}).\label{Covariance_of_fBm}
\end{equation}
The fBm is a self-similar and stationary increment process and generalization of Brownian motion, which includes Brownian motion as the case where the Hurst parameter is $1/2$.
This process is non-Markovian when the Hurst parameter is not equal to $1/2$. Moreover, the fBm has long-term dependence if $H>1/2$ and short-term dependence if $H<1/2$.
The Hurst parameter is named after the Hurst exponent for stochastic processes, which measures the long-term dependence of a stochastic process.
The Hurst exponent of fBm corresponds to the Hurst parameter.
The fBm is expected to be a model for non-Markovian time series data, which are not only of mathematical interest but are also observed as real phenomena, such as the water level of the Nile River.
The Hurst exponent was introduced to predict this water level.

We consider SDEs driven by fBm:
\[dY_t=\sigma(Y_t)dB_t+b(Y_t)dt,\]
where $\sigma$ is a diffusion coefficient and $b$ is a drift coefficient.
The simplest SDEs are driven by standard Brownian motion $W_t$:
\[dY_t=\sigma(Y_t,t)dW_t+b(Y_t,t)dt.\]
Various time series models use equations of this form.
Because the solution to this equation is a Markov process for any coefficient, we wish to generalize these SDEs to represent non-Markovian phenomena.
The most natural extension is to replace the standard Brownian motion with fBm.
However, we failed to define SDEs driven by fBm as in the case of Brownian motion.
The standard SDE is well-defined by the martingale theory.
In this theory, any semimartingale $X_t$ can be a driving process.
Also, the solution of the equation is a semimartingale, but fBm is not a semimartingale.
Nevertheless, a general theory for stochastic integration by fBm has already been established.

We will define the stochastic integral driven by fBm as equivalent to the symmetric integral.
The symmetric integral is an extension of the Stratonovich-type stochastic integral for Brownian motion and the first formulation of the stochastic integral by one-dimensional fBm with arbitrary Hurst parameters.
Although the one-dimensional rough integral is degenerate and equivalent to the symmetric integral, the rough integral notation makes it easier to treat the solution of the SDE.

We assume the initial value of the numerical solution $\exsol{0}$ is already known and the initial values of numerical solutions $\nusol{0}$ is equals to $\exsol{0}$.
Also, we set to be $r\in \mathbb{Z}_{\geq0}$ and $\tau_r^{(m)}=2^{-m}r$.
Moreover, we define numerical solutions for any $r$ and $\timeinunit$ inductively.
Then, the following three numerical schemes have been studied in detail in the numerical analysis of SDEs, especially in the study of asymptotic errors:
\begin{itemize}
\item the Euler-Maruyama scheme
\begin{align}
\nusol{t}=\nusoldis{r}+\sigma(\nusoldis{r})\fBmunit +b(\nusoldis{r})\timeunit,
\end{align}
\item the (k)-Milstein scheme
\begin{align}
\nusol{t}=&\nusoldis{r}+\sum_{l=1}^{k} \frac{1}{l!}\mathcal{D}^{l-1}\sigma(\nusoldis{r}) (\fBmunit)^l+b(\nusoldis{r})\timeunit\nonumber\\
&+\frac{1}{2}bb'(\nusoldis{r})\timeunit^2+\frac{1}{2}(\sigma'b+b'\sigma)(\nusoldis{r})\timeunit \fBmunit,\label{defMilstein}
\end{align}
\item the Crank-Nicolson scheme
\begin{align}
\nusol{t}=\nusoldis{r}+\frac{1}{2}(\sigma(\nusoldis{r})+\sigma(\nusol{t}))\fBmunit +\frac{1}{2}(b(\nusoldis{r})+b(\nusol{t}))\timeunit.\label{defCN}
\end{align}
We remark that the Crank-Nicolson scheme is an implicit scheme that is well-defined when $(\|\sigma\|_\infty\|B_\cdot\|_{H_-}2^{-mH_-}+\|b\|_\infty 2^{-m})<1$.
\end{itemize}
These three are natural extensions of the Euler scheme, the $k$th-order Taylor scheme, and the trapezoidal rule in ODEs, respectively.

Now, we can naturally consider the limit theorem for the numerical solution $\hat{Y}^{(n)}$ in $n\to\infty$.
Among such limit theorems, those corresponding to asymptotic expansions of the first order are called asymptotic expansions.
For example, the asymptotic error of the Euler-Maruyama scheme is known as

\begin{itemize}
\item If $1/4<H<1/2$ and $\sigma\in C_b^{q+3},b\in C_b^2$, then

\begin{equation}
\lim_{m\to\infty}2^{(2H-1)m}(\error{\cdot})=\frac{1}{2}\Jacobi{\cdot}\int_0^\cdot \Jacobi{t}^{-1}\sigma'\sigma(\exsol{t})dt\ a.s,
\end{equation}

where 

\begin{equation}
\Jacobi{t}\coloneqq \exp\left(\int_0^t \sigma'(\exsol{s})dB_s+\int_0^t b'(\exsol{s})ds\right).\label{def_J}
\end{equation}

This is revealed by Neuenkirch and Nourdin \cite{Neuenkirch-Nourdin2007}.
\end{itemize}
These three schemes are more important in theory than in practice.
In ODEs, the Euler scheme and trapezoidal rule are low-order numerical schemes, while the Taylor scheme is unstable.
Therefore, several fast and stable numerical schemes, mainly Runge-Kutta-type schemes, have been developed and widely used.
The same goes for SDE.
For practical purposes, we should use some stable Runge-Kutta scheme.
However, there are reasons to use the above three schemes over those more practical schemes when studying asymptotic errors.

\begin{itemize}
\item Firstly, the Euler-Maruyama scheme is the simplest numerical scheme for SDE, so it is natural to consider it first.
\item Secondly, if we can derive the asymptotic error of the $(k)$-Milstein scheme, we can compute numerical solutions of the same order under the same conditions.
Also, most numerical solution schemes have the same asymptotic error except for the $(k)$-Milstein scheme of the same order and the function under integration.
If we can derive the entire asymptotic error of the $(k)$-Milstein scheme, we can also derive the asymptotic error of numerical schemes of the same order under the same conditions.
\item Finally, since the Crank-Nicolson scheme is a second-order implicit Runge-Kutta scheme, we can identify the asymptotic error as the $(2)$-Milstein scheme.
However, the asymptotic error behavior of the Crank-Nicolson scheme differs from that of the $(2)$-Milstein scheme.
Furthermore, we know that the $(2)$-Milstein scheme does not converge at $H\leq 1/4$, while the Crank-Nicolson scheme is only known to not converge at $H\leq 1/6$.
Therefore, if we identify the asymptotic error of the $(2)$-Milstein scheme and evaluate the Crank-Nicolson scheme similarly, the asymptotic error is still an open problem when $1/6<H\leq 1/4$.
Indeed, while we fully identified the asymptotic error of the $(k)$-Milstein scheme in this study, we failed to identify the asymptotic error of the Crank-Nicolson scheme in $1/6<H\leq 1/4$.
\end{itemize}

For individual numerical schemes, it is possible to check whether they behave differently from Milstein-type schemes and to construct numerical solution schemes that behave differently from Milstein-type schemes.
It is unknown if naturally constructed Runge-Kutta or other types of schemes exhibit this behavior.
In this paper, we present the error estimates required to derive asymptotic errors for general numerical schemes and derive asymptotic errors for the (k)-Milstein scheme and the Crank-Nicolson scheme as special cases.

%先行研究でやったこと

Several studies have been conducted on the asymptotic error of numerical schemes, so we present our study of asymptotic errors restricted to one dimension, with and without a drift term.

\begin{itemize}
\item 
For SDE without drift, Gradinaru and Nourdin showed the asymptotic error for the ($k$)-Milstein scheme at $1/(k+1) < H < 1$ \cite{Gradinaru-Nourdin2009}.
The asymptotic error of the Crank-Nicolson scheme in $H\geq 1/2$ can also be shown as in this paper.
For the Crank-Nicolson scheme in $H\geq 1/2$, we need more advanced estimation.
Nourdin showed the asymptotic error distribution of the Crank-Nicolson scheme for a special diffusion coefficient for the Hurst index $1/6<H<1$ \cite{Nourdin2008}. 
For a general diffusion, Naganuma derived the asymptotic error distribution of the Crank-Nicolson scheme in $1/3 < H < 1/2$ \cite{Naganuma2015}. 
\item For SDE with drift, when $H\geq 1/2$, we can consider the drift term and the path of fBm in the same way.
Therefore, there is less theoretical need to restrict the problem to one dimension.
Few papers have dealt with this case, but the asymptotic error is easy to compute.

On the other hand, when $H<1/2$, Aida and Naganuma investigated the asymptotic error distribution of both the Crank-Nicolson and the ($2$)-Milstein schemes for $1/3 < H \leq 1/2$ \cite{Aida-Naganuma2020}. 
\end{itemize}

Furthermore, we present results on the convergence and divergence of numerical schemes.
Asymptotic error is defined only if the numerical scheme converges.
Nourdin proved that the ($k$)-Milstein scheme converges for $H> 1/(k+1)$ and it diverges when $H\leq 1/(k+1)$ and $H\leq 1/(k+2)$ for odd and even $k$, respectively, Similarly, he proved that the Crank-Nicolson scheme converges for $H>1/3$ and diverges for $H \leq 1/6$ \cite{Nourdin2005}.

%先行研究でやっていないこと

We summarize the results of these studies in the Table \ref{Milstein_no_drift}, \ref{CN_no_drift}, \ref{Milstein_previous_result} and \ref{CN_previous_result}.

\begin{table}[htb]
\begin{tabular}{|c||c|c|}\hline
Hurst index&convergence&asymptotic error\\\hline\hline
$1/(k+1)<H<1$ &Yes &determined \\\hline
$1/(k+2)<H\leq 1/(k+1)$ &No/\emph{unknown} &None/\emph{undefined} \\\hline
$0<H\leq 1/(k+2)$ &No &None \\\hline
\end{tabular}
\caption{the (k)-Milstein scheme($k$:odd/even), without drift}
\label{Milstein_no_drift}
\end{table}
\begin{table}[htb]
\begin{tabular}{|c||c|c|}\hline
Hurst index&convergence&asymptotic error\\\hline\hline
$1/3<H<1$ &Yes &determined \\\hline
$1/6<H\leq 1/3$ &Only in special case, Yes &Only in special case, determined \\\hline
$0<H\leq 1/6$ &No &None \\\hline
\end{tabular}
\caption{the Crank-Nicolson scheme, without drift}
\label{CN_no_drift}
\end{table}

\begin{table}[htb]
\begin{tabular}{|c||c|c|}\hline
Hurst index&convergence&asymptotic error\\\hline\hline
$1/3<H<1$ &Yes &determined \\\hline
$1/(k+1)<H\leq 1/3$ &Yes &\emph{unsolved} \\\hline
$1/(k+2)<H\leq 1/(k+1)$ &No/\emph{unknown} &None/\emph{undefined} \\\hline
$0<H\leq 1/(k+2)$ &No &None \\\hline
\end{tabular}
\caption{the (k)-Milstein scheme($k$:odd/even), with drift}
\label{Milstein_previous_result}
\end{table}
\begin{table}[htb]
\begin{tabular}{|c||c|c|}\hline
Hurst index&convergence&asymptotic error\\\hline\hline
$1/3<H<1$ &Yes &determined \\\hline
$1/6<H\leq 1/3$ &\emph{unknown} &\emph{undefined} \\\hline
$0<H\leq 1/6$ &No &None \\\hline
\end{tabular}
\caption{the Crank-Nicolson scheme, with drift}
\label{CN_previous_result}
\end{table}

%この論文でやったこと

We have derived the asymptotic error in the following part of this paper. We have also shown in Theorem \ref{theorem:Theorem_of_distribution_M} the conditions for a numerical scheme that allows us to derive the asymptotic error using the same way.

We will now describe the results of our study. We first define conditions ($\mathfrak{A}$) and ($\mathfrak{B}$), which correspond to the conditions shown in \cite{Nourdin2005} that convergence is known and a part of conditions shown in the same paper that convergence is unknown, respectively. Next, when numerical schemes and index $H$ satisfy these conditions, we estimate the remainder terms to show these are negligible. Using this result, we determine the asymptotic error of the ($k$)-Milstein scheme in $1/(k+1)<H\leq 1/3$ for odd $k$ and in $1/(k+2)<H\leq 1/3$ for even $k$, and the Crank-Nicolson scheme in $1/4<H\leq 1/3$. In other words, we identify the convergence and asymptotic error of ($k$)-Milstein scheme completely. All previous studies, except for \cite{Nourdin2008}, are compatible with Condition ($\mathfrak{A}$), which implies that finding the asymptotic error for numerical schemes that satisfy both ($\mathfrak{A}$) and ($\mathfrak{B}$) can be attributed to the computation of the main term. Additionally, we perform several experiments to verify the asymptotic error distribution.

At first, we will show the asymptotic error of the ($k$)-Milstein scheme.

\begin{thm}(Theorem \ref{theorem:Theorem_of_distribution_M})
Let $\sigma\in C_b^{k+1}, b\in C_b^2,0<H<1,k\in \mathbb{Z}_{\geq 2}, l\in \mathbb{Z}_{\geq 1}$.

\begin{itemize}
\item Then, if $b$ does not vanish, the asymptotic error of the ($k$)-Milstein scheme is determined for each $k$ as (\ref{Milstein_asymptotic_novanish_first}) $\sim$ (\ref{Milstein_asymptotic_novanish_last}).
\item On the other hand, if $b$ vanishes, the asymptotic error of the ($k$)-Milstein scheme is determined as (\ref{Milstein_asymptotic_vanish_first}) $\sim$ (\ref{Milstein_asymptotic_vanish_last}).
\end{itemize}
\end{thm}

Using the same estimation, we will show the error distribution of the Crank-Nicolson scheme in the range of $1/4<H<1$.
\begin{thm}[Theorem \ref{theorem:Theorem_of_distribution_CN}]
Given $\sigma\in C_b^{q+3},b\in C_b^3$.

Then, when $1/4<H<1/2$, whether or not $b$ vanishes, the asymptotic error distribution of the Crank-Nicolson scheme is determined as (\ref{asymptotic_error_distribtion_CN}).
\end{thm}
%この論文でやってないこと

Table \ref{Milstein_our_result} and \ref{CN_our_result} shows the results already presented and the obtained newly. Bolded text is the main result.

\begin{table}[htb]
\begin{tabular}{|c||c|c|}\hline
Hurst index&convergence&asymptotic error\\\hline\hline
$1/3<H<1$ &Yes &determined \\\hline
$1/(k+1)<H\leq 1/3$ &Yes &\textbf{determined} \\\hline
$1/(k+2)<H\leq 1/(k+1)$ &No/\textbf{Yes} &None/\textbf{determined} \\\hline
$0<H\leq 1/(k+2)$ &No &None \\\hline
\end{tabular}
\caption{the (k)-Milstein scheme($k$:odd/even)}
\label{Milstein_our_result}
\end{table}

\begin{table}[htb]
\begin{tabular}{|c||c|c|}\hline
Hurst index&convergence&asymptotic error\\\hline\hline
$1/3<H<1$ &Yes &determined \\\hline
$1/4<H\leq 1/3$ &\textbf{Yes} &\textbf{determined} \\\hline
$1/6<H\leq 1/4$ &\emph{unsolved} &\emph{unsolved} \\\hline
$0<H\leq 1/6$ &No &None \\\hline
\end{tabular}
\caption{the Crank-Nicolson scheme}
\label{CN_our_result}
\end{table}

%証明のスケッチ

We now present a rough sketch of the proof.
First, we can expand the exact solution:
\begin{equation}
\exsol{t}=\nusoldis{r}+\sum_{k=1}^q \frac{1}{k!}\mathcal{D}^{k-1}\sigma(\exsoldis{r})\fBminc{\bitime{r},t}^k +b(\exsoldis{r})(t-\bitime{r})+\epsilon_{\bitime{r},t}^{(1)},\\
\end{equation}
where $q= \lfloor H^{-1}\rfloor$.
Second, we can rewrite a numerical solution:
\begin{equation}
\nusol{t}=\nusoldis{r}+\sum_{k=1}^q \frac{1}{k!}\mathcal{D}^{k-1}\sigma(\nusoldis{r})\fBminc{\bitime{r},t}^k +b(\nusoldis{r})(t-\bitime{r})+\hat{\epsilon}_{\bitime{r},t}^{(m,1)}.\\
\end{equation}
Third, as a solution interpolated between $\exsol{t}$ and $\nusol{t}$, we introduce $\polsolstd{t}$ defined as
\begin{equation}
\polsolstd{t}=\polsolstddis{r}+\sum_{k=1}^q \frac{1}{k!}\mathcal{D}^{k-1}\sigma(\polsolstddis{r})\fBminc{\bitime{r},t}^k +b(\polsolstddis{r})(t-\bitime{r})+\rho\hat{\epsilon}_{\bitime{r},t}^{(m,1)}+(1-\rho)\epsilon_{\bitime{r},t}^{(1)}\\
\end{equation}
Then the error can be expressed as follows.
\begin{equation}
\nusol{t}-\exsol{t}=\int_0^1 \polJacobistd{t}\sum_{r=0}^{ \lceil 2^m t\rceil-1} (M_{\bitimelast{r+1}}^{(m,\rho)})^{-1}(\hat{\epsilon}_{\bitime{r},\bitime{r+1}\wedge t}^{(m,1)}-\epsilon_{\bitime{r},\bitime{r+1}\wedge t}^{(1)})d\rho,.
\end{equation}
where
\begin{equation}
\polJacobistd{t}=\prod_{r=0}^{ \lceil 2^m t\rceil-1}\left(1+\sum_{k=1}^q \frac{1}{k!}(\mathcal{D}^{k-1}\sigma)'(\polsolstddis{r})\fBminc{\bitime{r},\bitime{r+1}\wedge t}^k+b'(\polsolstd{\bitime{r}})(\bitime{r+1}\wedge t-\bitime{r})\right).
\end{equation}
In fact, we can obtain the approximation that $\polJacobistd{t}\approx \Jacobi{t}$, so that we have
\begin{equation}
\nusol{t}-\exsol{t}\simeq\Jacobi{t}\sum_{r=0}^{ \lceil 2^m t\rceil-1} \Jacobi{\bitime{r+1}}^{-1}(\hat{\epsilon}_{\bitime{r},\bitime{r+1}\wedge t}^{(m,1)}-\epsilon_{\bitime{r},\bitime{r+1}\wedge t}^{(1)}).
\end{equation}
Moreover, we can approximate the summand such as 
\[\Jacobi{\bitime{r+1}}^{-1}(\hat{\epsilon}_{\bitime{r},\bitime{r+1}\wedge t}^{(m,1)}-\epsilon_{\bitime{r},\bitime{r+1}\wedge t}^{(1)})\approx \Jacobi{\bitime{r}}^{-1}f(\exsoldis{r})\delta X_{\bitime{r},\bitime{r+1}\wedge t}\]
for some stochastic process $X_t$.
Then, when $m$ goes to infinity, we obtain the asymptotic error such that
\begin{equation}
\lim_{m\to\infty}(\nusol{t}-\exsol{t})=\Jacobi{t}\int_0^t \Jacobi{t}^{-1}f(\exsol{s})dX_s.
\end{equation}
We will justify this rough sketch in Section 4,5,6.
We will give the necessary estimates pathwisely in a similar way to the justification of the rough integral.

A new idea in this study is a unified evaluation for Stieltjes sums of arbitrary order.
To justify the approximation, we need a good treatment of Stieltjes sums for asymptotic expansions of order $q$.
To obtain the asymptotic error for any $H$, we need to use an evaluation method that can be applied to arbitrarily large $q$.
There are multiple ways of expressing the error, but the same problem arises no matter which way is used.
In this study, we have dealt with Stieltjes sums by corresponding them to identities in rough paths.
%この論文の章立て

The paper is structured as follows.
In Section 2, we introduce the notations used throughout the paper.
In Section 3, we present the main theorem, which is the estimate of the remainder term (Theorem \ref{theorem:maintheorem}) for general $H$. Moreover, we determine the asymptotic error of the ($k$)-Milstein scheme and of the Crank-Nicolson scheme.
In Section 4, we decompose the error term into a main term and several remainder terms in the way similar to \cite{Aida-Naganuma2023}.
In Section 5, we demonstrate that the entire remainder term can be estimated by assuming some basic estimate formulas, as presented in Theorem \ref{theorem:maintheorem}.
At last, in Section 6, we check that the numerical schemes satisfy the assumption of Theorem \ref{theorem:maintheorem}.
Also, we calculate the coefficients of the asymptotic error distribution of the ($k$)-Milstein scheme and the Crank-Nicolson scheme.

\begin{rmk}

%多次元の場合

Additionally, we present studies of asymptotic errors for multi-dimensional fSDEs.
We can define numerical solutions for multi-dimensional fSDEs as one-dimensional, and we can consider asymptotic errors for those numerical solutions.
In addition, several studies have been devoted to studying asymptotic errors for multi-dimensional fSDEs.
Moreover, many time series models have been devised using SDEs of arbitrary dimensions.
However, we describe these studies separately from previous studies for the main results.
This is because the formulation of a multi-dimensional problem requires a different theory than the one-dimensional one, and the results are very different from the one-dimensional case.

The most classical result is for $H=1/2$, i.e., Brownian motion.
Jacod and Protter derived the asymptotic error of the Euler scheme for semimartingale SDEs including Ito type SDEs \cite{Jacod-Protter1998}.
Latter, Yan also derived the asymptotic error of the Milstein scheme for semimartingale SDEs \cite{Yan2005}.

For $H>1/2$, asymptotic errors have been studied by Hu, Liu, and Nualart in detail.
They determined the asymptotic error of the Euler scheme and modified Euler scheme for $1/2 < H < 1$ \cite{Hu-Liu-Nualart2016}.
Later, they also determined the asymptotic error distribution of the Crank-Nicolson scheme for $1/2 < H < 1$ \cite{Hu-Liu-Nualart2021}.

The modified Euler scheme is a scheme in which the second-order power $\fBminc{s,t}\otimes \fBminc{s,t}$ of the (2)-Milstein method is replaced by its expectation value $I_d(t-s)^{2H}$,
where $\fBminc{s,t}=B_t-B_s$, $B_t=\{B_t^{(i)}\}_{i=1}^d$ is a d-dimensional fBm, $\otimes$ is a direct product and $I_d$ is a $d$-order fundamental diagonal matrix.
This modified Euler scheme appears naturally as a transformation of the Euler-Maruyama scheme when converting the Ito SDE to the Stratonovich SDE.
When $H>1/2$, fSDE is defined using the Young integral, and the distinction between Ito SDE and Stratonovich SDE does not exist.
Both the Euler-Maruyama scheme and the modified Euler scheme are valid in this case.
The modified Euler scheme is known to have a higher convergence rate than the Euler method.

For $H<1/2$, despite efforts by many researchers, asymptotic errors are only known when the Hurst exponent is greater than $1/3$.
Liu and Tindel showed the asymptotic error of the modified Euler scheme for $1/3 < H < 1/2$ \cite{Liu-Tindel2019}.
Aida and Naganuma investigated the asymptotic error of the Crank-Nicolson scheme for $1/3 < H \leq 1/2$ \cite{Aida-Naganuma2023}, on which the primary technique of this paper is based.

On the other hand, more result is known about whether numerical schemes converge.
Any numerical scheme diverges for $H \leq 1/4$ \cite{Coutin-Qian2002}.
Also, the second-order Taylor scheme converges for $H>1/3$ \cite{Deya-Neuenkirch-Tindel2012} and the third-order Taylor scheme((3)-Milstein scheme for the multi-dimensional case) converges for $H>1/4$ \cite{Friz-Riedel2014}.
Therefore, the main theoretical interest is the asymptotic error in $H\in(1/3,1/4]$.
\end{rmk}
\begin{table}[htb]
\begin{tabular}{|c||c|c|}\hline
Hurst index&convergence&asymptotic error\\\hline\hline
$1/3<H<1$ &Yes &determined \\\hline
$1/4<H\leq 1/3$ &\emph{unknown} &\emph{undefined} \\\hline
$0<H\leq 1/4$ &No &None \\\hline
\end{tabular}
\caption{the modified Euler and Crank-Nicolson scheme in multi-dimension}
\end{table}

\section{Definition and Notation}\label{section:notation}

In this section, we will introduce the basic definitions and notation used in this paper.
Especially, we devote two subsections for the stochastic differential equation driven by fractional Brownian motion(fBm) and discrete H\"older norm.

\begin{itemize}
\item Throughout this paper, we discuss asymptotic error in $[0,T]$ with arbitrary finite time $T$ for simplification.
Without loss of generality, we assume that $T$ is an integer.

\item We have introduced discrete time $\bitime{r}$ as $\bitime{r}\coloneqq 2^{-m}r$ in introduction.
Additionally, we denote the set of discrete time in $[0,T]$ by $\mathbb{P}_m$.
\item We call a function defined on $\Omega\subset [0,T]$ a pass. For any pass $x;t\mapsto x_t$, we set $\xunit$ as $x_t-x_s$ for $0<s<t<T$.
Also, for any two-varuable funtion $\Xi$ defined on $[0,T]$, we define $\delta \Xi_{s,u,t}$ as $\Xi_{s,t}-\Xi_{s,u}-\Xi_{u,t}$ for $0<s<u<t<T$.
\item Let $x$ is a path defined in $\Omega\subset [0,T]$. For $0<\alpha<1$, we denote the $\alpha$-H\"older norm by $\|\cdot\|_\alpha; x\mapsto \|x_\cdot\|_\alpha$.
On the other hand, to avoid confusion, for $p\geq 1$, we write $L^p$-norm as $\|\cdot\|_{L^p}$.

We note that when $\Omega=\mathbb{P}_m$ for some $m\in \mathbb{Z}_{\geq 0}$, $\|x_\cdot\|_\alpha$ always takes a finite value.
\item As mentioned in introduction, we have defined fBm as the centered Gaussian process with the parameter $H$ characterized as (\ref{Covariance_of_fBm}).

\item Fix a Hurst index $H$ and let the integer $q$ as $(q+1)^{-1}<H\leq q^{-1}$. Then, we take an appropriate constant $H_-,H',H^c$ satisfying that
\begin{align*}
H_-\in (1/(q+1),H),&&H'\leq H_-&&H^c+H'>1,&&qH_->H^c.
\end{align*}
Especially $H_-$ shall be sufficiently close to $H$, if necessary.

A sample path of fBm is the most important example of H\"older continuous path. It is known that fBm can be regarded as a random variable with values in $C[0,T]$. Further, fBm as $C[0,T]$-valued random variable is in 
\begin{equation}
\Omega_0=\{\omega:[0,T]\to \mathbb{R}:\|\omega\|_{H_-}<\infty,\|\omega\|_{H}=\infty\}\label{Omega_0}
\end{equation}
almost surely.

So, we will treat fBm as an $\Omega_0$-valued random variable.
\item To represent iterated integrals driven by $B_t$ and $t$, we define a triple $(\mathbb{G},B^\cdot,|\cdot|)$ as follows:

\begin{itemize}
\item An index set $\mathbb{G}$ equals $\bigcup_{l=1}^\infty\{0,1\}^l$. Also, we denote $(\underbrace{1,1,\cdots,1,1}_{l})\in \mathbb{G}$ by $(l)$ briefly.
\item $B^\cdot:\Gamma\in \mathbb{G}\mapsto B^\Gamma \in C([0,T]^2\to L^2(\Omega)) $ is an iterated integral map defined by the following recurrence relation.

\begin{align*}
\fBmitin_{s,t}^{(0)}=t-s,&&\fBmitin_{s,t}^{(1)}=B_t-B_s,&&\fBmitin_{s,t}^{\Gamma\oplus 0}=\int_s^t \fBmitin_{s,u}^{\Gamma} du,&&\fBmitin_{s,t}^{\Gamma\oplus 1}=\int_s^t \fBmitin_{s,u}^{\Gamma} dB_u,
\end{align*}
where $\Gamma\oplus i\coloneqq (\gamma_1,\cdots,\gamma_n,i)$ for $i =0$ or $1$ and $\fBmitin_{s,t}^{\Gamma\oplus i}\ (i=0,1)$ are defined by Proposition \ref{proposition:Operation for a controlled path} as controlled paths.

(In this case, we obtain $\fBmitin_{s,t}^\Gamma$ uniquely from this definition.
However, defining the stochastic integral Proposition as \ref{proposition:Operation for a controlled path} is equals to defining $\fBmitin_{s,t}^{(l)}=\frac{1}{l!}\fBminc{s,t}^l$ for $1\leq l\leq q$.
In the sense of rough path theory, we set the rough path of $x$ to be $\mathbf{x}_{s,t}=(\xunit^1,\cdots,(q!)^{-1}\xunit^q)$.)

Now, we remark that $\fBmitin_{s,t}^{(l)}=\frac{1}{l!}\fBminc{s,t}^l$ for any $l$.
\item $|\cdot|$ is a size map defined as $|\Gamma|=\#_0 \Gamma+\#_1 \Gamma H$, where we denote the number of zeros(resp. ones) in $\Gamma$ by $\#_0 \Gamma$(resp. $\#_1 \Gamma$).
\end{itemize}
Additionally, given $r,R\in \mathbb{R}_{\geq 0}$,\ we introduce $\mathbb{G}_{r,R}\subset \mathbb{G}$ as
\begin{align}
\mathbb{G}_{r,R}\coloneqq& \{\Gamma\in \mathbb{G}|r<|\Gamma|\leq R\}.\label{subscriptA}
\end{align}
\item Similarly, to represent power series of $\delta \fBminc{s,t}$ and $(t-s)$, we define another triple $(\hat{\mathbb{G}},B^\cdot,|\cdot|)$ as follows:

\begin{itemize}
\item An index set $\hat{\mathbb{G}}$ is $\{(\hat{\Gamma}_1,\hat{\Gamma}_2)|\hat{\Gamma}_1,\hat{\Gamma}_2\in \mathbb{Z}_{\geq 0}\}$.
\item The power of increments $\fBmitin_{\cdot,\cdot}^{\hat{\Gamma}}$ as $\fBmitin_{s,t}^{\hat{\Gamma}}=\fBminc{s,t}^{\hat{\Gamma}_1}(t-s)^{\hat{\Gamma}_2}$.
\item The size map $|\cdot|;\hat{\Gamma}\in \hat{\mathbb{G}}\mapsto |\hat{\Gamma}| \in \mathbb{R}_{\geq 0}$ is defined as $|\hat{\Gamma}|=\hat{\Gamma}_1H+\hat{\Gamma}_2$.
\end{itemize}
Additionally, given $r,R\in \mathbb{R}_{\geq 0}$,\ we introduce $\hat{\mathbb{G}}_{r,R}\subset \hat{\mathbb{G}}$as
\begin{align}
\hat{\mathbb{G}}_{r,R}\coloneqq& \{\hat{\Gamma}\in \hat{\mathbb{G}}|r<|\hat{\Gamma}|\leq R\}.\label{subscripthatA}
\end{align}

\item We denote $\fBmitin^{(1,1,0)},\fBmitin^{(1,0,1)},\fBmitin^{(0,1,1)}$ by $\fBmitin^{110},\fBmitin^{101},\fBmitin^{011}$, respectively. In addition, we define modified iterated integrals as
\begin{align*}
\fBmitin_{s,t}^{10*}&=\fBmitin_{s,t}^{10}-\frac{1}{2}\fBminc{s,t}(t-s),&\fBmitin_{s,t}^{110*},&=\fBmitin_{s,t}^{110}-\frac{1}{2(1+2H)}(t-s)^{1+2H},\\
\fBmitin_{s,t}^{101*}&=\fBmitin_{s,t}^{101}+\frac{1-2H}{2(1+2H)}(t-s)^{1+2H},&\fBmitin_{s,t}^{011*}&=\fBmitin_{s,t}^{011}-\frac{1}{2(1+2H)}(t-s)^{1+2H}.
\end{align*}
\item Let $f$ and $g$ be in $C^1(\mathbb{R}^d;\mathbb{R}^d)$ for $d\in \mathbb{N}$. We define the differential operator $\mathcal{D}_g$ as
\begin{equation*}
\mathcal{D}_gf(u)=\sum_{i=1}^d g_i(u)\partial_{u_i}f(u).
\end{equation*}

Also, for $\Gamma\in \mathbb{G}$, we define $\mathcal{D}^{\Gamma}$ inductively such that $\mathcal{D}^{\Gamma\oplus 0}=\mathcal{D}_b(\mathcal{D}^{\Gamma} f), \mathcal{D}^{\Gamma\oplus 1}=\mathcal{D}_\sigma(\mathcal{D}^{\Gamma} f)$.
Moreover, we abbreviate $\mathcal{D}_\sigma$ as $\mathcal{D}$.
In addition, we define $\mathcal{V}f\equiv \mathcal{D}_\sigma f-\mathcal{D}_f\sigma$.
When $d=1$, we can write simply as $\mathcal{D}_gf=f'g$.
\item We define $\Xi[f;x,s,t],\Xi'[f;x,s,t]$ for a function $f\in C^q$, a path $x$, an integer $q$ defined as $(q+1)^{-1}<H\leq q^{-1}$ in above and times $s,t$ as follows:

\begin{align}
\Xi[f;x,s,t]\coloneqq \sum_{k=1}^q \frac{1}{k!}\mathcal{D}^{k-1}f(x_s)\xunit^k,&&\Xi'[f;x,s,t]\coloneqq \sum_{k=1}^q \frac{1}{k!}(\mathcal{D}^{k-1}f)'(x_s)\xunit^k.\label{taylor_of_f}
\end{align}
\item For a pass $x_\cdot^{(1)}$ and a two-valuable function $x_{\cdot,\cdot}^{(2)}$, we introduce the forward Stieltjes sum $\mathcal{I}_t^{(m)}(x_\cdot^{(1)},d(x^{(2)}))$ and backward Stieltjes sum $\mathcal{I}_t^{(m)}(x_{\cdot+}^{(1)},d(x^{(2)}))$ as

\begin{equation*}
\mathcal{I}_t^{(m)}(x_\cdot^{(1)},d(x^{(2)}))=\sum_{r=0}^{ \lfloor 2^m t\rfloor-1}x_{\bitime{r}}^{(1)}x_{\bitime{r},\bitime{r+1}}^{(2)}.
\end{equation*}

\begin{equation*}
\mathcal{I}_t^{(m)}(x_{\cdot+}^{(1)},d(x^{(2)}))=\sum_{r=0}^{ \lfloor 2^m t\rfloor-1}x_{\bitime{r+1}}^{(1)}x_{\bitime{r},\bitime{r+1}}^{(2)}.
\end{equation*}

Besides this, we denote $\mathcal{I}_t^{(m)}(1,d(x^{(2)}))$ briefly by $\mathcal{I}_t^{(m)}(d(x^{(2)}))$. Additionally, we also write $\mathcal{I}_t^{(m)}(x^{(1)},d(x_{\cdot_1,\cdot_2}^{(2)}))$ in the sense of $\mathcal{I}_t^{(m)}(x^{(1)},d(x^{(2)}))$. For example, $\mathcal{I}^{(m)}(x,d(x_{\cdot_1,\cdot_2}^2-(\cdot_2-\cdot_1)^{2H}))$ means $\mathcal{I}^{(m)}(x^{(1)},d(x^{(2)})),$ where $x_t^{(1)}=x_t, \xunit^{(2)}=\xunit^2-(t-s)^{2H}$.
Throughout this paper, we consider that each integral is defined on $\mathbb{P}_m$.

\item Let $\{X^\lambda\}_{\lambda\in\Lambda}$ be an indexed family in the normed vector space $V$, which is also indexed by the multi-index $\mathbf{t}$. When we estimate that

\begin{equation*}
\sup_{\lambda\in\Lambda}\|X_\cdot^\lambda\|_{V}\leq C f(\mathbf{t}).
\end{equation*}
with the constant $C$ dependent on $\sigma, b, B, H_-, H', H^c$, and independent of $\lambda$ and $\mathbf{t}$, we abbreviate this inequality as $\|X_\cdot^{\lambda}\|_V\lesssim_\lambda f(\mathbf{t})$.
\item We use the notation $O(f(\mathbf{x}))$ as Landau's symbol.
\end{itemize}

\begin{rmk}
The above definitions allow us to handle not only the $(k)$-Milstein and Crank-Nicolson methods discussed in this theorem, but also many numerical solution schemes, including arbitrary Runge-Kutta methods.
However, to handle some numerical solution schemes, the above theorem needs to be extended slightly.

For example, as we will prove in this paper, the asymptotic error of the $(k)$-Milstein method has a convergence rate of $2^{-2Hm}$ or $2^{-(jH-1)m}$ for appropriate $j$.
If a correction term that cancels out this error is added to increase the accuracy of the numerical solution, that correction term takes the form of $(t - \bitime{r})^{jH}$ or $(t - \bitime{r})^{1+2H}$, consisting of a power term and an appropriate exponential term.

However, numerical solutions involving such power terms are not included in the above definition and must be redefined appropriately. Nevertheless, we expect the proofs in the remainder of this paper to be valid for the redefined triple.
\end{rmk}
\subsection{Pathwise integral driven by H\"older path}

In this subsection, we will introduce the 1-dim pathwise integral for fBm.
As explained in introduction, this integral is described a simplified version of Gubinelli's rough integral for controlled rough path \cite{Gubinelli2004}, but it equals the symmetric integral \cite{Russo-Vallois1993}.
Note that we treat only the solution $Y$ as a 1-dim path but for the discussion of Section 5, we define $Y$ as a vector-valued path.

Firstly, we introduce pathwise integral driven by $\alpha$-H\"older path.

\begin{prp}[Integral driven by $x$]\label{prp:continuous_rough_integral}

Let $\alpha$ be in $(0,1)$, $x$ be $\alpha$-H\"older path and $q$ be an integer such that $(q+1)^{-1}<\alpha\leq q^{-1}$.
Provided $z=(z^{(1)},\cdots,z^{(q)})\in C^\alpha([0,T];\mathbb{R}^{d\times q})$ satisfies the following estimation:
\begin{equation}\label{controlledpath}
z_t^{(k)}-z_s^{(k)}-\sum_{j=1}^{q-k}\frac{1}{j!}z_s^{(k+j)}\xunit^j=O(|t-s|^{(q-k+1)\alpha}),
\end{equation}
for all $1\leq k\leq q$ and $s,t\in [0,T]$.

Then, there exists $\mathcal{I}(z)\in C^\alpha([0,T];\mathbb{R}^d)$ uniquely such that
\begin{equation*}
\mathcal{I}(z)_t-\mathcal{I}(z)_s-\sum_{k=1}^q \frac{1}{k!} z_s^{(k)}\xunit^{k}=O(|t-s|^{(q+1)\alpha}).
\end{equation*}
for arbitrary $s,t\in [0,T]$.
\end{prp}
We call this path $z=(z^{(1)},\cdots,z^{(q)})$ a controlled path (of $x$) if it satisfies above condition.
Also, we call $\mathcal{I}(z)$ the integral of $z$ driven by $x$ and denote it by $\int_0^\cdot z_tdx_t$.
By abuse of notation, when we treat only the first component $z^{(1)}$ of $z$, we denote $z^{(1)}$ by $z$.

Secondly, we describe a few example of controlled path.

\begin{prp}\ \label{proposition:elementary path as a controlled path}

\begin{enumerate}
\item For any continuous path $h$, $(\int_0^\cdot h_tdt,0,\cdots,0)$ is a controlled path.
\item The driving path $(x,1,0,\cdots,0)$ is a controlled path.
\item An integral $\left(\int_0^\cdot z_t dx_t, z^{(1)},\cdots, z^{(q-1)}\right)$ is a controlled path.
\end{enumerate}
\end{prp}
Based on this proposition, we consider a Riemann integral $\int_0^\cdot h_tdt$ driving path $x$, pathwise integral $\int_0^\cdot z_tdx_t$ as controlled paths.

Thirdly, we define operations for controlled paths.

\begin{prp}[Operation for a controlled path]\label{proposition:Operation for a controlled path}\
We set each of $z=(z^{(1)},\cdots,z^{(q)})$ and $\overline{z}=(\overline{z}^{(1)}\cdots,\overline{z}^{(q)})$ to be $d$-dimensional controlled paths of $x$, respectively.

Then, we can obtain the new controlled paths as follows:

\begin{enumerate}
\item $z+\overline{z}$ is controlled path.
\item When we write $z^{(j)}=(z^{(j,1)},\cdots,z^{(j,d)})^t$, $z_r=(z^{(1,r)},\cdots,z^{(q,r)})$ is $1$-dimensional controlled path.
\item Let $\circ$ is element-wise product. Then, we define $z\circ \overline{z}$ as

\begin{align}
\left(z^{(1)}\circ\overline{z}^{(1)},z^{(1)}\circ\overline{z}^{(2)}+z^{(2)}\circ\overline{z}^{(1)},\cdots,\sum_{k=1}^q \binom{q-1}{k-1} z^{(k)}\circ\overline{z}^{(q-k)}\right)\label{pathproduct}.
\end{align}
is a controlled path.
\item For $f\in C_b^{q+1}(\mathbb{R}^{d};\mathbb{R}^{d_2})$, we define $f(z)=(f(z)^{(1)},\cdots,f(z)^{(q)})$ as

\begin{align}
f(z)_s^{(1)}=f(z_s^{(1)}),f(z)_s^{(k)}=k!\sum_{|\Gamma|\leq k} \frac{\partial^{\Gamma} f}{\Gamma!}(z_s^{(1)})\sum_{\begin{smallmatrix}
r_1+\cdots+r_d=k,\\
r_{j,1}+\cdots+r_{j,\gamma_j}=r_j.
\end{smallmatrix}}\prod_{j=1}^d\frac{z^{(j,r_1)}\cdots z^{(j,r_j)}}{(r_{j,1})!\cdots (r_{j,r_j})!}\label{pathcomposition},
\end{align}
where multi-index $\Gamma$ is defined as $(\gamma_1,\cdots,\gamma_d)$.

Then, $f(z)$ is a controlled path.
\end{enumerate}
\end{prp}

Additionaly, we define an inner product $z\cdot \overline{z}=\sum_{r=1}^d z_r\circ \overline{z}_r=\sum_{r=1}^d (z\circ \overline{z})_r$.

Combining these definitions, we can define $B^{\Gamma}$ almost surely.

Fourthly, we introduce the SDE driven by fBm.

\begin{prp}[Stochastic differential equation]
\label{prp:SDE}

Given $q\in \mathbb{N}_{\geq 1}, \alpha\in((q+1)^{-1},q^{-1}], \sigma\in C_b^{q+1}(\mathbb{R}^d; \mathbb{R}^d),b\in C_b^1(\mathbb{R}^d;\mathbb{R}^d), x\in C^{\alpha}([0,T];\mathbb{R})$, a controlled path $z$ uniquely exists such that

\begin{equation}
z_t=z_0+\int_0^t \sigma(z_s)dx_s+\int_0^t b(z_s)ds.
\end{equation}
for any $t\in [0,T]$.

We write this equation $dz_t = \sigma(z_t)dx_t+ b(z_t)dt$ briefly.

For the solution $z$, we can see that $z_t^{(i)}= \mathcal{D}^{i-2}\sigma(z_t^{(1)})$ for $2\leq i\leq q$.
Additionally, for $f\in C_b^{q+1}(\mathbb{R}^{d};\mathbb{R}^{d_2})$ with any $d_2$, we have $f(z_t)^{(i)}= \mathcal{D}^{i-2}\sigma(z_t^{(1)})$.
\end{prp}

\begin{prp}[It\^{o}'s formula]\label{proposition:Itoformula}

We set a $d$-dimension controlled path $z$ to be 

\[z_t = \int_0^t \overline{z}_u dx_u+\int_0^t \hat{z}_udu\]
 with some $d$-dimension controlled paths $\overline{z}$ and $\hat{z}$.

Then, for $f=(f_1,\cdots,f_{d_2})\in C^{q+2}(\mathbb{R}^d,\mathbb{R}^{d_2})$, we have

\begin{equation*}
f(z_t)=\int_0^t \nabla f(z_u)\cdot \overline{z}_u dx_u+\int_0^t \nabla f(z_u)\cdot \hat{z}_udu,
\end{equation*}
where $\nabla f(z)\cdot \overline{z}=(\nabla f_1(z)\cdot \overline{z},\cdots,\nabla f_{d_2}(z)\cdot \overline{z})$

\end{prp}

Since $B$ is $H_-$-H\"older continuous almost surely, we can formulate the equation $d\exsol{t} = \sigma(\exsol{t})dB_t+ b(\exsol{t})dt$ in the sense of pathwise integral.

\subsection{Discrete H\"older norm}\label{subsection:discrete_Hoelder_norm}

In this subsection, we discuss H\"older norm for path defined in $\mathcal{P}_m$.
This subsection aims to see a H\"older estimate of $\mathcal{I}(d(B^{\Gamma}))$.
At first, we show the Garsia-Rodemich-Ramsey (GRR) inequality, a key statement to estimate the H\"older norm of a continuous stochastic process.

\begin{prp}[Garsia-Rodemich-Ramsey inequality]\label{proposition:GRR}
Let $x:C^0([0,T]:\mathbb{R}),p>1,0<\theta<1$.

Then,

\begin{equation*}
\|x_\cdot\|_\theta^p\leq C_{p,\theta}\int_0^Tdt\int_0^tds \frac{|\xunit|^p}{|t-s|^{2+p\theta}},
\end{equation*}
where $C_{p,\theta}$ is a constant independent of $x$.

\end{prp}

\begin{defi}

We define $\|X_\cdot\|_{L^p,\theta}$ as 

\[\|X_\cdot\|_{L^p,\theta}=\sup_{0<s<t<T}\frac{\|\delta X_{s,t}\|_{L^p}}{|t-s|^\theta}.\]

\end{defi}

\begin{cor}\label{corollary:stochasticGRR}

Let $X$ to be continuous stochastic process such that $\|X_\cdot\|_{L^p,\theta'}\lesssim 1$.

Then, for $\theta<\theta'-p^{-1}$, $X$ is $\theta$-H\"older continuous almost surely. Moreover, $\|\ \|X_\cdot\|_\theta\|_{L^p}\lesssim \|X_\cdot\|_{L^p,\theta'}$.

\end{cor}

\begin{proof}

Let $\epsilon=p\theta'-1-p\theta$.
Then, $\epsilon>0$.
From GRR inequality, we have

\[\|X_\cdot\|_\theta^p\leq C_{p,\theta}\int_0^Tdt\int_0^tds\frac{|\delta X_{s,t}|^p}{|t-s|^{2+p\theta}}\]

Taking the expected value of both sides, we see

\begin{equation}
E[\|X_\cdot\|_\theta^p]\leq C_{p,\theta}\int_0^Tdt\int_0^tds \frac{E[|\delta X_{s,t}|^p]}{|t-s|^{2+p\theta}}.\label{GRR_expected}
\end{equation}
By assumption, $E[|\delta X_{s,t}|^p]=\|\delta X_{s,t}\|_{L^p}^p\leq \|X_\cdot\|_{L^p,\theta'}^p(t-s)^{p\theta'}$.
Since $p\theta'=1+p\theta+\epsilon$, we get 
\[\frac{E[|\delta X_{s,t}|^p]}{|t-s|^{2+p\theta}}\leq \frac{\|X_\cdot\|_{L^p,\theta'}^p}{|t-s|^{1-\epsilon}}.\]
From 
\[\int_0^Tdt\int_0^tds \frac{1}{|t-s|^{1-\epsilon}}=\frac{T^{1+\epsilon}}{ \epsilon(1+\epsilon)}<\infty,\]

we obtain immediately $E[\|X_\cdot\|_\theta^p]\simeq \|X_\cdot\|_{L^p,\theta'}^p$.
Raising both sides to the power of $1/p$, we have the second statement.
Also this gives $E[\|X_\cdot\|_\theta^p]<\infty$.
Hence, $\|X_\cdot\|_\theta$ is finite almost surely, which is the first statement.
\end{proof}

Next, we estimate H\"older norm of linear interpolation of discrete process.

\begin{lmm}\label{lemma:discretecontinuous}
Let $X^{(m)}$ be discrete process defined in $\mathbb{P}_m$ and $\overline{X}^{(m)}$ the linear interpolation of $X^{(m)}$.
Also, take $\theta\in (0,1)$ arbitrary.

Then, $\|X_\cdot^{(m)}\|_\theta\leq \|\overline{X}_\cdot^{(m)}\|_{\theta}\lesssim \|X_\cdot^{(m)}\|_\theta$.\newline
Moreover, $\|X_\cdot^{(m)}\|_{L^p,\theta}\leq \|\overline{X}_\cdot^{(m)}\|_{L^p,\theta}\lesssim \|X_\cdot^{(m)}\|_{L^p,\theta}$.

\end{lmm}

\begin{proof}

By definition of $\overline{X}$, $\|X_\cdot^{(m)}\|_{\theta}\leq \|\overline{X}_\cdot^{(m)}\|_{\theta}$ a.s. and $\|X_\cdot^{(m)}\|_{L^p,\theta'}\leq \|\overline{X}_\cdot^{(m)}\|_{L^p,\theta'}$.
So, we only have to show right-hand side inequalities.
For $s<t$, we denote $2^{-m}(\lceil 2^mt\rceil-1),\ 2^{-m}\lfloor 2^mt\rfloor,\ 2^{-m}\lceil 2^mt\rceil$ by $\eta^{(m)}(t),\ \tilde{\eta}^{(m)}(t),\ \eta_+^{(m)}(t)$, respectively.

If $\tilde{\eta}^{(m)}(t)=\tilde{\eta}^{(m)}(s)$, we can see that 

\[\delta \overline{X}_{s,t}^{(m)}=\frac{t-s}{\eta_+^{(m)}(t)-\tilde{\eta}^{(m)}(s)}\delta X_{\tilde{\eta}^{(m)}(s),\eta_+^{(m)}(t)}^{(m)}=2^m(t-s)\delta X_{\tilde{\eta}^{(m)}(s),\eta_+^{(m)}(t)}^{(m)}\]

Also, $|\delta X_{\tilde{\eta}^{(m)}(s),\eta_+^{(m)}(t)}^{(m)}|\leq 2^{-m\theta}\|X_\cdot^{(m)}\|_{\theta}$

Hence, $|\delta \overline{X}_{s,t}^{(m)}|\leq (2^m(t-s))^{1-\theta}\cdot (t-s)^{\theta}\|X_\cdot^{(m)}\|_{\theta}\leq (t-s)^{\theta}\|X_\cdot^{(m)}\|_{\theta}$.

On the other hand, if $\tilde{\eta}^{(m)}(t)\neq \tilde{\eta}^{(m)}(s)$,

\begin{align*}
\delta \overline{X}_{s,t}^{(m)}=&\delta \overline{X}_{s,\eta_+^{(m)}(s)}^{(m)}+\delta \overline{X}_{\eta_+^{(m)}(s),\tilde{\eta}^{(m)}(t)}^{(m)}+\delta \overline{X}_{\tilde{\eta}^{(m)}(t),t}^{(m)}\\
=&2^m(\eta_+^{(m)}(s)-s)\delta X_{\tilde{\eta}^{(m)}(s),\eta_+^{(m)}(s)}^{(m)}+\delta X_{\eta_+^{(m)}(s),\tilde{\eta}^{(m)}(t)}^{(m)}+2^m(t-\tilde{\eta}^{(m)}(t))\delta X_{\tilde{\eta}^{(m)}(t),\eta_+^{(m)}(s)}^{(m)},\\
|\delta \overline{X}_{s,t}^{(m)}|\leq& (\eta_+^{(m)}(s)-s)^\theta \|X_\cdot^{(m)}\|_\theta+(\tilde{\eta}^{(m)}(t)-\eta_+^{(m)}(s))^\theta \|X_\cdot^{(m)}\|_\theta+(t-\tilde{\eta}^{(m)}(t))^\theta \|X_\cdot^{(m)}\|_\theta\\
\leq& 3(t-s)^\theta\|X_\cdot^{(m)}\|_\theta.
\end{align*}
Now, we can conclude that $|\delta \overline{X}_{s,t}^{(m)}|\leq 3(t-s)^\theta\|X_\cdot^{(m)}\|_\theta$.

Simlarly, we have $\|\delta \overline{X}_{s,t}^{(m)}\|_{L^p}\leq 3(t-s)^\theta\|X_\cdot^{(m)}\|_{L^p,\theta}$.
\end{proof}

\begin{lmm}\label{lemma:discreteGRR}
Let $p>1,\theta'>p^{-1}$.
Also, let $\{X^{(m)}\}_{m=1}^\infty$ be a sequence of discrete stochastic processes such that $X^{(m)}$ is defined in $\mathbb{P}_m$ and $\|X_\cdot^{(m)}\|_{L^p,\theta'}\lesssim_m 1$.

Then, for any $\epsilon_1>0$ and $\theta<\theta'-p^{-1}$, $\sup_m 2^{-m \epsilon_1}\|X_\cdot^{(m)}\|_{\theta}<\infty$ a.s.
\end{lmm}

\begin{proof}
We can show that

\begin{align*}
&\left\|\sup_m 2^{-m \epsilon_1}\|X_\cdot^{(m)}\|_{\theta}\right\|_{L^p}\leq \left\|\sum_{m=1}^\infty 2^{-m \epsilon_1}\|X_\cdot^{(m)}\|_{\theta}\right\|_{L^p}\leq \sum_{m=1}^\infty 2^{-m \epsilon_1}\|\ \|X_\cdot^{(m)}\|_{\theta}\|_{L^p}\\
\leq& \sum_{m=1}^\infty 2^{-m \epsilon_1}\|\ \|\overline{X}_\cdot^{(m)}\|_{\theta}\|_{L^p}\lesssim \sum_{m=1}^\infty 2^{-m \epsilon_1}\|\overline{X}_\cdot^{(m)}\|_{L^p,\theta'}\lesssim \sum_{m=1}^\infty 2^{-m \epsilon_1}\|X_\cdot^{(m)}\|_{L^p,\theta'}\lesssim 1,
\end{align*}	
where each inequality follows from the positivity of $2^{-m \epsilon_1}\|X_\cdot^{(m)}\|_{\theta}$, triangle inequality of $L^p$ norm, Lemma \ref{lemma:discretecontinuous}, Corollary \ref{corollary:stochasticGRR}, Lemma \ref{lemma:discretecontinuous}, the assumption, respectively.

Hence, we can conclude that $\sup_m 2^{-m \epsilon_1}\|X_\cdot^{(m)}\|_{\theta}<\infty$, almost surely.
\end{proof}

Next, we show a better H\"older estimate for Wiener chaos using hypercontractivity. For Wiener chaos, see \cite{Nualart}

\begin{rmk}
As known (for example, shown in \cite{Aida-Naganuma2020}, Appendix A), we can regard $\fBminc{s,u}$ and $\fBmitin_{u,v}^{10*}$ as elements of first Wiener chaos $\mathscr{H}_1$.
$H_r((t-s)^{-H}\fBminc{s,t})$ is in $r$-th Wiener chaos $\mathscr{H}_r$., where $H_r(x)$ is Hermite polynomial.
\end{rmk}

\begin{prp}[Hypercontractivity of Wiener chaos]\label{proposition:hypercontractivity}
Suppose $\mathcal{X}$ is a random variable in $r$-th Wiener chaos.

Then, $\|\mathcal{X}\|_{L^p}\leq C_{p,q,r} \|\mathcal{X}\|_{L^q}$ for any $1<p,q,r<\infty$, where $C_{p,q,r}$ is constant independent of $\mathcal{X}$.
\end{prp}

\begin{cor}\label{corollary:discreteGRRforwienerchaos}
Suppose $\{X_t^{(m)}\}$ is a sequence of discrete stochastic processes defined in $\mathbb{P}_m$ such that $X_t^{(m)}$ in $r$-th Wiener chaos for each $t$.
Furthermore, assume $\|X_\cdot^{(m)}\|_{L^p,\theta'}\lesssim 1$ for some $p>1,0<\theta'<1$.

Then, for any $\epsilon_1>0$, $2^{-m \epsilon_1}\|X_\cdot^{(m)}\|_{\theta}\lesssim 1$ for any $\theta<\theta'$.
\end{cor}

\begin{proof}

Combining the assumption and hypercontractivity, we see $\|X_\cdot^{(m)}\|_{L^q,\theta'}\lesssim_m 1$ for any $q>1$.
Hence, by Lemma \ref{lemma:discreteGRR}, for any $\epsilon_1>0$ and $\theta<\theta'-q^{-1}$,

\begin{equation*}
\sup_m 2^{-m \epsilon_1}\|X_\cdot^{(m)}\|_{\theta}<\infty.
\end{equation*}
Taking $q$ sufficiently large, we can say the statement.
\end{proof}

Next, for some $\Gamma$, we show that $\mathcal{I}^{(m)}(d(B^{\Gamma}))$ are discrete H\"older bounded and converge to standard Brownian motion in law.

\begin{lmm}[\cite{Aida-Naganuma2020},Proposition 4.5]\label{lemma:standardBrown}
Let $k\geq 2$ be integer. Given $0<H< 1/2$, we denote $(t-s)^{1/2}H_l((t-s)^{-H}\fBminc{s,t})$ by $V_{s,t}^{(l)}$.

Then, we have
\begin{align*}
&\left(B_\cdot,\mathcal{I}_\cdot^{(m)}(d(V^{(2)}),\cdots,\mathcal{I}_\cdot^{(m)}(d(V^{(k)}))\right)\to\left(B_\cdot,C_{(2)}W_\cdot^{(2)},\cdots,C_{(k)}W_\cdot^{(k)}\right)
\end{align*}
in law, where each $W^{(l)}$ is standard Brownian motion independent of $B$ and each other, and each constant $C_{(l)}$ for $l\geq 2$ is defined as
\begin{align*}
C_{(l)}=&\sqrt{l!\left(1+\frac{1}{2^l}\sum_{r=1}^\infty(|r+1|^{2H}+|r-1|^{2H}-2|r|^{2H})^l\right)}.
\end{align*}
Moreover, for any $\epsilon_1,\epsilon_2>0$, $2^{-\epsilon_1m}\|\mathcal{I}_\cdot^{(m)}(d(V^{(l)}))\|_{1/2-\epsilon_2}\lesssim_m 1\ a.s.$

\end{lmm}

\begin{proof}
Similar to Proposition 4.5 in \cite{Aida-Naganuma2020}, we can prove convergence in law.
Moreover, the following estimate is shown in the proof of the proposition.
\begin{align*}
\|\mathcal{I}_\cdot^{(m)}(d(V^{(l)}))\|_{L^2}\lesssim_m (t-s)^{1/2}.
\end{align*}
Hence, using Corollary \ref{corollary:discreteGRRforwienerchaos}, we see that
\begin{align*}
2^{-\epsilon_1m}\|\mathcal{I}_\cdot^{(m)}(d(V^{(l)}))\|_{1/2-\epsilon_2}\lesssim_m 1\ a.s.
\end{align*}

\end{proof}
Next, for $B^{10*},B^{110*},B^{101*},B^{011*}$ we obtain similar to Lemma \ref{lemma:standardBrown}.
However, required discussion is beyond the scope of \cite{Aida-Naganuma2020}, so we describe these discussions.
Now, we can obtain the following proposition from the definition of H\"older norm.

\begin{prp}\label{proposition:degeneratediscreteHoelder}
Suppose that $X^{(m)}$ is discrete-time stochastic process defined over $\mathbb{P}_m$ and $0<\theta_1,\theta_2<1$ such that $\theta_1+\theta_2<1$.

Then, $\|X_\cdot^{(m)}\|_{\theta_1+\theta_2}\leq 2^{m\theta_1}\|X_\cdot^{(m)}\|_{\theta_2}$.
\end{prp}

Also, substituting $(t-s)^{-H}\fBminc{s,t}$ for the inverse explicit expression formula of Hermite polynomials, we have the following decomposition of $\fBminc{s,t}^l$ with the Hermite polynomial $H_l$.
We get this result 
\begin{prp}\label{proposition:decomposition_from_power_to_Hermite}
\begin{align}
\fBminc{s,t}^l=&(t-s)^{lH}\sum_{j=0}^{ \lfloor l/2\rfloor}\frac{l!}{2^j(l-2j)!j!}H_{l-2j}((t-s)^{-H}\fBminc{s,t}),\label{decompose_power_B}
\end{align}
\end{prp}

Now, combining Lemma \ref{lemma:standardBrown}, Proposition \ref{proposition:degeneratediscreteHoelder} and Proposition \ref{proposition:decomposition_from_power_to_Hermite}, we can see the next Corollary \ref{corollary:BdiscreteHoelder}.

\begin{cor}\label{corollary:BdiscreteHoelder}
Given $0<H< 1/2$, we have
\begin{align}
2^{(kH-H^c)m}\|\mathcal{I}_\cdot^{(m)}(d(\tilde{B}^{(k)}))\|_{H^c}\lesssim&_m 1\ a.s\ \text{for even k},\\
2^{(kH-H^c)m}\|\mathcal{I}_\cdot^{(m)}(d(B^k))\|_{H^c}\lesssim&_m 1\ a.s\ \text{for odd k},
\end{align}
where $\displaystyle\moment{k}=\frac{k!}{2^{ \lfloor k/2\rfloor}( \lfloor k/2\rfloor )!},\tilde{B}_{s,t}^{(k)}=\fBminc{s,t}^k-\moment{k}(t-s)^{kH}$.
\end{cor}

\begin{lmm}\label{lemma:estimateexpectedvalue}

For $0<H< 1/2, s,t,s',t'\in [0,T]$ such that $s\neq s',s<t,s'<t'$
\begin{align*}
E[\fBmitin_{s,t}^{10*}\fBmitin_{s',t'}^{10*}]=&O\left(|t-s|\cdot |t'-s'|\cdot |s-s'|^{2H}\left(\frac{|t'-s'|}{|s-s'|}+\frac{|t-s|}{|s-s'|}\right)^3\right),\\
E[\fBminc{s,t}\fBmitin_{s',t'}^{10*}]=&O\left(\frac{|t'-s'|^3}{|t-s'|^{2-2H}}+\frac{|t'-s'|^3}{|s-s'|^{2-2H}}\right),\\
E[\fBminc{s,t}\fBmitin_{t,t'}^{10*}]=&-\frac{1-2H}{4(1+2H)}|t'-t|^{1+2H}+O\left(\frac{|t'-t|^3}{|t-s|^{2-2H}}\right),
\end{align*}
where we use Landau's symbol $O(x)$ as mentioned in \ref{section:notation}.

\end{lmm}

\begin{rmk}
In this paper, we use the above estimations only when the values of comparison functions are smaller than $1$ if $m\geq 2$.
Also, when $s=s',t=t'$, the right-hand side of the first and second estimate is infinity. However, we can obtain that
\begin{equation}
E[(\fBmitin_{s,t}^{10*})^2]=C_H^{10*,10*}(t-s)^{2+2H},E[\fBminc{s,t}\fBmitin_{s,t}^{10*}]=\frac{1}{2}(t-s)^{1+2H}
\end{equation}
from the covariance function of fBm, where $C_H^{10*,10*}=E[(\fBmitin_{0,1}^{10*})^2]$

\end{rmk}

\begin{proof}

Denote $E[\fBminc{s,s+r_1(t-s)},\fBminc{s',s'+r_2(t'-s')})]$ for $0<r_1,r_2<1$ by $\text{Cov}_*^B(r_1,r_2)$.

Firstly, to estimate $E[\fBmitin_{s,t}^{10*}\fBmitin_{s',t'}^{10*}]$, we calculate $\text{Cov}^B(r_1,r_2)$ exactly and factor out $|s'-s|^{2H}$. Then, we get
\begin{align*}
\text{Cov}_*^B(r_1,r_2)=&\frac{1}{2}(|s'-s+r_2(t'-s')|^{2H}+|s'-s-r_1(t-s)|^{2H}\\
&-|s'-s+r_2(t'-s')-r_1(t-s)|^{2H}-|s'-s|^{2H})\\
=&\frac{1}{2}|s'-s|^{2H}\left(\left|1+r_2\frac{|t'-s'|}{|s'-s|}\right|^{2H}+\left|1-r_1\frac{|t-s|}{|s'-s|}\right|^{2H}\right.\\
&\left.-\left|1+\frac{r_2(t'-s')-r_1(t-s)}{|s-s'|}\right|^{2H}-1\right).
\end{align*}

Secondly, by using Taylor expansion:
\begin{equation*}
(1+x)^\alpha = 1+\alpha x+\frac{1}{2}\alpha(\alpha-1)x^2+\frac{1}{6}\alpha(\alpha-1)(\alpha-2)x^3+O(x^4),
\end{equation*}
we estimate $\text{Cov}_*^B(r_1,r_2)$ as follows:
\begin{align*}
&Q(r_1,r_2)\\
=&\frac{1}{2}|s'-s|^{2H}\left(2H(2H-1)r_1r_2\frac{|t'-s'|}{|s-s'|}\frac{|t-s|}{|s-s'|}+O\left(\left(\frac{|t'-s'|}{|s-s'|}+\frac{|t-s|}{|s-s'|}\right)^3\right)\right).
\end{align*}

Thirdly, we denote $\text{Cov}^B(r_1,r_2)-r_1r_2\text{Cov}_*^B(1,1)$ by $\tilde{R}(r_1,r_2)$. Then, it naturally follows that
\begin{equation}
\tilde{R}(r_1,r_2)=O\left(|s-s'|^{2H}\left(\frac{|t'-s'|}{|s-s'|}+\frac{|t-s|}{|s-s'|}\right)^3\right)\label{estimate_of_tilde_R_1}.
\end{equation}
At the same time, by the definition, we can rewrite $E[\fBmitin_{s,t}^{10*}\fBmitin_{s',t'}^{10*}]$ as follows:
\begin{equation*}
E[\fBmitin_{s,t}^{10*}\fBmitin_{s',t'}^{10*}]=(t-s)(t'-s')\int_{[0,1]^2}\tilde{R}(r_1,r_2) dr_1dr_2.
\end{equation*}
Substituting (\ref{estimate_of_tilde_R_1}) to this, we conclude the estimate for $E[\fBmitin_{s,t}^{10*}\fBmitin_{s',t'}^{10*}]$.

Next, to estimate $E[\fBminc{s,t}\fBmitin_{s',t'}^{10*}]$, we transform $\text{Cov}_*^B(1,r_2)$ In a similar way:
\begin{align*}
&\text{Cov}_*^B(1,r_2)\\
=&\frac{1}{2}(|s'-s+r_2(t'-s')|^{2H}+|s'-t|^{2H}-|s'-t+r_2(t'-s')|^{2H}-|s'-s|^{2H})\\
=&\frac{1}{2}|s'-s|^{2H}\left(\left|1+r_2\frac{|t'-s'|}{|s'-s|}\right|^{2H}-1\right)-\frac{1}{2}|s'-t|^{2H}\left(\left|1+r_2\frac{|t'-s'|}{|s'-t|}\right|^{2H}-1\right).
\end{align*}
Hence, by using Taylor expansion, we get
\begin{equation}
\tilde{R}(1,r_2)=O\left(|t-s'|^{2H}\frac{|t'-s'|^2}{|t-s'|^2}\right)+O\left(|s-s'|^{2H}\frac{|t'-s'|^2}{|s-s'|^2}\right)\label{estimate_of_tilde_R_2}.
\end{equation}

At the same time, by the definition, we have
\begin{equation*}
E[\fBminc{s,t}\fBmitin_{s',t'}^{10*}]=(t'-s')\int_0^1 \tilde{R}(1,r_2)dr_2.
\end{equation*}
Substituting (\ref{estimate_of_tilde_R_2}) to this, we conclude the estimate for $E[\fBminc{s,t}\fBmitin_{s',t'}^{10*}]$.

At last, to estimate $E[\fBminc{s,t}\fBmitin_{t,t'}^{10*}]$, we fix $s'=t$. Then, we can calculate $\text{Cov}_*^B(1,r_2)$ as
\begin{align*}
\text{Cov}_*^B(1,r_2)=&\frac{1}{2}(|t-s+r_2(t'-t)|^{2H}-|r_2(t'-t)|^{2H}-|t-s|^{2H})\\
=&\frac{1}{2}|t-s|^{2H}\left(\left|1+r_2\frac{|t'-t|}{|t-s|}\right|^{2H}-1\right)-\frac{1}{2}|r_2(t'-t)|^{2H}.
\end{align*}
Now, using Taylor expansion, we can estimate $\tilde{R}(1,r_2)$:
\begin{equation*}
\tilde{R}(1,r_2)=\frac{1}{2}|r_2(t'-t)|^{2H}-\frac{1}{2}r_2|t'-t|^{2H}+O\left(|t-s|^{2H}\frac{|t'-t|^2}{|t-s|^2}\right).
\end{equation*}
Following the above two estimates, we get
\begin{equation*}
E[\fBminc{s,t}\fBmitin_{t,t'}^{10*}]=(t'-t)^{1+2H}\int_0^1 \frac{1}{2}|r_2|^{2H}-\frac{1}{2}|r_2|dr_2+O\left(\frac{|t'-t|^3}{|t-s|^{2-2H}}\right).
\end{equation*}
Since $\int_0^1\frac{1}{2}|r|^{2H}dt=\frac{1}{2(1+2H)}$, we get the estimate for $E[\fBminc{s,t}\fBmitin_{t,t'}^{10*}]$.

\end{proof}

Similar to $B^{10*}$ being contained in first Wiener chaos, $B^{110*},B^{101*},B^{011*}$ are contained in second Wiener chaos.
For these iterated integrals, we can get the estimates similar to Proposition 4.5 in \cite{Aida-Naganuma2020}. (For $B^{10*}$, we can obtain the estimates using Lemma \ref{lemma:estimateexpectedvalue}.
For others, we omit estimates to avoid repetition of discussion in \cite{Aida-Naganuma2020})
By the same discussion as Lemma \ref{lemma:standardBrown}, we have following estimates for $B^{10*},B^{110*},B^{101*},B^{011*}.$

\begin{prp}\label{proposition:Hoelder_estimate_iteerated_integral}
Let $\Gamma =10*,110*,101*,011*$. Given $0<H< 1/2$, we see that for any $\epsilon_1,\epsilon_2>0$,
\begin{align*}
2^{(H+1/2-\epsilon_1)m}\|\mathcal{I}_\cdot^{(m)}(d(B^\Gamma))\|_{1/2-\epsilon_2}\lesssim_m 1,\ a.s,\\
\end{align*}
Moreover, we have
\begin{align*}
&\left(B_\cdot,2^{(H+1/2)m}\mathcal{I}_\cdot^{(m)}(d(B^{10*}))\right)\to\left(B_\cdot,C_{10*}W_\cdot^{(k)}\right),
\end{align*}
where $C_{10*}$ is positive constant defined as
\begin{equation*}
C_{10*}=\sqrt{E\left[(\fBmitin_{0,1}^{10*})^2\right]+2\sum_{r=1}^\infty E[\fBmitin_{0,1}^{10*}\fBmitin_{r,r+1}^{10*}]}.
\end{equation*}
\end{prp}

\begin{proof}
We only prove in the case of $B^{10*}$ using Lemma \ref{lemma:estimateexpectedvalue}. Let $s,t\in \mathbb{P}_m$ such that $s<t$.
\begin{align*}
&E[(\mathcal{I}_{s,t}^{(m)}(d(B^{10*})))^2]=\sum_{r_1,r_2 =\lfloor 2^ms\rfloor}^{ \lfloor 2^mt\rfloor-1}E[\fBmitin_{\bitime{r_1}}^{10*}\fBmitin_{\bitime{r_2}}^{10*}]\\
=&\sum_{r_1,r_2 =\lfloor 2^ms\rfloor}^{ \lfloor 2^mt\rfloor-1}O\left(2^{-2m}\cdot (2^{-m}\vee |\bitime{r_1}-\bitime{r_2}|)^{2H}\cdot\left(\frac{2^{-m}}{2^{-m}\vee |\bitime{r_1}-\bitime{r_2}|}\right)^3\right)\\
=&\sum_{r_1,r_2 =\lfloor 2^ms\rfloor}^{ \lfloor 2^mt\rfloor-1}2^{-(2+2H)m}O\left(\frac{1}{1\vee |r_1-r_2|^{3-2H}}\right)\\
=&O\left((t-s) 2^{-(1+2H)m}\right)
\end{align*}

Combining this estimate and Corollary \ref{corollary:discreteGRRforwienerchaos}, we get the first estimate.
\end{proof}

Using Proposition \ref{proposition:degeneratediscreteHoelder}, we naturally see the next corollary.
\begin{cor}\label{corollary:estimate110}

Let $\Gamma=110*,101*,011*$. Given $0<H< 1/2$, we see that for any $\epsilon_1,\epsilon_2>0$,
\begin{align*}
2^{(2H+\epsilon_1)m}\|\mathcal{I}_\cdot^{(m)}(d(B^\Gamma))\|_{1-\epsilon_1-\epsilon_2}\lesssim_m 1,\ a.s.\\
\end{align*}

\end{cor}
\section{Main theorem}
\label{section:maintheorem}

In this section, we will state determine the asymptotic error of the ($k$)-Milstein scheme and the Crank-Nicolson scheme.
This result is consequence of estimation for general numerical schemes.
Since the case $H\geq 1/2$ have been studied well by many reseachers, we only discuss in the case of $H<1/2$.

At first, we set a general form of numerical schemes below:

\begin{defi}\label{defi:primalcondition}
We consder numerical schemes $\nusol{t}$ with the parameter $m\in \mathbb{Z}$ in the form of
\begin{equation}
\nusol{t}=\begin{cases}
\nusoldis{r}+F[\sigma,b,B,\nusoldis{r},\nusol{t},m],&\timeinunit ,r\in \mathbb{Z}_{\geq 0},\\
Y_0,&t=0.
\end{cases}\label{numerical_scheme_general}.
\end{equation}
\end{defi}
Next, we introduce a condition for numerical solutions taking the form (\ref{numerical_scheme_general}).
\begin{con}\label{condition:secondalycondition}

For integers $L_1$ and $L_2$, let $\sigma\in C_b^{L_1}(\mathbb{R})$ and $b\in C_b^{L_2}(\mathbb{R})$. 
Also, denote $\min \{L_1H,L_2H+1\}$ by $R$.
Then, we define the Condition \ref{condition:secondalycondition} for numerical scheme as following property:

For each $\hat{\Gamma}\in \hat{\mathbb{G}}_{0,R},$ we denote $\lceil H^{-1}(R-|\hat{\Gamma}|)\rceil$ by $L_{\hat{\Gamma}}$.
Then, there exists $L_{\hat{\Gamma}}$-differentiable function $\hat{f}_{\hat{\Gamma}}$ which is a polynomial of $\{\sigma^{(j)}\}_{j=0}^{L_1}$ and $\{b^{(j)}\}_{j=0}^{L_2}$ such that for any positive constant $A<R$ and sufficiently large integer $m$
\begin{equation*}
F[\sigma,b,B,\nusoldis{r},\nusol{t},m]=\sum_{\hat{\Gamma}\in \hat{\mathbb{G}}_{0,A}} \hat{f}_{\hat{\Gamma}}(\nusoldis{r}) \fBmitinunit^{\hat{\Gamma}}+\hat{\epsilon}_{\bitime{r},t}^{(m,A)},
\end{equation*}
where $\hat{\epsilon}_{\bitime{r},t}^{(m,A)}$ estimated as $|\hat{\epsilon}_{\bitime{r},t}^{(m,A)}|\lesssim \timeunit^{A+\kappa}$ for some $\kappa>0$.
\end{con}

This condition corresponds to stochastic Taylor expansion for the exact solution..

\begin{prp}[stochastic Taylor exapnsion]\label{proposition:taylorfory}

Given $\sigma\in C_b^{L_1}(\mathbb{R}),b\in C_b^{L_2}(\mathbb{R})$ with integers $L_1,L_2$, for any positive constant $A$ such that $\min \{L_1H,L_2H+1\}>A$ and sufficiently large integer $m$, $\exsol{t}-\exsol{s}$ can be expanded into the following series:

\begin{equation*}
\exsol{t}-\exsol{s}=\sum_{\Gamma\in \mathbb{G}_{0,A}} f_{\Gamma}(\exsol{s}) B^{\Gamma}_{s,t}+\epsilon_{s,t}^{(A)},
\end{equation*}
with $f_{\Gamma}\coloneqq \mathcal{D}^{\Gamma} (Id)(x)$ and $\epsilon_{s,t}^{(A)}$ is a remainder part estimated as $|\epsilon_{s,t}^{(A)}|\lesssim (t-s)^{A+\kappa}\ a.s$ for some $\kappa>0$, where $Id$ is identity function such that $Id(x)=x$.
\end{prp}

Now, it is obvious that $(k)$-Milstein scheme and Crank-Nicolson scheme is defined in the form of (\ref{numerical_scheme_general}).
Also, we can see the ($k$)-Milstein scheme satisfies Condition \ref{condition:secondalycondition} naturally.
At the same time, we can show that the Crank-Nicolson scheme also satisfies Condition \ref{condition:secondalycondition} when they are well-defined.
For symplicity, we ignore $b$. For $\timeinunit$,

\begin{align}
\sigma(\nusol{t})=&\sigma(\nusoldis{r})+\sigma'(\nusoldis{r})(\nusol{t}-\nusoldis{r})+\frac{1}{2}\sigma''(\nusoldis{r})(\nusol{t}-\nusoldis{r})^2\nonumber\\
&+O(|t-\bitime{r}|^3)\nonumber,\\
\nusol{t}=&\nusoldis{r}+\frac{1}{2}(\sigma(\nusoldis{r})+\sigma(\nusol{t}))\fBmunit\nonumber\\
=&\nusoldis{r}+\sigma(\nusoldis{r})\fBmunit+\frac{1}{2}\sigma'(\nusoldis{r})(\nusol{t}-\nusoldis{r})\fBmunit\nonumber\\
&+\frac{1}{4}\sigma''(\nusoldis{r})(\nusol{t}-\nusoldis{r})^2\fBmunit+O(|t-\bitime{r}|^4)\nonumber\\
=&\nusoldis{r}+\sigma(\nusoldis{r})\fBmunit+\frac{1}{2}\sigma\sigma'(\nusoldis{r})\fBmunit^2\label{expansionCN}\\
&+\frac{1}{4}(\sigma')^2(\nusoldis{r})(\nusol{t}-\nusoldis{r})\fBmunit^2\nonumber\\
&+\frac{1}{4}\sigma''(\nusoldis{r})(\nusol{t}-\nusoldis{r})^2\fBmunit+O(|t-\bitime{r}|^4)\nonumber\\
=&\nusoldis{r}+\sigma(\nusoldis{r})\fBmunit+\frac{1}{2}\sigma\sigma'(\nusoldis{r})\fBmunit^2\nonumber\\
&+\frac{1}{4}((\sigma')^2\sigma+\sigma''\sigma^2)(\nusoldis{r})\fBmunit^3+O(|t-\bitime{r}|^4),\nonumber
\end{align}
where we use $O(x)$ in the sense of $(x\to 0)$.

Next, we set further conditions.

\begin{con}\label{definition:condition_A_B}
We denote $(l,0)\in\in \hat{\mathbb{G}}$ by $[l]$ for $l\in \mathbb{N}_{\geq 1}$.
Also, for $\nusol{\cdot}$ satisfies Condition \ref{condition:secondalycondition},  we denote $\hat{f}_{[l]}-\frac{1}{l!}f_{(l)}$ by $\overline{f}_l$.
Besides, given $H\in (0,1)$, $q$ is an integer $(q+1)^{-1}<H\leq q^{-1}$

Then, we define the Condition ($\mathfrak{A}$) and ($\mathfrak{B}$) as follows:

\begin{itemize}
\item Condition ($\mathfrak{A}$): $\overline{f}_k$ vanishes for any $1\leq k\leq q$.
\item Condition ($\mathfrak{B}$): $q$ is odd, and $\overline{f}_k$ vanishes for $1\leq k\leq q-1$.
\end{itemize}
\end{con}

\begin{exap}

Let
\begin{align}
\Xi[f;x,s,t]\coloneqq \sum_{k=1}^q \frac{1}{k!}\mathcal{D}^{k-1}f(x_s)\xunit^k.\label{Xi_f}
\end{align}
Then, the exact solution $\exsol{}$ holds that $\exsol{t}-\exsol{s}=\Xi[\sigma;\exsol{},s,t]+\epsilon_{s,t}^{(1)}.$\newline
On the other hand, Condition ($\mathfrak{A}$) means $\nusol{t}-\nusol{\bitime{r}}=\Xi[\sigma;\nusol{},\bitime{r},t]+\epsilon_{s,t}^{(1)}.$\newline
Also, Condition ($\mathfrak{B}$) means $\nusol{t}-\nusol{\bitime{r}}=\Xi[\sigma;\nusol{},\bitime{r},t]+\overline{f}_q(\nusol{\bitime{r}})+\epsilon_{s,t}^{(1)}.$

\end{exap}

\begin{exap}

For the ($k$)-Milstein scheme, we can see that

\begin{equation*}
\overline{f}_j= \begin{cases}
0,&1\leq j\leq k\\
-\frac{1}{j!}\mathcal{D}^{j-1}\sigma,&j\geq k+1
\end{cases},
\end{equation*}
by the definition. Simultaneously, we can conclude that

\begin{itemize}
\item for arbitrary integer $k$, the ($k$)-Milstein scheme satisfies Condition ($\mathfrak{A}$) if\newline$H>1/(k+1)$,
\item for even integer $k$, the ($k$)-Milstein scheme satisfies Condition ($\mathfrak{B}$) if\newline$1/(k+2)<H\leq 1/(k+1)$.
\end{itemize}
At the same time, for the Crank-Nicolson scheme, we can see from (\ref{expansionCN}) that

\begin{align}
\hat{f}_{[1]}=\sigma&&\hat{f}_{[2]}=\frac{1}{2}\sigma'\sigma&&\hat{f}_{[3]}=\frac{1}{4}\mathcal{D}^2\sigma.
\end{align}

Immediately,

\begin{align}
\overline{f}_1,\overline{f}_2\equiv 0,&&\overline{f}_3\equiv \frac{1}{12}\mathcal{D}^2\sigma.
\end{align}

Hence, we conclude that

\begin{itemize}
\item the Crank-Nicolson scheme satisfies Condition ($\mathfrak{A}$) if $H>1/3$.
\item the Crank-Nicolson scheme satisfies Condition ($\mathfrak{B}$) if $1/4<H\leq 1/3$.
\end{itemize}
\end{exap}
Each of Condition ($\mathfrak{A}$) and ($\mathfrak{B}$) is a sufficient condition of that $\nusol{\cdot}$ converges to $Y$, which follows from Lemma \ref{lemma:apriori_estimate}.

\begin{prp}\label{proposition:apriori_convergnece}

Let $\sigma\in C_b^{q+3},b\in C_b^2$. If Condition ($\mathfrak{A}$) or ($\mathfrak{B}$) holds, then $\nusol{\cdot}$ converges to $Y$ in $D[0,T]$ with the Skorokhod topology almost surely.
\end{prp}

Next, we formulate asymptotic error distribution.

\begin{defi}[Asymptotic error distribution]
We regard numerical scheme $\nusol{\cdot}$ and exact solution $\exsol{\cdot}$ as random variables valued in $D[0,T]$.
For $\nusol{\cdot}$ converge to $\exsol{\cdot}$, when there exists non-identically zero process $G$ and positive strictly decreasing sequence $R(m)$ such that
\begin{equation}
\lim_{m\to\infty}R(m)^{-1}(\error{\cdot})=G_\cdot
\end{equation}
in the sense of distribution or almost surely, we call the above limit the asymptotic error of $\nusol{\cdot}$ and call R(m) a convergence rate of $\nusol{t}$.
\end{defi}
By definition, the asymptotic error is not unique and dependent on $R(m)$, but unique except for constant-scale differences.

At last, we state the main theorem.

\begin{defi}\label{definition:ZMA}

For $\sigma\in C_b^{q+3},b\in C_b^2, A>1$, we define $\epsilon_{\bitime{r},t}^{\dagger,(A)}$ and $Z_t^{M,(A)}$ as follows:

\begin{align}
\epsilon_{\bitime{r},t}^{\dagger,(A)}=\sum_{\hat{\Gamma}\in \hat{\mathbb{G}}_{0,A}} \hat{f}_{\hat{\Gamma}}(\exsoldis{r}) \fBmitinunit^{\hat{\Gamma}}-\sum_{\Gamma\in \mathbb{G}_{0,A}}f_{\Gamma}(\exsoldis{r})\fBmitinunit^{\Gamma},&&Z_\cdot^{M,(A)}=\mathcal{I}^{(m)}(J_{\cdot+}^{-1},d(\epsilon^{\dagger,(A)})),
\end{align}
where $r$ is taken such that $\timeinunit $ and $\Jacobi{t}$ is Jacobian process as

\begin{equation}
\Jacobi{t}\coloneqq \exp\left(\int_0^t \sigma'(\exsol{s})dB_s+\int_0^t b'(\exsol{s})ds\right).\label{definitionJ}
\end{equation}
\end{defi}.

\begin{rmk}
$\Jacobi{t}$ corresponds the partial derivative of $\exsol{t}$ with respect to the initial value.
\end{rmk}

Now, we can state estimate for general numerical schemes. 
We will present a proof of this theorem in Section \ref{section:proof_of_maintheorem}.

\begin{thm}[estimate for general scheme]\label{theorem:maintheorem}
For $\sigma\in C_b^{q+3},b\in C_b^2$ and $0<H<1/2$, fix a numerical scheme $\nusol{\cdot}$ which satisfies Condition \ref{condition:secondalycondition} and Condition ($\mathfrak{A}$)\ (resp. ($\mathfrak{B}$)).
Also, given positive strictry decreasing sequence $R(m)$ and $A>1$ such that $2^{-m(A-1)}=O(R(m))$, assume that 
\begin{equation}
\displaystyle\lim_{m\to\infty}R(m)^{-1}2^{-\lambda m}\max_{t\in [0,T]}|Z_t^{M,(A)}|=0\label{main_theorem_assumption_lambda}
\end{equation}
a.s for arbitrary $\lambda>0$.
Additionally, suppose $\hat{f}_{\hat{\Gamma}}$ is Lipschitz continuous for any $\hat{\Gamma}\in \hat{\mathbb{G}}_{1,A}$.

Then the following limit holds for some $\kappa>0$:
\begin{equation*}
2^{m\kappa}R(m)^{-1}\max_{t\in [0,T]}|\nusol{t}-\exsol{t}-\Jacobi{t} Z_t^{M,(A)}|\xrightarrow{m\to\infty} 0\ a.s.
\end{equation*}
(resp.$2^{m\kappa}R(m)^{-1}\left\|\error{\cdot}-\Jacobi{\cdot} Z_\cdot^{M,(A)}\right\|_{H'}\to 0$\ a.s. if $R(m)^{-1}2^{-\lambda m}\|Z_\cdot^{M,(A)}\|_{H'}.\to 0$)

\end{thm}

\subsection{Asymptotic error of the ($k$)-Milstein scheme and the Crank-Nicolson scheme}

In this subsection, we will show the asymptotic error we determined.
These distributions are calculated in Section \ref{section:calculation}.
Throughout this subsection, each limit is defined in the sense of the Skorokhod topology.

Firstly, we show the asymptotic error of the ($k$)-Milstein scheme. 
We recall Nourdin \cite{Nourdin2005} made counterexamples which fails to converege for $H$ and $k$ not in the following list. 
Hence asymptotic error of $(k)$-Milstein scheme is completely deterimined.

\begin{thm}\label{theorem:Theorem_of_distribution_M}
Let $\sigma\in C_b^{k+1}, b\in C_b^2,0<H<1,k\in \mathbb{Z}_{\geq 2}, l\in \mathbb{Z}_{\geq 1}$.

We set $\tilde{g}_{k,1}^{M},\ \tilde{g}_{k,2}^M,\ \overline{g}_{2l+1,1}^M$ and $\tilde{g}_{(1,2)}^{M}$ as follows:

\begin{align*}
\tilde{g}_{k,1}^{M}=&-\frac{1}{k!}\mathcal{D}^{k-1}\sigma&\tilde{g}_{k,2}^{M}=&\tilde{g}_{k,1}^{M}-\sigma'\tilde{g}_{k-1,1}^{M}&\overline{g}_{2l+1,1}^M=&\tilde{g}_{2l+2,2}^M-\frac{1}{2}\mathcal{V}\tilde{g}_{2l+1,1}^M,
\end{align*}

\begin{equation*}
\tilde{g}_{(1,2)}^{M}=-\frac{6H-1}{4(1+2H)}(\sigma\sigma''b+\sigma\sigma'b')-\frac{3-2H}{4(1+2H)}((\sigma')^2b+\sigma^2b''),
\end{equation*}
where $\mathcal{V}:f\mapsto \sigma f'-\sigma'f$ defined in Section \ref{section:notation}.

Also, for $l\geq 2$, we set positive constant $\moment{l}=\displaystyle\frac{l!}{2^{ \lfloor l/2\rfloor}( \lfloor l/2\rfloor )!}$.

\begin{itemize}
\item Then, if $b$ does not vanish, the asymptotic error of the ($k$)-Milstein scheme is determined for each $k$ as follows:

\begin{itemize}
\item for $k=2$,

\begin{itemize}
\item if $1/4<H<1/2$ and $\sigma\in C_b^{q+3},b\in C_b^2$, then

\begin{equation}
\lim_{m\to\infty}2^{(4H-1)m}(\error{\cdot})=\moment{3}\Jacobi{\cdot}\int_0^\cdot \Jacobi{t}^{-1}\overline{g}_{3,1}^M(\exsol{t})dt\ a.s,\label{Milstein_asymptotic_novanish_first}
\end{equation}
\end{itemize}

\item for $k=3$,

\begin{itemize}
\item if $1/4<H<1/2$ and $\sigma\in C_b^{q+3},b\in C_b^2$, then

\begin{equation}
\lim_{m\to\infty}2^{(4H-1)m}(\error{\cdot})=\moment{4}\Jacobi{\cdot}\int_0^\cdot \Jacobi{t}^{-1}\tilde{g}_{4,1}^M(\exsol{t})dt\ a.s,
\end{equation}
\end{itemize}
\item for even $k\geq 4$,

\begin{itemize}
\item if $1/k<H<1/2$ and $\sigma\in C_b^{q+3},b\in C_b^3$, then

\begin{equation}
\lim_{m\to\infty}2^{2Hm}(\error{\cdot})=\Jacobi{\cdot}\int_0^\cdot \Jacobi{t}^{-1}\tilde{g}_{(1,2)}^M (\exsol{t})dt\ a.s,
\end{equation}
\item if $H=1/k$ and $\sigma\in C_b^{q+3},b\in C_b^3$, then

\begin{equation}
\lim_{m\to\infty}2^{(2/k)m}(\error{\cdot})=\Jacobi{\cdot}\int_0^\cdot \Jacobi{t}^{-1}(\tilde{g}_{(1,2)}^M+\moment{k+1}\overline{g}_{k+1,1}^M) (\exsol{t})dt\ a.s,
\end{equation}
\item if $1/(k+2)<H<1/k$ and $\sigma\in C_b^{q+3},b\in C_b^2$, then

\begin{equation}
\lim_{m\to\infty}2^{((k+2)H-1)m}(\error{\cdot})=\moment{k+1}\Jacobi{\cdot}\int_0^\cdot \Jacobi{t}^{-1}\overline{g}_{k+1,1}^M(\exsol{t})dt\ a.s,
\end{equation}
\end{itemize}
\item for odd $k\geq 5$,

\begin{itemize}
\item if $1/(k-1)<H<1/2$ and $\sigma\in C_b^{q+3},b\in C_b^3$, then

\begin{equation}
\lim_{m\to\infty}2^{2Hm}(\error{\cdot})=\Jacobi{\cdot}\int_0^\cdot \Jacobi{t}^{-1}\tilde{g}_{(1,2)}^M (\exsol{t})dt\ a.s,
\end{equation}
\item if $H=1/(k-1)$ and $\sigma\in C_b^{q+3},b\in C_b^3$, then

\begin{equation}
\lim_{m\to\infty}2^{(2/(k-1))m}(\error{\cdot})=\Jacobi{\cdot}\int_0^\cdot \Jacobi{t}^{-1}(\tilde{g}_{(1,2)}^M+\moment{k+1}\tilde{g}_{k+1,1}^M) (\exsol{t})dt\ a.s,
\end{equation}
\item if $1/(k+1)<H<1/(k-1)$ and $\sigma\in C_b^{q+3},b\in C_b^2$, then

\begin{equation}
\lim_{m\to\infty}2^{((k+1)H-1)m}(\error{\cdot})=\moment{k+1}\Jacobi{\cdot}\int_0^\cdot \Jacobi{t}^{-1}\tilde{g}_{k+1,1}^M(\exsol{t})dt\ a.s.\label{Milstein_asymptotic_novanish_last}
\end{equation}
\end{itemize}
\end{itemize}

(For each $k$, the case of $H\leq 1/3$ is our new result. On the other hand, the case of $1/3\leq H<1/2$ is shown in \cite{Aida-Naganuma2020} or can be prove by the method of \cite{Aida-Naganuma2020})

\item On the other hand, if $b$ vanishes, the asymptotic error of the ($k$)-Milstein scheme is determined as follows:

\begin{itemize}
\item for even $k$, if $1/(k+2)<H<1/2$ and $\sigma\in C_b^{(k+3)\vee (q+3)}$, then

\begin{equation}
\lim_{m\to\infty}2^{((k+2)H-1)m}(\error{\cdot})=\moment{k+1}\Jacobi{\cdot}\int_0^\cdot \Jacobi{t}^{-1}\moment{2l+1}\overline{g}_{k+1,1}^M(\exsol{t})dt\ a.s,\label{Milstein_asymptotic_vanish_first}
\end{equation}

\item for odd $k$, if $1/(k+1)<H< 1/2$ and $\sigma\in C_b^{(k+2)\vee (q+3)}$, then

\begin{equation}
\lim_{m\to\infty}2^{((k+1)H-1)m}(\error{\cdot})=\moment{k+1}\Jacobi{\cdot}\int_0^\cdot \Jacobi{t}^{-1}\tilde{g}_{k+1,1}^M(\exsol{t})dt\ a.s.\label{Milstein_asymptotic_vanish_last}
\end{equation}
\end{itemize}

(For even $k$, the case of $1/(k+2)<H\leq 1/(k+1)$ is our new result. The other cases are shown in \cite{Gradinaru-Nourdin2009})
\end{itemize}
\end{thm}

Next, we will show the error distribution of the Crank-Nicolson scheme in the range of $1/4<H<1$.

\begin{thm}\label{theorem:Theorem_of_distribution_CN}
Let $W^{10*}$ be a standard Brownian motion independent of $B$ and $W^{(3)}$. Also, we define $C_{10*}$ as

\[C_{10*}=\sqrt{E[(\fBmitin_{0,1}^{10*})^2]+2\sum_{r=1}^\infty E[\fBmitin_{0,1}^{10*}\fBmitin_{r,r+1}^{10*}]},\]
where $E[(\fBmitin_{0,1}^{10*})^2]=\frac{1-2H}{4(1+2H)}$ and

\begin{eqnarray*}
E[\fBmitin_{0,1}^{10*}\fBmitin_{r,r+1}^{10*}]&=&\frac{2r^{2H+2}-(r+1)^{2H+2}-(r-1)^{2H+2}}{2(2H+2)(2H+1)}\\
&&-\frac{(r+1)^{2H+1}-(r-1)^{2H+1}}{2(2H+1)}+\frac{2r^{2H}+(r+1)^{2H}+(r-1)^{2H}}{4}.
\end{eqnarray*}

Then, for $\sigma\in C_b^{q+3},b\in C_b^3, 1/4<H<1/2$, whether or not $b$ vanishes, the asymptotic error distribution of the Crank-Nicolson scheme is determined as follows:

\begin{equation}
\lim_{m\to\infty}2^{(3H-1/2)m}(\error{\cdot})=C_{(3)}\Jacobi{\cdot}\int_0^\cdot \Jacobi{t}^{-1}\tilde{g}_{3}^{CN}(\exsol{t})dW_t^{(3)}\ in\ law,\label{asymptotic_error_distribtion_CN}
\end{equation}
where $\tilde{g}_{3}^{CN}=\frac{1}{12}\mathcal{D}^2\sigma$, $W^{(3)}$ is standard Brownian motion independent of $B$, and $C_{(l)}$ is constant defined in Lemma \ref{lemma:standardBrown}.

(The case of $1/4<H\leq 1/3$ is our new result. On the other hand, the case of $1/3<H\leq 1/2$ and SDE has no drift is shown in \cite{Naganuma2015}. Also, the case of $1/3<H\leq 1/2$ and SDE has some drift is shown in \cite{Aida-Naganuma2020}.)

\end{thm}

\begin{rmk}
Similar to the ($k$)-Milstein's case, Nourdin showed that the ($k$)-Milstein scheme does not converge generally when $H<1/6$.
Then, when $1/6<H\leq 1/4$, the convergence of the Crank-Nicolson scheme is still an open problem.
\end{rmk}
\section{Interpolation between exact and numerical solution}\label{section:interpolation}

In this section, we will separate the error term into individually estimable terms.
To explicitly write each term, we will introduce a parametrized stochastic process $\polsolstd{}$ interpolated between the exact solution $Y$ and the numerical solution $\nusol{}$.

For $\sigma\in C_b^{q+1},b\in C_b^1$, we define $\Omega^{(m)}$ as follows:

\begin{equation*}
\Omega^{(m)}=\left\{\omega\in \Omega_0:\sum_{k=1}^q \frac{1}{k!}\|(\mathcal{D}^{k-1}\sigma)'\|_\infty \|B_\cdot\|_{H_-}^k2^{-kH_-m}+\|b'\|_\infty 2^{-m}<1/2\right\},
\end{equation*}
where $\Omega_0$ is defined as (\ref{Omega_0}).
Note that $\{\Omega^{(m)}\}_{m=1}^\infty$ is increasing family and $\Omega_0=\bigcup_{m=1}^\infty \Omega^{(m)}$.
When we deal the Crank-Nicolson scheme, we assume $\omega\in \Omega^{(m)}$ for appropriate $m$.
Also, from this section onward, we assume $\nusol{}$ satisfies Condition ($\mathfrak{A}$) or ($\mathfrak{B}$).
Definitions in this chapter is based on \cite{Aida-Naganuma2023}.
\subsection{Representation of the error term}

In this subsection, we transform the error term $\error{\cdot}$ using $\polsolstd{\cdot}$ to decompose error.

First, we introduce $\polsolstd{t}$ such that $\polsol{t}{0}=\exsol{t},\ \polsol{t}{1}=\nusol{t}$.

\begin{defi}
We define $\polsolstd{t}$, stochastic process parametrized by $\polpar\in (0,1)$, subsequently as follows:

\begin{equation}
\polsolstd{t}=\begin{cases}
\polsolstd{\bitime{r}}+\Xi[\sigma,\polsolstd{},\bitime{r},t]+b(\polsolstd{\bitime{r}})\timeunit\\
+\rho\left(\overline{f}_q(\nusoldis{r})\fBmunit^q+\hat{\epsilon}_{\bitime{r},t}^{(m,(q+1)H)}\right)+(1-\rho)\epsilon_{\bitime{r},t}^{((q+1)H)},&\timeinunit\\
Y_0,&t=0,
\end{cases}\label{definitionY}
\end{equation}
where $\Xi$ is defined by (\ref{Xi_f}) and $\overline{f}_q$ is defined in Conidition \ref{definition:condition_A_B}.
\end{defi}

Next, we give an explicit indication of the error
\begin{lmm}
\label{lemma:representation of error}

For $\timeinunit$, let $\sigma\in C_b^{q+1},b\in C_b^2$.
We set $\epsilon_t^{(m)}$ and $\Jacobiunit{t}$ as follows:
\begin{equation*}
\epsilon_t^{(m)}\coloneqq \overline{f}_q(\nusoldis{r})\fBmitinunit^{(q)}+\hat{\epsilon}_{\bitime{r},t}^{(m,(q+1)H)}-\epsilon_{\bitime{r},t}^{((q+1)H)},
\end{equation*}
\begin{equation*}
\Jacobiunit{t}\coloneqq 1+\Xi'[\sigma,\polsolstd{},\bitime{r},t]+b'(\polsolstd{\bitime{r}})\timeunit.
\end{equation*}
Moreover, we set $\polJacobistd{t}$ below:
\begin{equation}
\polJacobistd{t}=\prod_{r=0}^{ \lceil 2^m t\rceil-1}\Jacobiunit{\bitimelast{r+1}}\label{definition_M}.
\end{equation}

Then, for any $m\in \mathbb{N}, t\in [0,T], \rho\in (0,1) \omega\in \Omega^{(m)}$, $\Jacobiunit{t}$ and $\polJacobistd{t}$ is strictly positive.
In other words, $(M_t^{m,\rho})^{-1}$ is well-defined. Moreover, we have

\begin{equation*}
\polsol{t}{\rho_0}-\exsol{t}=\int_0^{\rho_0} \polJacobistd{t}\sum_{r=0}^{ \lceil 2^m t\rceil-1} (M_{\bitimelast{r+1}}^{(m,\rho)})^{-1}\epsilon_{\bitimelast{r+1}}^{(m)}d\rho.
\end{equation*}

Especially, we can transform the error term as follows:

\begin{equation}
\nusol{t}-\exsol{t}=\int_0^1 \polJacobistd{t}\sum_{r=0}^{ \lceil 2^m t\rceil-1} (M_{\bitimelast{r+1}}^{(m,\rho)})^{-1}\epsilon_{\bitimelast{r+1}}^{(m)}d\rho\label{transform_error}.
\end{equation}

\end{lmm}

\begin{proof}
By definition, for $\timeinunit$, $\polsolstd{t}$ is differentiable with respect to $\rho$ and 
\begin{equation*}
\partial_\rho \polsolstd{t}=\Jacobiunit{t}\partial_\rho \polsolstd{\bitime{r}}+\epsilon_t^{(m)}.
\end{equation*}
Hence, we can see:
\begin{align*}
\partial_\rho \polsolstd{t}=&\sum_{r=0}^{ \lceil 2^m t\rceil-1} \left(\prod_{r_1=r+1}^{ \lceil 2^m t\rceil-1}\Jacobiunit{\bitimelast{r_1+1}}\right)\epsilon_{\bitimelast{r+1}}^{(m)}\\
=&\polJacobistd{t}\sum_{r=0}^{ \lceil 2^m t\rceil-1} (M_{\bitimelast{r+1}}^{(m,\rho)})^{-1}\epsilon_{\bitimelast{r+1}}^{(m)}.
\end{align*}

Integrating both sides of this equality, we can see that
\begin{equation*}
\polsol{t}{\rho_0}-\exsol{t}=\int_0^{\rho_0} \polJacobistd{t}\sum_{r=0}^{ \lceil 2^m t\rceil-1} (M_{\bitimelast{r+1}}^{(m,\rho)})^{-1}\epsilon_{\bitimelast{r+1}}^{(m)}d\rho.
\end{equation*}
\end{proof}

\subsection{Decomposition of the error term}\label{subsection:decompose_error}

In this subsection, we decompose (\ref{transform_error}) into estimable form.
For simplicity, we set 
\begin{align}
Z_t^{(m,\rho)}\coloneqq (\polJacobistd{t})^{-1}\partial_\rho \polsolstd{t}, && N_t^{(m,\rho)}\coloneqq(\polJacobistd{t})^{-1}\partial_\rho \polJacobistd{t}.\label{definitionZN}
\end{align}
As preparation, let 
\[(\polJacobistd{t})^{-1}=(M_t^{(m,0)})^{-1}+((\polJacobistd{t})^{-1}-(M_t^{(m,0)})^{-1})\eqqcolon V_t^{1,(m)}+V_t^{2,(m,\rho)}.\]

First, we decompose $Z_t^{(m,\rho)}$ into $Z_t^{1,(m)}$ and $Z_t^{2,(m,\rho)}$.
\begin{defi}
We define $Z_t^{1,(m)}$ and $Z_t^{2,(m,\rho)}$ as follows:
\begin{align*}
Z_t^{1,(m)}\coloneqq \mathcal{I}_t^{(m)}(V_{\cdot+}^{1,(m)},d(\epsilon^{(m)})).&&Z_t^{2,(m,\rho)}\coloneqq \mathcal{I}_t^{(m)}(V_{\cdot+}^{2,(m,\rho)},d(\epsilon^{(m)})).
\end{align*}

\end{defi}

Next, we split $Z_t^{1,(m)}$ into $Z_t^{1+,(m)}$ and $Z_t^{1-,(m)}$.

\begin{defi}
We define the following stochastic processes  $Z_t^{1+,(m)}$ and $Z_t^{1-,(m)}$. 
\begin{align*}
Z_t^{1+,(m)}\coloneqq \mathcal{I}_t^{(m)}(J_{\cdot+}^{-1},d(\epsilon^{(m)})),&&Z_t^{1-,(m)}\coloneqq \mathcal{I}_t^{(m)}(V_{\cdot+}^{1,(m)}-J_{\cdot+}^{-1},d(\epsilon^{(m)})).
\end{align*}

\end{defi}

Next, we split $Z_t^{1+,(m)}$ into $Z_t^{1,i,(m,A)}\ (i=1,2,3)$.
In order to this, let $\epsilon_t^{(m)}=\epsilon_t^{1,(m,A)}+\epsilon_t^{2,(m,A)}+\epsilon_t^{3,(m,A)}$
\begin{defi}\label{def:e}

Given $A>1$, for $\timeinunit$, we decompose $\epsilon_t^{(m)}$ as follows:
\begin{align*}
\epsilon_t^{(m)}=&\epsilon_t^{1,(m,A)}+\epsilon_t^{2,(m,A)}+\epsilon_t^{3,(m,A)},\\
\epsilon_t^{1,(m,A)}&\coloneqq \overline{f}_q(\exsoldis{r})\fBmitinunit^{(q)}+\sum_{\hat{\Gamma}\in \hat{\mathbb{G}}_{1,A}} f_{\hat{\Gamma}}(\exsoldis{r})\fBmitinunit^{\hat{\Gamma}}-\sum_{\Gamma\in \mathbb{G}_{1,A}}f_{\Gamma}(\exsoldis{r})\fBmitinunit^{\Gamma},\\
\epsilon_t^{2,(m,A)}&\coloneqq (\overline{f}_q(\nusoldis{r})-\overline{f}_q(\exsoldis{r}))\fBmitinunit^{(q)} +\sum_{\hat{\Gamma}\in \hat{\mathbb{G}}_{1,A}} (f_{\hat{\Gamma}}(\nusoldis{r})-f_{\hat{\Gamma}}(\exsoldis{r}))\fBmitinunit^{\hat{\Gamma}},\\
\epsilon_t^{3,(m,A)}&\coloneqq \hat{\epsilon}_{\bitime{r},t}^{(m,A)}-\epsilon_{\bitime{r},t}^{(A)},
\end{align*}
where $\mathbb{G}_{1,A},\hat{\mathbb{G}}_{1,A}$, defined in (\ref{subscriptA}) and (\ref{subscripthatA}).
Additionally, we set $e_t^{(m)}\coloneqq \mathcal{I}_t^{(m)}(d(\epsilon^{(m)}))$ and $e_t^{i,(m,A)}\coloneqq \mathcal{I}_t^{(m)}(d(\epsilon^{i,(m,A)}))$.
Clearly, $e_t^{(m)}=e_t^{1,(m,A)}+e_t^{2,(m,A)}+e_t^{3,(m,A)}$.
\end{defi}

\begin{rmk}
Whereas $\epsilon_t^{(m)}$ is independent of $A$, $\epsilon_t^{1,(m,A)},\epsilon_t^{2,(m,A)}$ and $\epsilon_t^{3,(m,A)}$ is dependent on $A$.
\end{rmk}

\begin{defi}
Let $i=1,2,3$, we define $Z_t^{1,i,(m,A)}$ as $Z_t^{1,i,(m,A)}\coloneqq \mathcal{I}_t^{(m)}(J_{\cdot+}^{-1},d(\epsilon^{i,(m,A)}))$.
\end{defi}
Now, it is clear that $Z_t^{1+,(m)}=Z_t^{1,1,(m,A)}+Z_t^{1,2,(m,A)}+Z_t^{1,3,(m,A)}$.
By the above definition, we can see the following decomposition:

\begin{align*}
\nusol{t}-\exsol{t}=&\Jacobi{t} Z_t^{1,1,(m,A)}+\int_0^1 \polJacobistd{t}\left((1-\Jacobi{t}(\polJacobistd{t})^{-1})Z_t^{1,1,(m,A)}\right.\\
&+\left.Z_t^{1,2,(m,A)}+Z_t^{1,3,(m,A)}+Z_t^{1-,(m)}+Z_t^{2,(m,\rho)}\right)d\rho.
\end{align*}
Note that $Z_t^{1,1,(m,A)}=Z_t^{M,(A)}$ defined in (\ref{definition:ZMA}).

\section{Estimate of each term and whole remainder term.}\label{section:proof_of_maintheorem}

In this section, we aims to prove of Theorem \ref{theorem:maintheorem}, an estimate for the whole remainder term.
In first subsection, we will estimate for $\polsolstd{},\polJacobistd{},Z^{(m,\rho)},N^{(m,\rho)}$, defined as (\ref{definitionY}), (\ref{definition_M}) and (\ref{definitionZN}).
In second subsection, we will also estimate for $Z_t^{1,2,(m,A)},Z_t^{1,3,(m,A)},Z_t^{1-,(m)},Z_t^{2,(m,\rho)}$, defined in Section \ref{subsection:decompose_error}.
In third subsection, we will prove Theorem \ref{theorem:maintheorem}.
To simplify our discussion, we treat $\polJacobistd{}$ as a discrete-time process defined in $\mathbb{P}_m$.
Additionally, we assume $\nusol{}$ satisfies Condition ($\mathfrak{A}$) or ($\mathfrak{B}$) as in the previous chapter.

\subsection{Estimate of $Y,M,Z,N$}\label{subsection:yMN}

We will show fundamental estimates for
$\polsolstd{},\polJacobistd{},Z^{(m,\rho)},N^{(m,\rho)},\Jacobi{\cdot}(\polJacobistd{\cdot})^{-1}.$
Each estimate will be shown in Lemma \ref{lemma:Hoelder_of_y}, \ref{lemma:Hoelder_of_M}, \ref{lemma:apriori_estimate}, \ref{lemma:Hoelder_of_N}, \ref{lemma:1-JMrho}, respectively.

First, we will show that $\|\polsolstd{\cdot}\|_{H_-}\lesssim_{m,\rho}1$ and $\|\polJacobistd{\cdot}\|_{H_-}\lesssim_{m,\rho}1$. To avoid complicated notation, we firstly treat $\polsolstd{}$, and secondly see that we can treat $(\polsolstd{},\polJacobistd{})$ in the same way.
To show these estimates, we consider exact and numerical solutions of SDE on discrete time $\mathbb{P}_m$.

First, we define discrete controlled path and Stieltjes sum for contorolled path.
\begin{prp}[Discrete controlled path]\label{prp:young_sum}
Let $q\in \mathbb{Z}_{\geq 1}, \alpha \in ((q+1)^{-1}, q^{-1}]$ and $x\in C^{\alpha}([0,T];\mathbb{R})$.
Also, we set $R_{s,t}^{z,(k)}$ for $z = (z^{(1)},\cdots,z^{(q)})\in C^\alpha(\mathbb{R};\mathbb{R}^{d\times q})$, $1\leq k\leq q$ and $s,t\in [0,T]\cap \mathbb{P}_m$ to be

\begin{equation}\label{dcp_def}
R_{s,t}^{z,(k)}=z_t^{(k)}-z_s^{(k)}-\sum_{j=1}^{q-k}\frac{1}{j!}z_s^{(k+j)}\xunit^j
\end{equation}

When $z$ satisfies $R_{s,t}^{z,(k)}=O(|t-s|^{(q-k+1)\alpha})$ for all $k$, we call $z$ a discrete controlled path (of $x$).

Now, we set $\Xi_{s,t}^z=\sum_{l=1}^q \frac{1}{l!}z_s^{(l)}\xunit^l$ for arbitrary $s,t\in [0,T]$.

Then, we have
\begin{equation*}
\mathcal{I}_{s,t}^{(m)}(d(\Xi^z))-\Xi_{s,t}^z\lesssim_m |t-s|^{(q+1)\alpha}.
\end{equation*}
\end{prp}

This proposition is incomplete case of Proposition \ref{prp:continuous_rough_integral}, so that we omit the proof.

\begin{exap}\ 

\begin{itemize}
\item When a controlled path $z$ is restricted to $\mathbb{P}_m$, $z$ is discrete controlled path.
\item Let $z$ be a discrete controlled path and define $f(z)$ such as (\ref{pathcomposition}).
Then $f(z)$ is a discrete controlled path.
\end{itemize}

\end{exap}

Next, we aim to show that $\polsolstd{}$ is a discrete controlled path, so that we have to a few preparation.

\begin{defi}\label{definition:Asymptoticexpansion}

Let $\sigma\in C_b^{q+1},b\in C_b^2$. For $s,t\in \mathbb{P}_m$, we set $\Xi_{s,t}^{Y,(m,\rho)}$ and $R_{s,t}^{Y,(m,\rho)}$ as follows:
\begin{align*}
\Xi_{s,t}^{Y,(m,\rho)}\coloneqq& \Xi[\sigma,\polsolstd{},s,t]+b(\polsolstd{s})(t-s),\\
R_{s,t}^{Y,(m,\rho)}\coloneqq&\polsolstd{t}-\polsolstd{s}-\Xi_{s,t}^{Y,(m,\rho)}\label{Y_m_rho_Xi_R}.
\end{align*}
Next, we decompose $R_{s,t}^{Y,(m,\rho)}$ into $R_{s,t}^{Y,1,(m,\rho)}$ and $R_{s,t}^{Y,2,(m,\rho)}$, where
\begin{align*}
R_{s,t}^{Y,1,(m,\rho)}\coloneqq&\overline{f}_q(\polsolstd{s})\mathcal{I}_{s,t}^{(m)}(d(B^{(q)})),\\
R_{s,t}^{Y,2,(m,\rho)}\coloneqq&R_{s,t}^{Y,(m,\rho)}-R_{s,t}^{Y,1,(m,\rho)}.
\end{align*}

\end{defi}
We remark that $\delta R_{s,u,t}^{Y,2,(m,\rho)}=\delta R_{s,u,t}^{Y,(m,\rho)}=-\delta \Xi_{s,u,t}^{Y,(m,\rho)}$ by definition.
From Corollary \ref{corollary:BdiscreteHoelder}, we have $|R_{s,t}^{Y,1,(m,\rho)}|\lesssim_{m,\rho} (t-s)^{qH_-}$.
Also, if $|Y_{s,t}^{(m,\rho)}|\lesssim_{m,\rho} (t-s)^{H_-}$, we get $|R_{s,t}^{Y,2,(m,\rho)}|\lesssim_{m,\rho} (t-s)^{(q+1)H_-}$ using the sewing lemma.
 Naturally, if $\|\polsolstd{\cdot}\|_{H_-}\lesssim_{m,\rho} 1,$ and $\|R_\cdot^{Y,2,(m,\rho)}\|_{(q+1)H_-}\lesssim_{m,\rho} 1$, we have (\ref{dcp_def}) for $k=1$, $z^{(k)}=\polsolstd{}$. So, we will get these estimates.

\begin{prp}\label{proposition:discrete_sewing_for_rough_path}

Given $f\in C^{q+1},g\in C^1, x\in C^\alpha$ for $(q+1)^{-1}<\alpha\leq q^{-1}$, we set $Z_t,\Xi_{s,t}^Z, R_{s,t}^Z$ as
\begin{align*}
\Xi_{s,t}^Z=& \sum_{k=1}^q \frac{1}{k!}\mathcal{D}_f^{k-1}f(Z_s)\xunit^k+g(Z_s)(t-s),\\
R_{s,t}^Z=&Z_t-Z_s-\Xi_{s,t}^Z.
\end{align*}

Then, there exists constants $C_a$ and $C_b$ depend on $\|f\|_{C^{q+1}}, \|g\|_{C^1}$ and $\|x_\cdot\|_\alpha$ such that 

\begin{equation*}
|\delta \Xi_{s,u,t}^Z|\leq C_a(t-s)^{(q+1)\alpha}+C_b P_q(|R_{s,u}^Z|,(t-s)^{\alpha}),
\end{equation*}
where $P_q(x,y)$ is polynomial defined as
\begin{equation*}
P_q(x,y)=\sum_{k_1=1}^q\sum_{k_2=0}^{q^2} x^{k_1}y^{k_2}.
\end{equation*}

\end{prp}

This estimate appears in proof of Proposition \ref{prp:continuous_rough_integral}, so that we omit the proof.

Since we have already prepared to show that $\polsolstd{}$ satisfies (\ref{dcp_def}), we start to estimate.
Firstly, we will show a short-time range estimate of (\ref{dcp_def}) when $k=1$. Also, we obtain H\"older estimate of $\polsolstd{}$.
We use Davie's technique in \cite{Davie2008}.

\begin{lmm}\label{lemma:expansion_small_distance}
Let $\sigma\in C_b^{q+1},b\in C_b^2$. Also, we define $\|\cdot\|_{H_-,[s,t]}$ as the discrete H\"older norm restricted to the interval [s,t].

Then, there exists positive constants $\delta, C_0,C_1$ independent of $m,\rho, s, t$ such that

\begin{align}
|\polsolstd{s,t}|\leq C_0(t-s)^{H_-},&&|R_{s,t}^{Y,2,(m,\rho)}|\leq C_1(t-s)^{(q+1)H_-}\label{short_time_estimate}
\end{align}
for any $s,t\in \mathbb{P}_m$ satisfies $|t-s|<\delta$.
\end{lmm}
\begin{proof}
Throughout this proof, we use $C_l$ for integer $l$ as some constant independent of $m$ and $\rho$.

Firstly, we aim to define $C_0^{(m,k)}$ and $C_1^{(m,k)}$ for all positive integers $m$ and $k$ such that

\begin{align}
|\polsolstd{s,t}|\leq C_0^{(m,k)}(t-s)^{H_-},&&|R_{s,t}^{Y,2,(m,\rho)}|\leq C_1^{(m,k)}(t-s)^{(q+1)H_-}\label{short_time_estimate_local}
\end{align}
for any discrete times $s,t\in \mathbb{P}_m$ satisfying that $0<t-s\leq k\cdot 2^{-m}$.

First, when $k=1$, by definition,
\begin{align}
Y_{s,t}^{(m,\rho)}=\Xi_{s,t}^{Y,(m,\rho)}+R_{s,t}^{Y,2,(m,\rho)},&&R_{s,t}^{Y,2,(m,\rho)}=R_{s,t}^{Y,(m,\rho)}=\rho\hat{\epsilon}_{s,t}^{(m)}+(1-\rho)\epsilon_{s,t}\label{Rymrho_small}.
\end{align}
Hence, we can set constants $C_0^{(1)},C_1^{(1)}$ such that  $C_0^{(m,1)}=C_0^{(1)},C_1^{(m,1)}=C_1^{(1)}$.

Next, suppose that we have already set $C_0^{(m,k)}$ and $C_1^{(m,k)}$ for fixed $k$.
Take $s,t\in \mathbb{P}_m$ such that $t-s=(k+1)2^{-m}$ arbitrary.
In order to estimate $R_{s,t}^{Y,2,(m,\rho)}$, we take $u$ as the largest discrete time in $\mathbb{P}_m$ such that $|u-s|\leq (t-s)/2$. Also, we denote $u+2^{-m}$ by $v$. Now, we decompose $R_{s,t}^{Y,2,(m,\rho)}$ as

\begin{equation*}
R_{s,t}^{Y,2,(m,\rho)}=R_{s,u}^{Y,2,(m,\rho)}+R_{u,v}^{Y,2,(m,\rho)}+R_{v,t}^{Y,2,(m,\rho)}-\delta R_{s,u,t}^{Y,2,(m,\rho)}-\delta R_{u,v,t}^{Y,2,(m,\rho)}.
\end{equation*}
We estimate each term on the right-hand side.

\begin{itemize}
\item For $R_{s,u}^{Y,2,(m,\rho)}$, by assumption, we have $|R_{s,u}^{Y,2,(m,\rho)}|\leq C_1^{(m,k)}(u-s)^{(q+1)H_-}$.
Since $(u-s)\leq (t-s)/2$, we obtain $|R_{s,u}^{Y,2,(m,\rho)}|\leq 2^{(q+1)H_-}C_1^{(m,k)}(t-s)^{(q+1)H_-}$.
\item For $R_{u,v}^{Y,2,(m,\rho)}$, since $|v-u|=2^{-m}$, $|R_{u,v}^{Y,2,(m,\rho)}|\leq C_0^{(1)}2^{-m(q+1)H_-}$.
\item For $R_{v,t}^{Y,2,(m,\rho)}$, such as $R_{s,u}^{Y,2,(m,\rho)}$, $|R_{v,t}^{Y,2,(m,\rho)}|\leq 2^{(q+1)H_-}C_1^{(m,k)}(t-s)^{(q+1)H_-}$.
\item For $\delta R_{s,u,t}^{Y,2,(m,\rho)}$, by definition, we have $\delta R_{s,u,t}^{Y,2,(m,\rho)}=-\delta \Xi_{s,u,t}^{Y,(m,\rho)}-\delta R_{s,u,t}^{Y,1,(m,\rho)}$.

From Corollary \ref{corollary:BdiscreteHoelder}, $\|R_\cdot^{Y,1,(m,\rho)}\|_{qH_-}\leq \|\overline{f}_q\|_\infty\|\mathcal{I}_\cdot^{(m)}(d(B^{(q)}))\|_{qH_-}\lesssim_m 1$.
Hence, we set $C_3$ such that $\|R_\cdot^{Y,1,(m,\rho)}\|_{qH_-}<C_3$.

Also, from Proposition \ref{proposition:discrete_sewing_for_rough_path}, we set $C_7$ such as
\begin{equation*}
|\delta\Xi_{s,u,t}^{Y,(m,\rho)}|\leq C_7((t-u)^{(q+1)H_-}+P_q(|R_{s,u}^{Y,(m,\rho)}|,(t-s)^{H_-})).
\end{equation*}

For $R_{s,u}^{Y,(m,\rho)}$, we have
\begin{equation*}
|R_{s,u}^{Y,(m,\rho)}|\leq |R_{s,u}^{Y,1,(m,\rho)}|+|R_{s,u}^{Y,1,(m,\rho)}|\leq C_3(u-s)^{qH_-}+C_1^{(m,k)}(u-s)^{(q+1)H_-}.
\end{equation*}
Naturally, we can estimate $\delta\Xi_{s,u,t}^{Y,(m,\rho)}$ again as follows:
\begin{align*}
|\delta\Xi_{s,u,t}^{Y,(m,\rho)}|\leq& C_7((t-s)^{(q+1)H_-}+(t-s)^{H_-}\\
&P_q(C_3(u-s)^{qH_-}+C_1^{(m,k)}(u-s)^{(q+1)H_-},(t-s)^{H_-})).
\end{align*}
So, we define $C_8^{(m,k+1)}$ as
\begin{eqnarray*}
&C_7(1+((k+1)\cdot 2^{-m})^{-qH_-}\\
&\cdot P_q(((k+1)\cdot 2^{-m})^{qH_-}C_3+((k+1)\cdot 2^{-m})^{(q+1)H_-}C_1^{(m,k)},((k+1)\cdot 2^{-m})^{H_-})),
\end{eqnarray*}
so that $|\delta R_{s,u,t}^{Y,2,(m,\rho)}|\leq C_8^{(m,k+1)}(t-s)^{(q+1)H_-}$.
\item For $\delta R_{u,v,t}^{Y,2,(m,\rho)}$, by definition, we have $\delta R_{u,v,t}^{Y,2,(m,\rho)}=-\delta \Xi_{u,v,t}^{Y,(m,\rho)}-\delta R_{u,v,t}^{Y,1,(m,\rho)}$.
Combining Proposition \ref{proposition:discrete_sewing_for_rough_path} and (\ref{Rymrho_small}), we can set $C_6$ such that
\begin{equation*}
|\delta \Xi_{u,v,t}^{Y,(m,\rho)}|\leq C_a(t-u)^{(q+1)H_-}+C_bP_q(|R_{u,v}^{Y,(m,\rho)}|,(t-v)^{H_-})\leq C_6(t-u)^{(q+1)H_-}.
\end{equation*}
\end{itemize}

Now, we can estimate $R_{s,t}^{Y,2,(m,\rho)}$ as follows:
\begin{align*}
|R_{s,t}^{Y,2,(m,\rho)}|=&|R_{s,u}^{Y,2,(m,\rho)}|+|R_{u,v}^{Y,2,(m,\rho)}|+|R_{v,t}^{Y,2,(m,\rho)}|+|\delta R_{s,u,t}^{Y,2,(m,\rho)}|+|\delta R_{u,v,t}^{Y,2,(m,\rho)}|\\
\leq& (2^{-((q+1)H_--1)}C_1^{(m,k)}+C_8^{(m,k)}+C_0^{(1)}+C_6)(t-s)^{(q+1)H_-}.
\end{align*}

On the other hand, by a priori estimate, we set $C_2$ such that $\|\Xi^{Y,(m,\rho)}\|_{H_-}\leq C_2$.
Then, for $Y_{s,t}^{(m,\rho)}$, we have
\begin{align*}
Y_{s,t}^{(m,\rho)}=&\Xi_{s,t}^{Y,(m,\rho)}+R_{s,t}^{Y,(m,1)}+R_{s,t}^{Y,2,(m,\rho)},\\
|Y_{s,t}^{(m,\rho)}|\leq& C_2(t-s)^{H_-}+C_3(t-s)^{qH_-}+C_1^{(m,k+1)}(t-s)^{(q+1)H_-}.
\end{align*}

Next, we define $C_0^{(m,k+1)}$ and $C_1^{(m,k+1)}$:
\begin{eqnarray*}
C_0^{(m,k+1)}\coloneqq&C_2+C_3((k+1)2^{-m})^{(q-1)H_-}+C_1^{(m,k+1)}((k+1)2^{-m})^{qH_-},\\
C_1^{(m,k+1)}\coloneqq&2^{-((q+1)H_--1)}C_1^{(m,k)}+C_8^{(m,k)}+C_0^{(1)}+C_6.
\end{eqnarray*}

Using the estimates for $R_{s,t}^{Y,2,(m,\rho)}$ and $Y_{s,t}^{(m,\rho)}$, we have (\ref{short_time_estimate_local}) for any $s,t$ such that $t-s\leq (k+1)2^{-m}$.

Subsequently, we can take $C_0^{(m,k)}$ and $C_1^{(m,k)}$ such that (\ref{short_time_estimate_local}) for any $m$ and $k$.

Secondly, we aim to discover constants $C_0$, $C_1$ and $\delta$, independent of $m$ and $\rho$, such that $C_0^{(m,k)}\leq C_0$ and $C_1^{(m,k)}\leq C_1$ for any $m,k$ such that $k\cdot 2^{-m}< \delta$.

We set the constants $C_0,C_1,C_8$ as follows:
\begin{align*}
C_8&\coloneqq C_7\left(1+P_q(C_3+2,1)\right),&
C_1&\coloneqq\left(\frac{C_0^{(1)}+C_6+C_8}{1-2^{(q+1)H_--1}}\right)\vee C_1^{(1)},\\
C_0&\coloneqq (C_2+C_3+C_1)\vee C_0^{(1)}.
\end{align*}
Also, we choose the constant $\delta\in (0,1)$ such that $C_1\delta^{H_-}<1,C_4\delta^{qH_-}<\frac{1}{2}$.
When $k=1$, $C_1^{(m,1)}\leq C_1,C_0^{(m,1)}\leq C_0$ by definition.

Next, for some $k$ such that $(k+1)\cdot 2^{-m}<\delta$, we suppose that $C_1^{(m,j)}\leq C_1,C_0^{(m,j)}\leq C_0$ for any $1\leq j\leq k$. Then,
\begin{align*}
C_8^{(m,k+1)}=&C_7(1+((k+1)\cdot 2^{-m})^{-qH_-}P_q(((k+1)\cdot 2^{-m})^{qH_-}C_3\\
&+((k+1)\cdot 2^{-m})^{(q+1)H_-}C_1^{(m,k)},((k+1)\cdot 2^{-m})^{H_-}))\\
\leq& C_7(1+((k+1)\cdot 2^{-m})^{-qH_-}\\
&P_q(((k+1)\cdot 2^{-m})^{qH_-}C_3+((k+1)\cdot 2^{-m})^{(q+1)H_-}C_1,1))\\
\leq& C_7(1+((k+1)\cdot 2^{-m})^{-qH_-}P_q(((k+1)\cdot 2^{-m})^{qH_-}(C_3+1),1))\\
\leq& C_7(1+P_q(C_3+1,1))=C_8.
\end{align*}
Also,
\begin{align*}
C_1^{(m,k+1)}=&2^{-((q+1)H_--1)}C_1^{(m,k)}+C_8^{(m,k)}+C_0^{(1)}+C_6\\
\leq& 2^{-((q+1)H_--1)}C_1+C_8^{(m,k)}+C_0^{(1)}+C_6\\
\leq& 2^{-((q+1)H_--1)}C_1+C_8+C_0^{(1)}+C_6\leq C_1.
\end{align*}
Moreover,
\begin{align*}
C_0^{(m,k+1)}=&C_2+C_3((k+1)2^{-m})^{(q-1)H_-}+C_1^{(m,k+1)}(((k+1)2^{-m})^{qH_-}.\\
\leq& C_2+C_3+C_1\leq C_0.
\end{align*}

By the induction, we can conclude $C_0^{(m,k)}\leq C_0,C_1^{(m,k)}\leq C_1$ if $k\cdot 2^{-m}<\delta$.

Hence, we conclude (\ref{short_time_estimate}) for any $s,t$ such that $|t-s|<\delta$.
\end{proof}

Secondly, we get the long-time range estimate from the short one.

\begin{lmm}\label{lemma:partition_remainder}

Let $\mathscr{P}=\{t_0,\cdots,t_n\}$ be the partition on $\mathbb{P}_m$.
Suppose that for any $1\leq i\leq n-1$,
\begin{align*}
|Y_{t_{i+1},t_i}^{(m,\rho)}|\leq C_{10}(t_{i+1}-t_i)^{H_-},&&|R_{t_i,t_{i+1}}^{Y,2,(m,\rho)}|\leq C_{11}(t_{i+1}-t_i)^{(q+1)H_-}
\end{align*}
with the constant $C_{10},C_{11}$.

Then, the following inequality holds:

\begin{align*}
|Y_{t_0,t_n}^{(m,\rho)}|\leq nC_{10}(t_n-t_0)^{(q+1)H_-},&&|R_{t_0,t_n}^{Y,2,(m,\rho)}|\leq (C_{11}+(n-1)C_{12})(t_n-t_0)^{(q+1)H_-},
\end{align*}
where $C_{12}$ is a constant depends on $C_{10}$ and $C_{11}$.
\end{lmm}

\begin{proof}

The proof is based on induction.

At first, the conclusion equals the assumption in the $n=1$ case.

Next, we assume the conclusion holds in the $n=k$ case. Let $\mathscr{P}=\{t_0,\cdots,t_{k+1}\}$. In this case, we can see

\begin{align*}
Y_{t_0,t_n}^{(m,\rho)}=&Y_{t_0,t_1}^{(m,\rho)}+Y_{t_1,t_n}^{(m,\rho)},\\
R_{t_0,t_n}^{Y,2,(m,\rho)}=&R_{t_0,t_1}^{Y,2,(m,\rho)}+R_{t_1,t_n}^{Y,2,(m,\rho)}+\delta R_{t_0,t_1,t_n}^{Y,2,(m,\rho)}.
\end{align*}
Since $\delta R_{t_0,t_1,t_n}^{Y,2,(m,\rho)}=-\delta \Xi_{t_0,t_1,t_n}^{Y,2,(m,\rho)}$, we can set a constant $C_{12}$ depends on $C_{10}$ and $C_{11}$ such that $|\delta R_{t_0,t_1,t_n}^{Y,2,(m,\rho)}|\leq C_{12}(t_{k+1}-t_0)^{(q+1)H_-}$ from Proposition \ref{proposition:discrete_sewing_for_rough_path}.
Now, $\mathscr{P}'=\{t_1,\cdots,t_{k+1}\}$ is the $k$-partition of the interval $[t_1,t_{k+1}]$. Then, by the assumption of induction, we can show the following estimates:

\begin{align*}
|Y_{t_0,t_n}^{(m,\rho)}|\leq& |Y_{t_0,t_1}^{(m,\rho)}|+|Y_{t_1,t_{k+1}}^{(m,\rho)}|\\
\leq& C_{10}(t_1-t_0)^{H_-}+kC_{10}(t_{k+1}-t_1)^{H_-}\leq (k+1)C_{10}(t_{k+1}-t_0)^{H_-},
\end{align*}

\begin{align*}
|R_{t_0,t_n}^{Y,2,(m,\rho)}|\leq& |R_{t_0,t_1}^{Y,2,(m,\rho)}|+|R_{t_1,t_n}^{Y,2,(m,\rho)}|+|\delta R_{t_0,t_1,t_n}^{Y,2,(m,\rho)}|\\
\leq& C_{11}(t_1-t_0)^{(q+1)H_-}+(C_{11}+(k-1)C_{12})(t_{k+1}-t_1)^{(q+1)H_-}\\
&+C_{12}(t_{k+1}-t_0)^{(q+1)H_-}\\
\leq& (C_{11}+kC_{12})(t_{k+1}-t_0)^{(q+1)H_-}.
\end{align*}
In other words, the conclusion holds in the $n=k+1$ case.

By induction, we complete the proof.
\end{proof}

Combining Lemma \ref{lemma:expansion_small_distance} and \ref{lemma:partition_remainder}, we get long-range estimate of H\"older norm of $\polsolstd{}$ and $R^{Y,2(m,\rho)}$ in (\ref{dcp_def}).

\begin{lmm}\label{lemma:Hoelder_of_y}
Let $\sigma\in C_b^{q+1},b\in C_b^2$.
Then,
\begin{align}
\|\polsolstd{\cdot}\|_{H_-}\lesssim_{m,\rho} 1,&&\|R_\cdot^{Y,2,(m,\rho)}\|_{(q+1)H_-}\lesssim_{m,\rho} 1.\label{Ymrho_Holder}
\end{align}

\end{lmm}

\begin{proof}

By Lemma \ref{lemma:expansion_small_distance}, there exists $\delta>0$ such that for all $|t-s|<\delta$,

\begin{align*}
|\polsolstd{s,t}|\leq C_0(t-s)^{H_-},&&|R_{s,t}^{Y,2,(m,\rho)}|\leq C_1(t-s)^{(q+1)H_-}
\end{align*}
with constants $C_0,C_1$.
Then, for any $s,t\in\mathbb{P}_m$ we choose a partition $\mathscr{P}_{s,t}$ as follows:
\begin{itemize}
\item If $2^{-m}\leq \delta$, we choose $\mathscr{P}_{s,t}$ of $[s,t]$ on $\mathbb{P}_m$ such that $|\mathscr{P}_{s,t}|<\delta$ and $\#\mathscr{P}_{s,t}\leq 2\lfloor (t-s)/\delta\rfloor$.

Using Lemma \ref{lemma:partition_remainder}, we can show that there exists $C_9>0$ such that

\begin{align*}
|Y_{s,t}^{(m,\rho)}|\leq 2\lfloor (t-s)/\delta\rfloor C_0(t-s)^{H_-},&&|R_{s,t}^{Y,2,(m,\rho)}|\leq 2\lfloor (t-s)/\delta\rfloor C_9(t-s)^{(q+1)H_-}.
\end{align*}
As a consequence, we have

\begin{align}
\|\polsolstd{\cdot}\|_{H_-}\leq 2 \lfloor T/\delta\rfloor C_1&\|R_\cdot^{Y,2,(m,\rho)}\|_{(q+1)H_-}\leq 2 \lfloor T/\delta\rfloor C_9\label{long_range_for_small_m}
\end{align}
\item If $2^{-m}>\delta$, we set $\mathscr{P}_{s,t}$ to be $\mathbb{P}_m\cap[s,t]$. 
Using Lemma \ref{lemma:partition_remainder}, we can estimate $\polsolstd{}$ and $R^{Y,2,(m,\rho)}$ for each $m$ similar to the case $2^{-m}\leq \delta$.
Remark that such estimate is independent of $\rho$.
\end{itemize}
When $2^{-m}\leq \delta$, the right-hand side of (\ref{long_range_for_small_m}) is independent of m and $\rho$.
Also, when $2^{-m}>\delta$, there are only a finite number of such $m$ and the estimate can be independent of $\rho$.

Hence, we get (\ref{Ymrho_Holder}).
\end{proof}

This lemma means that $\polsolstd{}$ satisfies (\ref{dcp_def}) for $k=1$ case.
Now, we can get (\ref{dcp_def}) for all $k$.

\begin{lmm}\label{lmm:Ymrho_bounded}

Let $\sigma \in C_b^{q+2}, b\in C_b^2$.

Then, $(\polsolstd{},\{\mathcal{D}^{k-1}\sigma(\polsolstd{})\}_{k=1}^{q-1})$ is discrete controlled path such that the remainder part defined in (\ref{dcp_def}) is bounded with respect to $m$ and $\rho$.
Additionally, for $f\in C^{q+1}$, $(\{\mathcal{D}^{k-1}f(\polsolstd{})\}_{k=1}^q)$ is also discrete controlled path bounded with respect to $m$ and $\rho$.
\end{lmm}

\begin{proof}

For Lipschitz continuous path $w$ and $\sigma\in C_b^{q+2}$, let $y_t$ is solution of $dt=\sigma(y_t)dw_t;y_0:$fixed.
Then, for $0<s<t$ and $f\in C^{q+1}$, $f(y_t)=f(y_s)+\int_s^t f'\sigma(y_t)dw_t$, which is a degenerated case of Ito formula.
Therefore, we can expand $f(y_t)$ as 
\[f(y_t)=f(y_s)+\sum_{i=1}^p \frac{1}{i!}\mathcal{D}^if(y_s)\delta w_{s,t}^k+R_{s,t},\]
where $|R_{s,t}|\lesssim (\|w\|_{Lip}(t-s))^{p+1}$.

On the other hand, substituting $f$ for $f\equiv 1$, we have extension:
\[y_t=y_s+\sum_{i=1}^p\frac{1}{i!} \mathcal{D}^{i-1}\sigma(y_s)\delta w_{s,t}^i+R_{s,t}.\]
Hence, we substitute $z_2$ for $y_t$ and $z_1$ for $y_s$ in the Taylor expansion:
\[f(z_2)=f(z_1)+\sum_{i=1}^p \frac{1}{i!}f^{(i)}(z_1)(z_2-z_1)^i+R_{z_1,z_2}^*\]
Now, we can get explicit functions $\{g_i\}_{i=1}^p$ such that
\[f(y_t)=f(y_s)+\sum_{i=1}^p g_i(y_s)\delta w_{s,t}^i+R_{s,t}^{**},\]
where $|R_{s,t}^{**}|\lesssim (\|w\|_{Lip}(t-s))^{p+1}$.
Moving $t$ closer to $s$, we obtain $g_i(y_s)=\frac{1}{i!}\mathcal{D}^{i-1}f(y_s)$ by uniqueness of power series.

Substituting $f\equiv \mathcal{D}^{k-2}\sigma$ for $2\leq k\leq q$ and replacing $\delta x_{s,t}$ with linear approximation of it, we obtain (\ref{dcp_def}) for each $k$.
Immediately,  $(\polsolstd{},\{\mathcal{D}^{k-1}\sigma(\polsolstd{})\}_{k=1}^{q-1})$ is discrete controlled path.
For $(\{\mathcal{D}^{k-1}f(\polsolstd{})\}_{k=1}^q)$, we can obtain substituting $\mathcal{D}^kf$ for $f$.

\end{proof}

Next, we show that $\|\polJacobistd{\cdot}\|_{H_-}\lesssim_{m,\rho} 1$. To apply the technique of Lemma \ref{lemma:Hoelder_of_y}, we have to prove the boundedness of $M$ at first.

\begin{lmm}\label{lemma:estimate_log_remainder}
Given $\sigma\in C_b^{q+2},b\in C_b^2$, we define $R_{s,t}(x)$ as
\begin{equation}
R_{s,t}(x)=\log (1+\Xi'[\sigma,x,s,t]+b'(x)(t-s))-\Xi[\sigma',x,s,t]-b'(x)(t-s).,\label{definitionRexp}
\end{equation}
Then, $\|R_{s,t}\|_{C^1}\lesssim (t-s)^{(q+1)H_-}$.
\end{lmm}

\begin{proof}

At first, we set $\tilde{\sigma},\tilde{b}\in C_b(\mathbb{R}^2;\mathbb{R})$ such that
\begin{align*}
\tilde{\sigma}(x_1,x_2)=(\sigma(x_1),\sigma'(x_1)x_2),&&\tilde{b}(x_1,x_2)=(b(x_1),b'(x_1)x_2)
\end{align*}
for $x_1\in \mathbb{R}$ and $x_2\in [-\sup_{m,\rho}|\polJacobistd{}|,\sup_{m,\rho}|\polJacobistd{}|]$.

Recall that $J$ is defined as (\ref{definitionJ}). Then, from Proposition \ref{proposition:Itoformula}, we can see that $(y,J)$ satisfies
\begin{equation*}
d(y,J)=\tilde{\sigma}(y,J)dB_t+\tilde{b}(y,J)dt.
\end{equation*}
Hence, we have
\begin{equation}
(\exsol{t},\Jacobi{t})=(\exsol{s},J_s)+\sum_{k=1}^q \frac{1}{k!}\mathcal{D}_{\tilde{\sigma}}^{k-1}\tilde{\sigma}(\exsol{s},J_s)\fBminc{s,t}^k+\tilde{b}(\exsol{s},J_s)(t-s)+O(|t-s|^{(q+1)H_-})\label{yJexpand}.
\end{equation}

We recall that $\mathcal{D}_{\tilde{\sigma}}=\sigma(x_1)\partial_1+x_2\sigma'(x_1)\partial_2$.

We can check that $\mathcal{D}_{\tilde{\sigma}}(f'(x_1)x_2)=\sigma f''(x_1)x_2+\sigma'f(x_1)x_2=(\mathcal{D}f)'(x_1)x_2$, so
\begin{equation*}
\mathcal{D}_{\tilde{\sigma}}^{k}\tilde{\sigma}(x_1,x_2)=(\mathcal{D}^{k-1}\sigma(x_1),(\mathcal{D}^{k-1}\sigma)'(x_1)x_2).
\end{equation*}

Substituting this equation to (\ref{yJexpand}), we can easily see that
\begin{equation}
\Jacobi{t}(J_s)^{-1}=\sum_{k=1}^q \frac{1}{k!}(\mathcal{D}^{(k-1)}\sigma)'(\exsol{s})\fBminc{s,t}^k+O(|t-s|)\label{expansion_of_J}.
\end{equation}

For $s,t$ such that $|J_{s,t}-1|<1/2$, by Taylor expansion,
\begin{equation*}
\log \Jacobi{t}-\log J_s=\sum_{k=1}^q \frac{1}{k!}g_k(\exsol{s})\fBminc{s,t}^k+O(|t-s|),
\end{equation*}
where
\begin{equation*}
g_k\coloneqq k!\sum_{j=1}^k\sum_{\begin{smallmatrix}
i_1+\cdots+i_j=k,\\
1\leq i_l\leq q.
\end{smallmatrix}}\frac{(-1)^j}{j!}\prod_{l=1}^j (\mathcal{D}^{(i_l-1)}\sigma)'.
\end{equation*}
In other words, when we consider $\log \Jacobi{t}$ to be a rough path, we get $(\log \Jacobi{t})^{(k+1)}=g_k(\exsol{t})$ for $1\leq k\leq q-1$.

On the other hand, we can see that
\begin{equation*}
\log \Jacobi{t}-\log J_s=\int_s^t \sigma'(Y_u)dB_u+\int_s^t b'(Y_u)du=\sum_{k=1}^q\frac{1}{k!} (\mathcal{D}^{(k-1)}\sigma')(\exsol{s})\fBminc{s,t}^k+O(|t-s|).
\end{equation*}
Consequently, $(\log \Jacobi{t})^{(k+1)}=(\mathcal{D}^{(k-1)}\sigma')(\exsol{t})$ for $1\leq k\leq q-1$.

Comparing these two determinations, we have $g_l\equiv (\mathcal{D}^{(l-1)}\sigma')$.

Hence, we can transform $R_{s,t}(x)$ as follows:
\begin{equation}
R_{s,t}(x)=\sum_{k_1=q+1}^{q^2}\sum_{k_2=1}^q\tilde{g}_{k_1,k_2}(x)\fBmitin{s,t}^{k_1}(t-s)^{k_2}+h(\Xi'[\sigma,x,s,t]+b'(x)(t-s))
\end{equation}
where $\tilde{g}_{k_1,k_2}$ is $C_b^1$-class function determined by the Taylor expansion of $h$:

\begin{equation*}
h(x)\coloneqq \left(\log(1+x)-\sum_{k=1}^q\frac{(-1)^{k+1}}{k}x^k\right).
\end{equation*}

From Taylor's theorem, we get $x^{-(q+1)}h(x)$ is $C_b^\infty$-class function on $[-1/2,1/2]$.
At the same time, since $\sigma$ is $C_b^{q+2}$-class function, $\Xi[\sigma',x,s,t]$ is $C_b^2$-class.
So, $R_{s,t}\in C_b^2$ holds for sufficiently large $m$.
Moreover, we can estimate that $\|R_{s,t}\|_{C^1}\lesssim (t-s)^{(q+1)H_-}$ by estimating each term.
\end{proof}

\begin{cor}\label{corollary:M-1J}
Let $\sigma\in C_b^{q+2},b\in C_b^2$.

Then, $\|\Jacobi{\cdot}(\polJacobi{\cdot}{0})^{-1}\|_1\lesssim 2^{-m((q+1)H_--1)}, \|M_\cdot^{(m,0)}J_\cdot^{-1}\|_1\lesssim 2^{-m((q+1)H_--1)}$.
\end{cor}

\begin{proof}
To avoid repetition, we will prove the former.
By definition,

\begin{align*}
\Jacobi{t}(M_t^{(m,0)})^{-1}=&\prod_{r=0}^{ \lfloor 2^m t\rfloor-1 } J_{\bitime{r+1}}J_{\bitime{r}}^{-1}(M_{\bitime{r+1}}M_{\bitime{r}}^{-1})^{-1}\\
=&\prod_{r=0}^{ \lfloor 2^m t\rfloor-1 } \exp(\Xi[\sigma',\exsoldis{r},\bitime{r},\bitime{r+1}]+b'(\exsoldis{r})2^{-m}+O(2^{-m(q+1)H_-}))\\
&(1+\Xi'[\sigma,\exsoldis{r},\bitime{r},\bitime{r+1}]+b'(\exsoldis{r})2^{-m})^{-1}\\
=&\prod_{r=0}^{ \lfloor 2^m t\rfloor-1 } \exp(-R_{\bitime{r},\bitime{r+1}}(\exsoldis{r})+O(2^{-m(q+1)H_-}))\\
=& \exp\left(-\left(\sum_{r=0}^{ \lfloor 2^m t\rfloor-1 } (R_{\bitime{r},\bitime{r+1}}(\exsoldis{r})+O(2^{-m(q+1)H_-}))\right)\right).
\end{align*}

As a result, we can estimate as
\begin{align*}
&\Jacobi{t}(M_t^{(m,0)})^{-1}-J_s(M_s^{(m,0)})^{-1}\\
=&\left(\exp\left(-\left(\sum_{r= \lfloor 2^ms\rfloor}^{ \lfloor 2^m t\rfloor-1 } (R_{\bitime{r},\bitime{r+1}}(\exsoldis{r})+O(2^{-m(q+1)H_-}))\right)\right)-1\right)\\
&\exp\left(-\left(\sum_{r=0}^{ \lfloor 2^m s\rfloor-1 } (R_{\bitime{r},\bitime{r+1}}(\exsoldis{r})+O(2^{-m(q+1)H_-}))\right)\right).
\end{align*}

From Lemma \ref{lemma:estimate_log_remainder}, we get the conclusion.
\end{proof}

\begin{cor}\label{corollary:Minfinity}

Let $\sigma\in C_b^{q+2},b\in C_b^2$. Then, $\|\polJacobistd{\cdot}\|_\infty \lesssim_{m,\rho}1 $ and $\|(\polJacobistd{\cdot})^{-1}\|_\infty\lesssim_{m,\rho}1$.

\end{cor}

\begin{proof}

To avoid repetition, we prove only former inequality.
By definition of $R_{s,t}(x)$ (\ref{definitionRexp}), we obtain
\begin{equation*}
\polJacobistd{t}=\exp\left(\mathcal{I}_t^{(m)}(d(\Xi[\sigma',\polsolstd{},\cdot_1,\cdot_2]+b'(Y_{\cdot_1}^{(m,\rho)})(\cdot_2-\cdot_1)+R_{\cdot_1,\cdot_2}(Y_{\cdot_1}^{(m,\rho)})))\right).
\end{equation*}
From Lemma \ref{lmm:Ymrho_bounded}, we have $Y^{(m,\rho)}$ is discrete controlled path bounded with respect to $m$ and $\rho$, and we also have
\begin{equation*}
\mathcal{I}_t^{(m)}(d(\Xi[\sigma',\polsolstd{},\cdot_1,\cdot_2]+b'(Y_{\cdot_1}^{(m,\rho)})(\cdot_2-\cdot_1)))
\end{equation*}
is bounded in the same sense.
From Lemma \ref{lemma:estimate_log_remainder} and an estimate of Stieltjes sum, $\mathcal{I}_t^{(m)}(d(R_{\cdot_1,\cdot_2}(Y_{\cdot_1}^{(m,\rho)})))$ is also bounded.
Combining these, we obtain the former inequality.
For the latter, we can obtain as same as former.
\end{proof}

\begin{lmm}\label{lemma:Hoelder_of_M}
Let $\sigma\in C_b^{q+2},b\in C_b^2$. Then, $\|\polJacobistd{\cdot}\|_{H_-}\lesssim_{m,\rho}1$. Moreover, $(\polsolstd{},\polJacobistd{})$ is discrete controlled path.

Additionally, $(\{\mathcal{D}_{\tilde{\sigma}}^{k-1} f(\polsolstd{},\polJacobistd{})\}_{k=1}^q)$ is discrete controlled path for $f\in C^{q+1}$.
\end{lmm}

\begin{proof}

By definition of $Y^{(m,\rho)}$ and $M^{(m,\rho)}$, for $\timeinunit $, we have

\begin{align*}
(\polsolstd{t},\polJacobistd{t})=&(\polsolstd{\bitime{r}},M_{\bitime{r}}^{(m,\rho)})+\sum_{k=1}^q \frac{1}{k!}\mathcal{D}_{\tilde{\sigma}}^{k-1}\tilde{\sigma}(\polsolstd{\bitime{r}},M_{\bitime{r}}^{(m,\rho)})\fBmunit^k\\
&+\tilde{b}(\polsolstd{\bitime{r}},M_{\bitime{r}}^{(m,\rho)})\timeunit+(\overline{f}_q(\nusoldis{r})\fBmunit^{q},0)\\
&+O(\timeunit^{(q+1)H_-}).
\end{align*}

$\polJacobistd{}$ is pathwisely bounded with respect to $m$ and $\rho$.
Hence, we can cut off $\tilde{\sigma},\tilde{b}$ such that they become bounded and coincides to $\tilde{\sigma},\tilde{b}$ in the range of $(\polsolstd{},\polJacobistd{})$.
Hence, we can apply the same discussion in Lemma \ref{lemma:Hoelder_of_y} and \ref{lmm:Ymrho_bounded} replacing $R_{s,t}^{Y,m,\rho,i}$ with $(R_{s,t}^{Y,m,\rho,i},0)$ for $i = 1,2,3$. 

Now, we conclude that $\|\polJacobistd{\cdot}\|_{H_-}\lesssim_{m,\rho}1$.
\end{proof}

Next, we obtain the elementary estimate of $Z^{(m,\rho)}$.

\begin{lmm}\label{lemma:apriori_estimate}
Let $\sigma\in C_b^{q+3},b\in C_b^2$.

Then,  $Z^{(m,\rho)}$ be estimated as follows:
\begin{enumerate}
\item Under Condition ($\mathfrak{A}$), $\|Z_\cdot^{(m,\rho)}\|_{1}\lesssim_{\rho} 2^{-((q+1)H_--1)m}$.
\item Under Condition ($\mathfrak{B}$), $\|Z_\cdot^{(m,\rho)}\|_{H^c}\lesssim_{\rho} 2^{-(qH_--H^c)m}$.
\end{enumerate}
Especially, $\|\nusol{}-Y\|_1\lesssim_\rho 2^{-((q+1)H_--1)m}$(resp.$\|\nusol{}-Y\|_{H^c}\lesssim_\rho 2^{-(qH_--H^c)m}$) under Condition ($\mathfrak{A}$)\ (resp.Condition ($\mathfrak{B}$)).

\end{lmm}
\begin{proof}

From Lemma \ref{lemma:Hoelder_of_M}, $\|\polJacobistd{\cdot}\|_{H_-}\lesssim_{m,\rho} 1$ holds.

Next, by definition of $e$ (Definition \ref{def:e}) and Corollary \ref{corollary:BdiscreteHoelder}, $\|e_\cdot^{(m)}\|_{1}\lesssim 2^{-((q+1)H_--1)m}$ on the Condition ($\mathfrak{A}$) and $\|e_\cdot^{(m)}\|_{H^c}\lesssim 2^{-(qH_--H^c)m}$ on the Condition ($\mathfrak{B}$).
Since $Z^{(m,\rho)}=\mathcal{I}^{(m)} ((\polJacobistd{})^{-1},d(\epsilon^{(m)}))$, we get the conclusion by the estimate of Young sum.
\end{proof}
Next we obtain estimates of $N^{(m,\rho)}$.

\begin{lmm}\label{lemma:Hoelder_of_N}
Let $\sigma\in C_b^{q+3},b\in C_b^2$.
Also, we define

\begin{equation*}
\overline{N}_\cdot^{(m,\rho)}\coloneqq \mathcal{I}_\cdot^{(m)}\left(\polJacobistd{},d(\Xi'[\sigma',\polsolstd{},\cdot_1,\cdot_2]+b''(Y_{\cdot_1}^{(m,\rho)})(\cdot_2-\cdot_1)\right).
\end{equation*}
Then, it holds that
\begin{equation}
N_t^{(m,\rho)}=\mathcal{I}_t^{(m)}(Z^{(m,\rho)},d(\overline{N}_{\cdot_1,\cdot_2}^{(m,\rho)}+R_{\cdot_1,\cdot_2}'(Y_{\cdot_1}^{(m,\rho)})).\label{integral_form_N}
\end{equation}
Also, $\|\overline{N}_\cdot^{(m,\rho)}\|_{H_-}\lesssim_{m,\rho}1$. 
Moreover, $ \|N_\cdot^{(m,\rho)}\|_{H_-}\lesssim \|Z_\cdot^{(m,\rho)}\|_{H^c}\lesssim_{m,\rho} 1$.
\end{lmm}

\begin{proof}
First, we show (\ref{integral_form_N}).
By definition of $R_{s,t}(x)$ (\ref{definitionRexp}), we obtain

\begin{equation*}
\polJacobistd{t}=\exp \left(\mathcal{I}_t^{(m)}\left(d(\Xi[\sigma',\polsolstd{},\cdot_1,\cdot_2]+b'(Y_{\cdot_1}^{(m,\rho)})(\cdot_2-\cdot_1)+R_{\cdot_1,\cdot_2}(Y_{\cdot_1}^{(m,\rho)}))\right)\right).
\end{equation*}
Then, from $N_t^{(m,\rho)} =(\polJacobistd{t})^{-1} \partial_\rho \polJacobistd{t}$, we can see
\begin{align*}
N_t^{(m,\rho)}=&(\polJacobistd{t})^{-1} \partial_\rho \polJacobistd{t}\\
=&\mathcal{I}_t^{(m)}\left(\partial_\rho \polsolstd{},d(\Xi'[\sigma',\polsolstd{},\cdot_1,\cdot_2]+b''(\polsolstd{})(\cdot_2-\cdot_1)+R_{\cdot_1,\cdot_2}'(Y_{\cdot_1}^{(m,\rho)}))\right)\\
=&\mathcal{I}_t^{(m)}(Z^{(m,\rho)},d(\overline{N}_{\cdot_1,\cdot_2}^{(m,\rho)}+M_{\cdot_1}^{(m,\rho)}R_{\cdot_1,\cdot_2}'(Y_{\cdot_1}^{(m,\rho)})).
\end{align*}

Next,  we show that $\|\overline{N}_\cdot^{(m,\rho)}\|_{H_-}\lesssim_{m,\rho}1$.
We introduce $\tilde{N}$:
\begin{equation*}
\tilde{N}\coloneqq (\sigma''(\polsolstd{})\polJacobistd{},\cdots,(\mathcal{D}^{(q-1)}\sigma')'(\polsolstd{})\polJacobistd{}).
\end{equation*}
To estimate $\overline{N}^{(m,\rho)}$, we prove $\tilde{N}$ is a discrete controlled path. 
As shown in Lemma \ref{lemma:Hoelder_of_M}, $(\{\mathcal{D}_{\tilde{\sigma}}^{k-1} f(\polsolstd{},\polJacobistd{})\}_{k=1}^q)$ is a discrete controlled path for sufficiently smooth $f$.
Since $\mathcal{D}_{\tilde{\sigma}}(f'(x_1)x_2)=(\mathcal{D}f)'(x_1)x_2$, $\tilde{N}$ is discrete controlled path. As a result, we get $\|\overline{N}_\cdot^{(m,\rho)}\|_{H_-}\lesssim_{m,\rho}1$.

Next, we show that $ \|N_\cdot^{(m,\rho)}\|_{H_-}\lesssim \|Z_\cdot^{(m,\rho)}\|_{H^c}\lesssim_{m,\rho} 1$.
By Lemma \ref{lemma:apriori_estimate}, we have $\|Z_\cdot^{(m,\rho)}\|_{H^c}\lesssim_{m,\rho} 1$.
At the same time, by Lemma \ref{lemma:estimate_log_remainder} and Corollary \ref{corollary:Minfinity}, we get $|R_{\bitime{r},t}'(\polsolstd{\bitime{r}})M_{\bitime{r}}^{(m,\rho)}|\lesssim_{m,\rho} \timeunit.$

By the above discussion and Young integral, we conclude that
\begin{equation*}
\|N_\cdot^{(m,\rho)}\|_{H_-}\lesssim \|Z_\cdot^{(m,\rho)}\|_{H^c}\lesssim_{m,\rho} 1.
\end{equation*}
\end{proof}

At last, we show that $\|\Jacobi{\cdot}(\polJacobistd{\cdot})^{-1}\|_{H_-}\lesssim_{\rho} 2^{-m((q+1)H_--1)}$.
To estimate $\Jacobi{\cdot}(\polJacobistd{\cdot})^{-1}$, we show an elementary lemma.

\begin{lmm}\label{lemma:norm_space_convergence}
We set a norm space $(V,\|\cdot \|)$ and sequences $\{v^{(m)}\},\{v_0^{(m)}\}$ on $V$.
Next, given $\lambda_1,\kappa$ and monotonically increasing sequence $R_0(m)$, if the following condition holds:
\begin{align*}
\|v_0^{(m)}\|\lesssim R_0(m),&&\|v^{(m)}\|\lesssim 2^{-\lambda_1m},&&\|v^{(m)}-v_0^{(m)}\|\lesssim 2^{-\kappa m}\|v^{(m)}\|+2^{-\kappa m}R_0(m).
\end{align*}

Then, we have
\begin{align}
\|v^{(m)}-v_0^{(m)}\|\lesssim 2^{-\kappa m}R_0(m),&&\|v^{(m)}\|\lesssim R_0(m).\label{norm_state}
\end{align}
\end{lmm}

\begin{proof}
For $n,m\in \mathbb{N}_{>0}$, we define positive integer $R^{(n)}(m)$ inductively.

First, when $n=1$, we set $R^{(1)}(m)=2^{-\lambda_1m}$.

Next, when $R^{(n)}(m)$ is defined if $n=k$, we set $R^{(k+1)}(m)=\max\{R_0(m),2^{-\kappa m}R^{(k)}(m)\}$.

By the assumption, $\|v^{(m)}\|\lesssim R^{(1)}(m)$.
If $\|v^{(m)}\|\lesssim R^{(k)}(m)$ for fixed $k$, then we have

\begin{align*}
\|v^{(m)}\|&\leq \|v^{(m)}-v_0^{(m)}\|+\|v_0^{(m)}\|\lesssim 2^{-\kappa m}\|v^{(m)}\|+2^{-\kappa m}R_0(m)+R_0(m)\\
&\lesssim \max\{R_0(m),2^{-\kappa m}R^{(k)}(m)\}=R^{(k+1)}(m).
\end{align*}

Inductively, $\|v^{(m)}\|\lesssim R^{(n)}(m)$ for all $n$. Also, $R^{(n)}(m)=R_0(m)$ for sufficiently large $n$. 
Hence, we have (\ref{norm_state})
\end{proof}

Now, we estimate $\Jacobi{\cdot}(\polJacobistd{\cdot})^{-1}$.
\begin{lmm}\label{lemma:1-JMrho}
Let $\sigma\in C_b^{q+3},b\in C_b^2$, then $\|\Jacobi{\cdot}(\polJacobistd{\cdot})^{-1}\|_{H_-}\lesssim_{\rho} 2^{-m((q+1)H_--1)}$.
\end{lmm}
\begin{proof}

we derive and integrate $\Jacobi{\cdot}(\polJacobistd{\cdot})^{-1}$ with $\rho$ as follows:
\begin{align*}
\Jacobi{\cdot}(\polJacobistd{\cdot})^{-1}=&\Jacobi{\cdot} (\polJacobi{\cdot}{0})^{-1}-\Jacobi{\cdot}\int_0^\rho \partial_{\rho_1} (M_\cdot^{(m,\rho_1)})^{-1}d\rho_1\\
=&\Jacobi{\cdot} (\polJacobi{\cdot}{0})^{-1}-\Jacobi{\cdot}\int_0^\rho (M_\cdot^{(m,\rho_1)})^{-2}\partial_{\rho_1} (M_\cdot^{(m,\rho_1)})d\rho_1\\
=&\Jacobi{\cdot} (\polJacobi{\cdot}{0})^{-1}-\int_0^\rho \Jacobi{\cdot}(M_\cdot^{(m,\rho_1)})^{-1}(N_\cdot^{(m,\rho_1)})d\rho_1.
\end{align*}
Hence, we have
\begin{equation*}
\sup_\rho \|\Jacobi{\cdot}(\polJacobistd{\cdot})^{-1}-\Jacobi{\cdot} (\polJacobi{\cdot}{0})^{-1}\|_{H'}\lesssim \sup_\rho \|\Jacobi{\cdot}(\polJacobistd{\cdot})^{-1}\|_{H'}\sup_\rho\|N_\cdot^{(m,\rho)}\|_{H'}.
\end{equation*}
Estimating $N^{(m,\rho)}$ by Lemma \ref{lemma:apriori_estimate} and \ref{lemma:Hoelder_of_N}, it follows for some $\kappa>0$ that
\begin{equation*}
\sup_\rho \|\Jacobi{\cdot}(\polJacobistd{\cdot})^{-1}-\Jacobi{\cdot} (\polJacobi{\cdot}{0})^{-1}\|_{H'}\lesssim 2^{-\kappa m}\sup_\rho \|\Jacobi{\cdot}(\polJacobistd{\cdot})^{-1}\|_{H'}.
\end{equation*}
At the same time, from Corollary \ref{corollary:M-1J}, we get $\|\Jacobi{\cdot}(\polJacobi{\cdot}{0})^{-1}\|_1\lesssim 2^{-m((q+1)H_--1)}$.

Using Lemma \ref{lemma:norm_space_convergence}, we get the conclusion.
\end{proof}
\subsection{Estimate of the decomposed terms}

In this subsection, we will estimate $Z_t^{1,2,(m,A)},Z_t^{1,3,(m,A)},Z_t^{1-,(m)},Z_t^{2,(m,\rho)}$ as Lemma \ref{lemma:m,1,2,A}, \ref{lemma:m,1,3,A}, \ref{lemma:m,1-} and \ref{lemma:m,2} using Section \ref{subsection:yMN}, respectively.

First, we check $Z_t^{1,2,(m,A)}$.

\begin{lmm}\label{lemma:m,1,2,A}
Fix $A>1$. Let $\sigma\in C_b^{q+2},b\in C_b^2$.
Assume that $\hat{f}_{\hat{\Gamma}}$ is Lipschitz continuous for any $\hat{\Gamma}\in \hat{\mathbb{G}}_{1,A}$.

Then, there exists $\kappa>0$ such that $\|Z_\cdot^{1,2,(m,A)}\|_{H'}\lesssim 2^{-m\kappa}\|Z_\cdot^{(m,\rho)}\|_{H'}$.
\end{lmm}
\begin{proof}\

\begin{itemize}
\item At first, by definition of $Z^{1,2,(m,A)}$, we check that
\begin{align*}
\|Z_\cdot^{1,2,(m,A)}\|_{H'}\leq& \|\mathcal{I}_\cdot^{(m)}(J_{\cdot+}^{-1}\overline{f}_q(\nusol{\cdot})-\overline{f}_q(\exsol{\cdot}), d(B^{(q)}))\|_{H'}\\
&+\sum_{\hat{\Gamma}\in \mathbb{G}_{1,A}}\left\|\mathcal{I}_\cdot^{(m)}\left(J_{\cdot+}^{-1}(\hat{f}_{\hat{\Gamma}}(\nusol{\cdot})-\hat{f}_{\hat{\Gamma}}(\exsol{\cdot} )),d\left(B^{\Gamma}\right)\right)\right\|_{H'}.
\end{align*}

\item Next, we will estimate right-hand side for each $\hat{\Gamma}\in \mathbb{G}_{1,A}$.
By assumption, $\hat{f}_{\hat{\Gamma}}$ is Lipschitz, so it holds that
\begin{equation*}
\|\hat{f}_{\hat{\Gamma}}(\nusol{\cdot})-\hat{f}_{\hat{\Gamma}}(\exsol{\cdot})\|_\infty\leq \|\hat{f}_{\hat{\Gamma}}\|_1\|\error{\cdot}\|_\infty.
\end{equation*}

At the same time, from (\ref{transform_error}) we can estimate $\nusol{s}-\exsol{s}$ as
\begin{equation*}
\|\error{\cdot}\|_\infty\leq \sup_\rho \|\polJacobistd{\cdot}\|_\infty\sup_\rho \|Z_\cdot^{(m,\rho)}\|_\infty.
\end{equation*}

As a result, we can show that
\begin{align*}
&\left\|\mathcal{I}_\cdot^{(m)}\left(J_{\cdot+}^{-1}(\hat{f}_{\hat{\Gamma}}(\nusol{\cdot})-\hat{f}_{\hat{\Gamma}}(\exsol{\cdot} )),d\left(B^{\Gamma}\right)\right)\right\|_1\\
\lesssim& \|J_{\cdot+}^{-1}(\hat{f}_{\hat{\Gamma}}(\nusol{\cdot})-\hat{f}_{\hat{\Gamma}}(\exsol{\cdot} ))\|_{\infty}\left\|\mathcal{I}_\cdot^{(m)}(d(B^{\Gamma}))\right\|_1\\
\lesssim& \|\Jacobi{\cdot}^{-1}\|_\infty\|\hat{f}_{\hat{\Gamma}}(\nusol{\cdot})-\hat{f}_{\hat{\Gamma}}(\exsol{\cdot} )\|_{\infty}\left\|\mathcal{I}_\cdot^{(m)}(d(B^{\Gamma}))\right\|_1\\
\lesssim& \|\Jacobi{\cdot}^{-1}\|_\infty\|\hat{f}_{\hat{\Gamma}}\|_1\|\error{\cdot}\|_\infty\left\|\mathcal{I}_\cdot^{(m)}(d(B^{\Gamma}))\right\|_1\\
\lesssim& \sup_\rho\|\Jacobi{\cdot}^{-1}\|_{H_-}\|\hat{f}_{\hat{\Gamma}}\|_1\|\polJacobistd{\cdot}\|_{H_-}\|Z_\cdot^{(m,\rho)}\|_{H'}\left\|\mathcal{I}_\cdot^{(m)}(d(B^{\Gamma}))\right\|_1.
\end{align*}
By definition, $\|J_\cdot^{-1}\|_{H_-}<\infty$. Next, by the assumption, $\|\hat{f}_{\hat{\Gamma}}\|_1<\infty$. From Section \ref{subsection:yMN}, we have $\|\polJacobistd{\cdot}\|_{H_-}\lesssim_{m,\rho} 1$. As a priori estimate, $\|\mathcal{I}_\cdot^{(m)}(d(B^{\Gamma}))\|_1\lesssim 2^{-(|\hat{\Gamma}|-1)m}$.
So, we conclude $\|Z_\cdot^{1,2,(m,A)}\|_{H'}\lesssim 2^{-((q+1)H_--1)m}\|Z_\cdot^{(m,\rho)}\|_{H'}$.
\item At last, when $\nusol{}$ satisfies Condition ($\mathfrak{B}$), $\overline{f}_q\in C_2^q$ because $\overline{f}_q$ is a polynomial of $\{\sigma^{(k)}\}_{k=0}^{q}$.
Now, we can rewrite $\overline{f}_q(\nusol{s})-\overline{f_q}(\exsol{s})$ as
\begin{equation*}
\overline{f}_q(\nusol{s})-\overline{f}_q(\exsol{s})=\int_0^1 f'(\polsolstd{s}) \partial_\rho \polsolstd{s}d\rho=\int_0^1 \overline{f}_q'(\polsolstd{s}) M_s^{(m,\rho)} Z_s^{(m,\rho)}d\rho.
\end{equation*}
Hence, we get the estimate that
\begin{equation*}
\|\overline{f}_q(\nusol{\cdot})-\overline{f}_q(\exsol{\cdot})\|_{H'}\lesssim \sup_\rho\|\overline{f}_q''\|_\infty \|\polsolstd{\cdot}\|_{H_-}\|\polJacobistd{\cdot}\|_{H_-}\|Z_\cdot^{(m,\rho)}\|_{H'}.
\end{equation*}
As a result, we can show that
\begin{align*}
&\left\|\mathcal{I}_\cdot^{(m)}\left(J_{\cdot+}^{-1}(\overline{f}_q(\nusol{\cdot})-\overline{f}_q(\exsol{\cdot} )),d\left(B^{(q)}\right)\right)\right\|_{H^c}\\
\lesssim& \|J_{\cdot+}^{-1}(\overline{f}_q(\nusol{\cdot})-\overline{f}_q(\exsol{\cdot} ))\|_{H'}\left\|\mathcal{I}_\cdot^{(m)}(d(B^{(q)}))\right\|_{H^c}\\
\lesssim& \|\Jacobi{\cdot}^{-1}\|_{H'}\|\overline{f}_q(\nusol{\cdot})-\overline{f}_q(\exsol{\cdot} )\|_{H'}\left\|\mathcal{I}_\cdot^{(m)}(d(B^{(q)}))\right\|_{H^c}\\
\lesssim& \sup_\rho\|\Jacobi{\cdot}^{-1}\|_{H_-}\|\overline{f}_q''\|_\infty \|\polsolstd{\cdot}\|_{H_-}\|\polJacobistd{\cdot}\|_{H_-}\|Z_\cdot^{(m,\rho)}\|_{H'}\left\|\mathcal{I}_\cdot^{(m)}(d(B^{(q)}))\right\|_{H^c}.
\end{align*}
By estimates used for $\hat{f}_{\hat{\Gamma}}$ and $\left\|\mathcal{I}_\cdot^{(m)}(d(B^{(q)}))\right\|_{H^c}\lesssim 2^{-(qH_--H^c)m}$, which follows from Corollary \ref{corollary:BdiscreteHoelder}, we conclude $\|Z_\cdot^{1,2,(m,A)}\|_{H'}\lesssim 2^{-(qH_--H^c)m}\|Z_\cdot^{(m,\rho)}\|_{H'}$.
\end{itemize}

Combining estimates for each term, we obtain the statement.
\end{proof}

Next, we check $Z_t^{1,3,(m,A)}$.

\begin{lmm}\label{lemma:m,1,3,A}
Let $\sigma\in C_b^{q+2},b\in C_b^2$. Then, $\|Z_\cdot^{1,3,(m,A)}\|_1\lesssim 2^{-(A-1+\kappa)m}$ for some $\kappa>0$.
\end{lmm}
\begin{proof}
By definition, $\|\epsilon_\cdot^{3,(m,A)}\|_1\lesssim 2^{-(A+\kappa)m}$ holds.
Clearly, $\|e_\cdot^{3,(m,A)}\|_1\lesssim 2^{-(A-1+\kappa)m}$.

From $\|\Jacobi{\cdot}^{-1}\|_{H_-}<\infty$, we can see that
\begin{equation*}
\|Z_\cdot^{1,3,(m,A)}\|_1=\|\mathcal{I}_\cdot^{(m)}(J_{\cdot+}^{-1},d(e^{3,(m,A)}))\|_1\leq \|\Jacobi{\cdot}^{-1}\|_{H_-}\|e^{3,(m,A)}\|_1\lesssim 2^{-(A-1+\kappa)m}.
\end{equation*}
\end{proof}

Next, we check $Z_t^{1-,(m)}$.
\begin{lmm}\label{lemma:m,1-}
Let $\sigma\in C_b^{q+2},b\in C_b^2$. Then, $\|Z_\cdot^{1-,(m)}\|_{H'}\lesssim 2^{-m((q+1)H_--1)}\|Z_\cdot^{(m,\rho)}\|_{H'}$.
\end{lmm}
\begin{proof}
At first, we transform $Z_\cdot^{1-,(m)}$ as follows:
\begin{align*}
Z_\cdot^{1-,(m)}=&\mathcal{I}_\cdot^{(m)}(((M_{\cdot+}^{(m,0)})^{-1}-J_{\cdot+}^{-1}),d(\epsilon^{(m)}))\\
=&\mathcal{I}_\cdot^{(m)}((1-J_{\cdot+}^{-1}M_{\cdot+}^{(m,0)})(M_{\cdot+}^{(m,0)})^{-1},d(\epsilon^{(m)}))\\
=&\mathcal{I}_\cdot^{(m)}((1-M_{\cdot+}^{(m,0)}J_{\cdot+}^{-1}),d(Z^{(m,\rho)})
\end{align*}
At the same time, Corollary \ref{corollary:M-1J}, we have
\begin{equation*}
\|1-\polJacobi{\cdot}{0}\Jacobi{\cdot}^{-1}\|_1=\|\polJacobi{\cdot}{0}\Jacobi{\cdot}^{-1}\|_1\lesssim 2^{-m((q+1)H_--1)}.
\end{equation*}
Hence, using estimate of Young-Riemann sum, we can derive the following estimate:
\begin{align*}
\|Z_\cdot^{1-,(m)}\|_{H'}\leq& \|1-\polJacobi{\cdot}{0}\Jacobi{\cdot}^{-1}\|_1\|Z_\cdot^{(m,\rho)}\|_{H'}\lesssim 2^{-m((q+1)H_--1)}\|Z_\cdot^{(m,\rho)}\|_{H_-}.
\end{align*}

\end{proof}

Finally, we check $Z_t^{2,(m,\rho)}$.

\begin{lmm}\label{lemma:m,2}
Let $\sigma\in C_b^{q+3},b\in C_b^2$. Then, $\left\|Z_\cdot^{2,(m,\rho)}\right\|_{H'}\lesssim \|Z_\cdot^{(m,\rho)}\|_{H'}\|Z_\cdot^{(m,\rho)}\|_{H^c}$.
\end{lmm}
\begin{proof}
By the fundamental theorem of calculus,
\begin{align*}
V_t^{2,(m,\rho)} =& (\polJacobistd{t})^{-1}-(M_t^{(m,0)})^{-1}=\int_0^\rho \partial_\rho (M_t^{(m,\rho_1)})^{-1}d\rho_1\\
=&-\int_0^\rho (M_t^{(m,\rho_1)})^{-2}\partial_\rho (M_t^{(m,\rho_1)})d\rho_1=-\int_0^\rho (M_t^{(m,\rho_1)})^{-1}N_t^{(m,\rho_1)}d\rho_1
\end{align*}
As shown in Lemma \ref{lemma:Hoelder_of_N}, we have
\begin{align*}
V_t^{2,(m,\rho)} =& -\int_0^\rho (M_t^{(m,\rho_1)})^{-1}\mathcal{I}_t^{(m)}(Z_\cdot^{(m,\rho_1)},d(\overline{N}_\cdot^{(m,\rho_1)}))d\rho_1.
\end{align*}
From these equality, we can transform $Z_t^{(m,\rho)}$ as follows:
\begin{align*}
Z_t^{2,(m,\rho)}=&-\int_0^\rho \mathcal{I}_t^{(m)}((M_{\cdot+}^{(m,\rho_1)})^{-1}\mathcal{I}_{\cdot+}^{(m)}(Z_\cdot^{(m,\rho_1)},d(\overline{N}_\cdot^{(m,\rho_1)})),d(\epsilon^{(m,\rho)})) d\rho_1\\
=&-\int_0^\rho \mathcal{I}_t^{(m)}(\mathcal{I}_{\cdot+}^{(m)}(Z_\cdot^{(m,\rho_1)},d(\overline{N}_\cdot^{(m,\rho_1)}))),d(Z_\cdot^{(m,\rho)})) d\rho_1.
\end{align*}
By the formula of summation by parts, we obtain
\begin{equation*}
\mathcal{I}_t^{(m)}(Z_\cdot^{(m,\rho)},d(\overline{N}_\cdot^{(m,\rho)}))=Z_t^{(m,\rho)}\overline{N}_t^{(m,\rho)}-\mathcal{I}_t^{(m)}(\overline{N}_{\cdot+}^{(m,\rho)},d(Z_\cdot^{(m,\rho)})).
\end{equation*}
Using estimate of Young-Riemann sum, we can state the following inequality:
\begin{align*}
\left\|Z_\cdot^{2,(m,\rho)}\right\|_{H'}\lesssim& \|Z_\cdot^{(m,\rho)}\|_{H'}\|\overline{N}_\cdot^{(m,\rho)}\|_{H_-}\|Z_\cdot^{(m,\rho)}\|_{H^c}.
\end{align*}

From Lemma \ref{lemma:Hoelder_of_N}, we deduce the conclusion.
\end{proof}
\subsection{Proof of Theorem \ref{theorem:maintheorem}}

We are now ready to prove Theorem \ref{theorem:maintheorem}.
In this subsection, we will establish the proof using the estimates discussed above.

\begin{proof}[Proof of Theorem \ref{theorem:maintheorem}]
We can prove this theorem in the same way under both conditions ($\mathfrak{A}$) and ($\mathfrak{B}$).
To avoid repetition, we write the proof for the Condition ($\mathfrak{B}$).
To get a proof under Condition ($\mathfrak{A}$), it suffices to replace $\|\cdot\|_{H'}$ and $\|\cdot\|_{H^c}$ with $ \|\cdot\|_{\infty}$ and $\|\cdot\|_{1}$.

At first, we rewrite $\error{t}-\Jacobi{t} Z_t^{1,1,(m,A)}$ for $\timeinunit$ as
\begin{equation}
\error{t}-\Jacobi{t} Z_t^{1,1,(m,A)}=\int_0^1\polJacobistd{\bitime{r}}[Z_{\bitime{r}}^{(m,\rho)}-J_{\bitime{r}} (M_{\bitime{r}}^{(m,\rho)})^{-1} Z_{\bitime{r}}^{1,1,(m,A)}]d\rho+\epsilon_t^{(m)}.\label{decomposeYYJZ}
\end{equation}
By assumption $2^{-m(A-1)}=O(R(m))$, we can see $R(m)^{-1}\|\epsilon^{(m)}\|_\infty\to 0$, so that it suffices to estimate $Z_t^{(m,\rho)}-\Jacobi{t} (\polJacobistd{t})^{-1} Z_t^{1,1,(m,A)}$.
Now, we decompose $Z_t^{(m,\rho)}-\Jacobi{t} (\polJacobistd{t})^{-1} Z_t^{1,1,(m,A)}$ as follows:
\begin{align*}
&Z_t^{(m,\rho)}-\Jacobi{t} (\polJacobistd{t})^{-1} Z_t^{1,1,(m,A)}\\
=&Z_t^{1,2,(m,A)}+Z_t^{1,3,(m,A)}+Z_t^{1-,(m)}+Z_t^{2,(m,\rho)}+(1-\Jacobi{t}(\polJacobistd{t})^{-1})Z_t^{1,1,(m,A)}.
\end{align*}
Using Lemma \ref{lemma:m,1,2,A},\ref{lemma:m,1,3,A},\ref{lemma:m,1-},\ref{lemma:m,2} and \ref{lemma:1-JMrho} for each terms of right-hand side, respectively, we can get
\begin{align*}
&\|Z_\cdot^{(m,\rho)}-\Jacobi{\cdot} (\polJacobistd{\cdot})^{-1} Z_\cdot^{1,1,(m,A)}\|_{H'}\\
\leq& \|Z_\cdot^{1,2,(m,A)}\|_{H'}+\|Z_\cdot^{1,3,(m,A)}\|_{H'}+\|Z_\cdot^{1-,(m)}\|_{H'}\\
&+\|Z_\cdot^{2,(m,\rho)}\|_{H'}+\|1-\Jacobi{\cdot}(\polJacobistd{\cdot})^{-1}\|_{H'}\|Z_\cdot^{1,1,(m,A)}\|_{H'}\\
\lesssim& 2^{-m\kappa}\|Z_\cdot^{(m,\rho)}\|_{H'}+2^{-(A-1+\kappa)m}+2^{-m((q+1)H_--1)}\|Z_\cdot^{(m,\rho)}\|_{H'}\\
&+\|Z_\cdot^{(m,\rho)}\|_{H'}\|Z_\cdot^{(m,\rho)}\|_{H^c}+2^{-m((q+1)H_--1)}\|Z_\cdot^{1,1,(m,A)}\|_{H'}.
\end{align*}

By assumption (\ref{main_theorem_assumption_lambda}), it holds for any $\lambda>0$ that $\|Z_\cdot^{1,1,(m,A)}\|_{H'}\lesssim 2^{m\lambda}R(m)$ and $2^{-(A-1)m}\lesssim R(m)$.
At the same time, from Lemma \ref{lemma:apriori_estimate}, we have $\|Z_\cdot^{(m,\rho)}\|_{H^c}\lesssim 2^{-m\kappa}$ for some $\kappa>0$.
Hence, for some $\kappa>0$, we get
\begin{equation*}
\|Z_\cdot^{(m,\rho)}-\Jacobi{\cdot} (\polJacobistd{\cdot})^{-1} Z_\cdot^{1,1,(m,A)}\|_{H'}\lesssim_{\rho} 2^{-\kappa m}(\|Z_\cdot^{(m,\rho)}\|_{H',}+R_{\lambda}(m))\ a.s,
\end{equation*}
where we denote $2^{\lambda m}R(m)$ by $R_{\lambda}(m)$.

Next, combining Lemma \ref{lemma:apriori_estimate} and assumption, we get the assumption of Lemma \ref{lemma:norm_space_convergence}. Hence, we can get
\begin{equation*}
\|Z_\cdot^{(m,\rho)}-\Jacobi{\cdot} (\polJacobistd{\cdot})^{-1} Z_\cdot^{1,1,(m,A)}\|_{H'}\lesssim_{\rho} 2^{-\kappa m}R_{\lambda}(m)\ a.s.
\end{equation*}
From this inequality and (\ref{decomposeYYJZ}), we can estimate $\error{\cdot} -\Jacobi{\cdot} Z_\cdot^{1,1,(m,A)}$ as
\begin{equation*}
\|\error{\cdot} -\Jacobi{\cdot} Z_\cdot^{1,1,(m,A)}\|_\infty \lesssim 2^{-\kappa m}R_{\lambda}(m).
\end{equation*}
When we take $\lambda$ such that $\lambda<\kappa$, we can conclude that
\begin{equation*}
R(m)^{-1}\|\error{\cdot} -\Jacobi{\cdot} Z_\cdot^{1,1,(m,A)}\|_\infty \to 0\ a.s.
\end{equation*}
\end{proof}
\section{Calculation of asymptotic error}\label{section:calculation}

In this section, we will calculate the asymptotic errors of the ($k$)-Milstein scheme when $0<H<1/2$ and the Crank-Nicolson scheme when $1/4<H<1/2$.
Theorem \ref{theorem:maintheorem} shows that asymptotic distributions coincides with the weighted limit of $Z_t^{M,(A)}=Z^{1,1,(m,A)}$.
Hence, it suffices to check these schemes satisfy the assumption and to calculate the limit of the main term.

The main term $Z^{1,1,(m,A)}$ is in the forms of backward Stieltjes sum, so that we have to transform it as:
\begin{equation*}
Z^{M,(A)}=\mathcal{I}^{(m)} (J_{\cdot+}^{-1},d(\epsilon^{1,(m,A)}))=\mathcal{I}^{(m)} (\Jacobi{\cdot}^{-1},d(J_{\cdot_1}J_{\cdot_2}^{-1}\epsilon_{\cdot_1,\cdot_2}^{1,(m,A)})+O(|\cdot_2-\cdot_1|^{A+\kappa})).
\end{equation*}

At the same time, we expand $J_s\Jacobi{t}^{-1}$ as the inverse of (\ref{expansion_of_J}):

\begin{equation*}
J_s\Jacobi{t}^{-1}=1+\sum_{l=1}^q \frac{1}{l!}f_l(\exsol{s})\fBminc{s,t}^l+O(|t-s|^{(q+1)H_-}),
\end{equation*}
where $f_1\equiv -\sigma',\ f_2\equiv (\sigma')^2-\sigma''\sigma,\cdots$.

So, we can split $Z^{M,(A)}$ to $\mathcal{I}^{(m)}(\Jacobi{\cdot}^{-1},d(\epsilon^{1,*,(m,A)}+o(2^{mA})))$, where $\epsilon_{\bitime{r},t}^{1,*,(m,A)}$ is defined as

\begin{equation}
\epsilon_{\bitime{r},t}^{1,*,(m,A)}=\sum_{\begin{smallmatrix}\hat{\Gamma}\in \hat{\mathbb{G}}_{0,A},l\in \mathbb{N}_{\geq 0}\\
|\hat{\Gamma}|+lH\leq A
\end{smallmatrix}} g_k\hat{f}_{\Gamma}(\exsoldis{r} ) \fBmunit^l \fBmitinunit^{\hat{\Gamma}}-\sum_{\begin{smallmatrix}\Gamma\in \mathbb{G}_{0,A},l\in \mathbb{N}_{\geq 0}\\
|\Gamma|+lH\leq A
\end{smallmatrix}} g_l f_{\Gamma}(\exsoldis{r} ) \fBmunit^l \fBmitinunit^{\Gamma}\label{definiton_of_m11A*}.
\end{equation}

In the first subsection, we will enumerate $\epsilon^{1,*,(m,A)}$ of the ($k$)-Milstein scheme and the Crank-Nicolson scheme for all cases.
In the second subsection, we will calculate the limit of individual terms of $\epsilon^{1,*,(m,A)}$.
In the third subsection, we get the asymptotic error as the summation of individual limits.

\subsection{Decomposition of $\epsilon^{1,(m,A)}$}\label{subsection:decompose_epsilon}

In this subsection, we will calculate $\epsilon_{\bitime{r},t}^{1,*,(m,A)}$ of the ($k$)-Milstein scheme and the Crank-Nicolson scheme.
Results are classified by $H,A$, whether $b$ vanishes or not, whether each scheme holds Condition ($\mathfrak{A}$) or ($\mathfrak{B}$).

First, we have to expand that $\epsilon_{\bitime{r},t}^{((q+1)H)}$ of exact solution $\exsol{}$.
By Proposition \ref{proposition:taylorfory}, for any integer $K>0$ and any $s<t$, we have
\begin{align*}
\exsol{t}=&\exsol{s}+\sum_{k=1}^K \frac{1}{k!}\mathcal{D}^{k-1}\sigma (\exsol{s})\fBminc{s,t}^k+b(\exsol{s})(t-s)+b\sigma'(Y_{s}) \fBmitin_{s,t}^{10}+\sigma b'(Y_{s})\fBmitin_{s,t}^{01}\\
&+b(\sigma'\sigma)'(Y_{s}) \fBmitin_{s,t}^{110}+\sigma(b\sigma')'(Y_{s})\fBmitin_{s,t}^{101}+\sigma (\sigma b')'(Y_{s})\fBmitin_{s,t}^{011}\\
&+O(|t-s|^2)+O(|t-s|^{1+3H_-})+O(|t-s|^{(K+1)H_-}).
\end{align*}
Naturally, if $b$ does not vanish, when $0<H<1/2$ and $A=2H+1$, we get the following result:
\begin{align*}
\epsilon_{s,t}^{((q+1)H)}=&\frac{1}{(q+1)!}\mathcal{D}^q\sigma(Y_{s}) \fBminc{s,t}^{q+1}+\frac{1}{(q+2)!}\mathcal{D}^{q+1}\sigma(Y_{s})\fBminc{s,t}^{q+2}\\
&+b\sigma'(Y_{s}) \fBmitin_{s,t}^{10}+\sigma b'(Y_{s})\fBmitin_{s,t}^{01}\\
&+b(\sigma'\sigma)'(Y_{s}) \fBmitin_{s,t}^{110}+\sigma(b\sigma')'(Y_{s})\fBmitin_{s,t}^{101}+\sigma (\sigma b')'(Y_{s})\fBmitin_{s,t}^{011}+\epsilon_{s,t}^{(2H+1)}.
\end{align*}
In the same way, if $b$ vanishes, when $0<H< 1/2$, with $A=(K+1)H$, we get
\begin{align*}
\epsilon_{s,t}^{((q+1)H)}=&\sum_{j=q+1}^K\frac{1}{j!}\mathcal{D}^{j-1}\sigma(Y_{s}) \fBminc{s,t}^j+\epsilon_{s,t}^{((K+1)H)}.
\end{align*}

Next, for the ($k$)-Milstein scheme, we expand $\hat{\epsilon}_{\bitime{r},t}^{(m,(q+1)H)}$ for each $H$.
From the definition (\ref{defMilstein}), we get the following results:
\begin{itemize}
\item when $0<H<1/2$ and $k=q-1$, with $A=2H+1$,

\begin{align*}
\hat{\epsilon}_{\bitime{r},t}^{(m,(q+1)H)}=&-\frac{1}{q!}\mathcal{D}^{q-1}\sigma(\nusoldis{r}) \fBmunit^q+\frac{1}{2}(b\sigma'+\sigma b')(\nusoldis{r}) \timeunit \fBmunit +\hat{\epsilon}_{\bitime{r},t}^{(m,A)},
\end{align*}
\item when $0<H<1/2$ and $k=q$, with $A=2H+1$,

\begin{align*}
\hat{\epsilon}_{\bitime{r},t}^{(m,(q+1)H)}=&\frac{1}{2}(b\sigma'+\sigma b')(\nusoldis{r}) \timeunit \fBmunit +\hat{\epsilon}_{\bitime{r},t}^{(m,A)},
\end{align*}
\item when $0<H<1/2$ and $k=q+1$, with $A=2H+1$,

\begin{equation*}
\hat{\epsilon}_{\bitime{r},t}^{(m,(q+1)H)}=\frac{1}{(q+1)!}\mathcal{D}^q\sigma(\nusoldis{r}) \fBmunit^{q+1}+\frac{1}{2}(b\sigma'+\sigma b')(\nusoldis{r}) \timeunit \fBmunit +\hat{\epsilon}_{\bitime{r},t}^{(m,A)},
\end{equation*}
\item when $0<H<1/2$ and $k\geq q+2$, with $A=2H+1$,

\begin{align*}
\hat{\epsilon}_{\bitime{r},t}^{(m,(q+1)H)}=&\frac{1}{(q+1)!}\mathcal{D}^q\sigma(\nusoldis{r}) \fBmunit^{q+1}+\frac{1}{(q+2)!}\mathcal{D}^{q+1}\sigma(\nusoldis{r})\fBmunit^{q+2}\\
&+\frac{1}{2}(b\sigma'+\sigma b')(\nusoldis{r}) \timeunit \fBmunit +\hat{\epsilon}_{\bitime{r},t}^{(m,A)}.
\end{align*}
\end{itemize}
Also, if $b$ vanishes,  when $k\geq q$, we have
\begin{align*}
\hat{\epsilon}_{\bitime{r},t}^{((q+1)H)}=&\sum_{j=q+1}^k\frac{1}{j!}\mathcal{D}^{j-1}\sigma(Y_{s}) \fBminc{s,t}^j.
\end{align*}
In addition, when $k=q-1$, we have
\begin{align*}
\hat{\epsilon}_{\bitime{r},t}^{((q+1)H)}=&-\frac{1}{q!}\mathcal{D}^{q-1}\sigma(Y_{s}) \fBminc{s,t}^q.
\end{align*}

Now, we calculate $\epsilon_{\bitime{r},t}^{1,*,(m,A)}$ for the ($k$)-Milstein scheme. Let
\begin{align*}
\tilde{g}_{j,1}^{M}=&-\frac{1}{(j-1)!}\mathcal{D}^{j-1}\sigma,&\tilde{g}_{j,2}^{M}=\tilde{g}_{j,1}^{M}-\sigma'\tilde{g}_{j-1,1}^{M},\\
\tilde{g}_{j,3}^{M}=&\tilde{g}_{j,1}^{M}-\sigma'\tilde{g}_{j-1,1}^{M}-\frac{1}{2}(\sigma''\sigma-2(\sigma')^2)\tilde{g}_{j-2,1}^{M}.
\end{align*}
Then, we write $\epsilon_{\bitime{r},t}^{1,*,(m,A)}$ explicitly.
If $b$ does not vanish, with $A=2H+1$,
\begin{itemize}
\item when $0<H< 1/2$ and $k=q-1$,
\begin{align*}
\epsilon_{\bitime{r},t}^{1,*,(m,A)}=&\tilde{g}_{q,1}^M(\exsoldis{r}) \fBmunit^q+\tilde{g}_{q+1,2}^M(\exsoldis{r}) \fBmunit^{q+1}+\tilde{g}_{q+2,3}^M(\exsoldis{r})\fBmunit^{q+2}\\
&+(\sigma b'-b\sigma')(\exsoldis{r}) \fBmitinunit^{10*}\\
&-b(\sigma'\sigma)'(\exsoldis{r}) \fBmitinunit^{110}-\sigma(b\sigma')'(\exsoldis{r})\fBmitinunit^{101}-\sigma (\sigma b')'(\exsoldis{r})\fBmitinunit^{011},
\end{align*}
\item when $0<H< 1/2$ and $k=q$,

\begin{align*}
\epsilon_{\bitime{r},t}^{1,*,(m,A)}=&\tilde{g}_{q+1,1}^M(\exsoldis{r}) \fBmunit^{q+1}+\tilde{g}_{q+2,2}^M(\exsoldis{r})\fBmunit^{q+2}+(\sigma b'-b\sigma')(\exsoldis{r}) \fBmitinunit^{10*}\\
&-b(\sigma'\sigma)'(\exsoldis{r}) \fBmitinunit^{110}-\sigma(b\sigma')'(\exsoldis{r})\fBmitinunit^{101}-\sigma (\sigma b')'(\exsoldis{r})\fBmitinunit^{011},
\end{align*}
\item when $0<H< 1/2$ and $k=q+1$,

\begin{align*}
\epsilon_{\bitime{r},t}^{1,*,(m,A)}=&\tilde{g}_{q+2,1}^M\sigma(\exsoldis{r})\fBmunit^{q+2}+(\sigma b'-b\sigma')(\exsoldis{r}) \fBmitinunit^{10*}\\
&-b(\sigma'\sigma)'(\exsoldis{r}) \fBmitinunit^{110}-\sigma(b\sigma')'(\exsoldis{r})\fBmitinunit^{101}-\sigma (\sigma b')'(\exsoldis{r})\fBmitinunit^{011},
\end{align*}
\item when $0<H< 1/2$ and $k\geq q+2$,

\begin{align*}
\epsilon_{\bitime{r},t}^{1,*,(m,A)}=&(\sigma b'-b\sigma')(\exsoldis{r}) \fBmitinunit^{10*}\\
&-b(\sigma'\sigma)'(\exsoldis{r}) \fBmitinunit^{110}-\sigma(b\sigma')'(\exsoldis{r})\fBmitinunit^{101}-\sigma (\sigma b')'(\exsoldis{r})\fBmitinunit^{011}.
\end{align*}
\end{itemize}
On the other hand, if $b$ vanishes, when $0<H<1/2$ and $A=(k+2)H$,
\begin{align*}
\epsilon_{\bitime{r},t}^{1,*,(m,A)}=&-\tilde{g}_{k+1,1}^M\sigma(Y_{\bitime{r}})\fBminc{\bitime{r},t}^{k+1}-\tilde{g}_{k+1,2}^M(Y_{\bitime{r}})\fBminc{\bitime{r},t}^{k+2}\\
\end{align*}

Next, for the the Crank-Nicolson scheme, we expand $\hat{\epsilon}_{\bitime{r},t}^{(m,(q+1)H)}$.
In the same way to (\ref{expansionCN}), we have
\begin{align*}
\nusol{t}=&\nusoldis{r}+\sigma(\nusoldis{r})\fBmunit +b(\nusoldis{r})\timeunit\\
&+\frac{1}{2}\sigma\sigma'(\nusoldis{r})\fBmunit^2+\frac{1}{2}(\sigma b'+b\sigma')(\nusoldis{r})\fBmunit \timeunit\\
&+\frac{1}{4}(\sigma(\sigma\sigma')')(\nusoldis{r})\fBmunit^3+\left(\frac{1}{8}\sigma(\sigma')^3+\frac{3}{8}\sigma^2\sigma'\sigma''+\frac{1}{12}\sigma^3\sigma'''\right)\fBmunit^4\\
&+O(\timeunit^{2})+O(\timeunit^{1+2H_-})+O(\timeunit^{5H_-}).
\end{align*}
From this expansion, we define
\begin{align*}
\hat{g}_3^{CN}\coloneqq \frac{1}{4}\mathcal{D}^2\sigma,&&\hat{g}_4^{CN}\coloneqq \frac{1}{8}\sigma(\sigma')^3+\frac{3}{8}\sigma^2\sigma'\sigma''+\frac{1}{12}\sigma^3\sigma'''
\end{align*}
such that
\begin{align*}
\nusol{t}=&\nusoldis{r}+\sigma(\nusoldis{r})\fBmunit +b(\nusoldis{r})\timeunit+\frac{1}{2}bb'(\nusoldis{r})\timeunit^2\\
&+\frac{1}{2}\sigma\sigma'(\nusoldis{r})\fBmunit^2+\frac{1}{2}(\sigma b'+b\sigma')(\nusoldis{r})\fBmunit \timeunit\\
&+\hat{g}_3^{CN}(\nusoldis{r})\fBmunit^3+\hat{g}_4^{CN}(\nusoldis{r})\fBmunit^4\\
&+O(\timeunit^{2})+O(\timeunit^{5H_-})+O(\timeunit^{1+2H_-}).
\end{align*}
Then, when $1/4<H<1/2$, taking $A=3H+1/2$, we have
\begin{align*}
\hat{\epsilon}_{\bitime{r},t}^{(m,(q+1)H)}=&\hat{g}_3^{CN}\sigma(\nusoldis{r}) \fBmunit^3+\hat{g}_4^{CN}(\nusoldis{r})\fBmunit^4+\frac{1}{2}(b\sigma'+\sigma b')(\nusoldis{r}) \timeunit \fBmunit .
\end{align*}

Now, we calculate $\epsilon_{\bitime{r},t}^{1,*,(m,A)}$ for the Crank-Nicolson scheme. Define 
\begin{align*}
\tilde{g}_3^{CN}=&\hat{g}_3^{CN}-\frac{1}{3!}\mathcal{D}^2\sigma=\frac{1}{12}(\sigma''\sigma^2+\sigma(\sigma')^2),\\
\tilde{g}_4^{CN}=&\left(\hat{g}_4^{CN}-\frac{1}{4!}\mathcal{D}^3\sigma\right)-\sigma'\tilde{g}_3^{CN}=\frac{1}{24}\left(3\sigma^2\sigma'\sigma''+\sigma^3\sigma'''\right).
\end{align*}
For $1/4<H<1/2$, let $A=3H+1/2$, we get
\begin{align*}
\epsilon_{\bitime{r},t}^{1,*,(m,A)}=&\tilde{g}_3^{CN}(\exsoldis{r}) \fBmunit^3+\tilde{g}_4^{CN}(\exsoldis{r})\fBmunit^4+(\sigma b'-b\sigma')(\exsoldis{r}) \fBmitinunit^{10*}.
\end{align*}
\subsection{Calculation of the limits of individual terms}\label{subsection:limit_of_individual_sum}

In this subsection, we will take the limit of $\mathcal{I}_\cdot^{(m)}(d(\epsilon^{1,*,(m,A)})$.

At first, we set
\begin{align*}
V_{s,t}^{10*}=&(t-s)^{-(H+1/2)}\fBmitin_{s,t}^{10*},&&&V_{s,t}^{(l)}=&(t-s)^{1/2}H_l((t-s)^{-H}\fBminc{s,t}),\\
V_{s,t}^{110}=&(t-s)^{-2H}\fBmitin_{s,t}^{110},&V_{s,t}^{101}=&(t-s)^{-2H}\fBmitin_{s,t}^{101},&V_{s,t}^{011}=&(t-s)^{-2H}\fBmitin_{s,t}^{011}.
\end{align*}
Then, the term to be calculated is expressed as a sum of terms of the form $\mathcal{I}_\cdot^{(m)}(J^{-1} f_*(Y),d(V^*))$.
Therefore, we will take the limits for each $V^*$ to derive asymptotic errors.

We remark that the limit of $\mathcal{I}_\cdot^{(m)}(f_*(B),d(V^*))$ is studied in detail and these studies are called Hermite variation (for example, \cite{Nourdin-Nualart-Tudor2010}).
We could use moment estimates such as these studies, and several studies on asymptotic errors use it.
However, in this case, we need the higher order Malliavin derivative of $J^{-1}f(Y)$ for the moment estimate, which is much more complicated than that of $f(B)$. 
Hence, for simplicity, we apply the methodology of Proposition 4.7 in \cite{Liu-Tindel2019} and estimate only pathwisely.

First, we take a limit for $V^{(l)}$.
\begin{prp}\label{proposition:Itointegrallower}
Given any $f\in C_b^2$ and $1/4<H<1/2$,
\begin{equation*}
\left(B_\cdot,\mathcal{I}_\cdot^{(m)}(J^{-1} f(Y),d(V^{(l)}))\right)\xrightarrow{d}\left(B_\cdot,C_{(k)}\int_0^\cdot \Jacobi{t}^{-1}f(\exsol{t})dW_t^{(l)}\right)
\end{equation*}
in the sense of the Skorokhod topology, where $C_{(l)}$ and $W^{(l)}$ is defined in Lemma \ref{lemma:standardBrown}.
Moreover, we can estimate for any $\epsilon_1,\epsilon_2>0$ that\newline$\sup_m 2^{m \epsilon_1}\|\mathcal{I}_\cdot^{(m)}(J^{-1} f(Y),d(V^{(k)}))\|_{(1/2-\epsilon_2)}<\infty\ a.s$.
\end{prp}
Proposition 4.7 in \cite{Liu-Tindel2019}, where each term is treated as $L^p$-integrable random variables.
However, we can apply this proposition when each term has a real value.

This proposition follows from Lemma \ref{lemma:standardBrown} and 
At the same time, we remark that Proposition \ref{proposition:Itointegrallower} corresponds to Theorem 4.3 in \cite{Aida-Naganuma2020}.

\begin{prp}\label{proposition:Hermitevariationvanish}

Given any $f\in C_b^2,0<H\leq 1/4$, and $l\geq 3$,
\begin{equation*}
2^{-(1/2-H)m}\mathcal{I}_\cdot^{(m)}(J^{-1} f(Y),d(V^{(l)}))\xrightarrow{a.s.}0
\end{equation*}
in the sense of the Skorokhod topology.
Moreover, we can estimate that for any $\epsilon_1,\epsilon_2>0$,\newline$\limsup_m 2^{-(1/2-H)m}\|\mathcal{I}_\cdot^{(m)}(J^{-1} f(Y),d(V^{(l)}))\|_{1-H-\epsilon}=0\ a.s$ for any $\epsilon>0$.

\end{prp}

\begin{proof}

Following to Proposition 4.7 in \cite{Liu-Tindel2019}, we can decompose $J_u^{-1}f(Y_u)V_{u,v}^{(l)}$ as
\begin{align*}
J_u^{-1}f(Y_u)V_{u,v}^{(l)}=&J_s^{-1}f(\exsol{s})V_{u,v}^{(l)}+J_s^{-1}\mathcal{V}f(\exsol{s})\fBminc{s,u}V_{u,v}^{(l)}+\hat{R}_{s,u}V_{u,v}^{(l)}\\
\eqqcolon&Z_{u,v}^{1,(s)}+Z_{u,v}^{2,(s)}+Z_{u,v}^{3,(s)},
\end{align*}
where we define $\hat{R}_{s,u}$ as
\begin{equation}
\hat{R}_{s,u}\coloneqq J_u^{-1}f(Y_u)-J_s^{-1}f(\exsol{s})-J_s^{-1}\mathcal{V}f(\exsol{s})\fBminc{s,u}.\label{hatRdefinition}
\end{equation}
Then, $Z_{u,v}^{3,(s)}=\hat{R}_{s,u}V_{u,v}^{(l)}$.

Now, taking Stieltjes sum of both sides, we have
\begin{equation*}
\mathcal{I}_{s,t}^{(m)}(J^{-1}f(y),d(V^{(l)}))=\mathcal{I}_{s,t}^{(m)}(d(Z^{1,(s)}))+\mathcal{I}^{(m)}(d(Z^{2,(s)}))+\mathcal{I}^{(m)}(d(Z^{3,(s)})).
\end{equation*}

\begin{itemize}
\item Firstly, we estimate $\mathcal{I}^{(m)}(d(Z^{1,(s)}))$. From Lemma \ref{lemma:standardBrown}, we have
\begin{align*}
\sup_m2^{-m\epsilon}\|\mathcal{I}_\cdot^{(m)}(d(Z^{1,(s)}))\|_{1/2-\epsilon}<\infty
\end{align*}
for any $\epsilon>0$.
\item Secondly, we estimate $\mathcal{I}^{(m)}(d(Z^{2,(s)}))$. We can expand $\fBminc{s,u}V_{u,v}^{(l)}$ as
\begin{align*}
\fBminc{s,u}V_{u,v}^{(l)}=&(v-u)^{1/2-lH}I_1(\delta_{s,u})I_l(\delta_{u,v}^{\otimes l})\\
=&(v-u)^{1/2-lH}(l\langle \delta_{s,u},\delta_{u,v}\rangle_\mathscr{H}I_{l-1}(\delta_{u,v}^{\otimes l})+I_{l+1}(\delta_{s,u}\otimes \delta_{u,v}^{\otimes l})).
\end{align*}
In this case, following the way to Lemma \ref{lemma:integral10*}, we obtain
\begin{align*}
\sup_m 2^{(H-\epsilon)m}\|\mathcal{I}_\cdot^{(m)}(d(Z^{2,(s)}))\|_{1/2-\epsilon}<\infty.
\end{align*}
\item Lastly, we estimate $\mathcal{I}^{(m)}(d(Z^{3,(s)}))$. 

Because of the definition of $J$ (\ref{def_J}) and Proposition \ref{proposition:elementary path as a controlled path} and \ref{proposition:Operation for a controlled path}, $J^{-1}f(Y)$ is a controlled path and $(J^{-1}f(Y))^{(2)}=J^{-1} \mathcal{V}f(Y)$.
Hence, we have 
\begin{equation}
|\hat{R}_{u,v}|\lesssim (v-u)^{2H_-}\label{hatRestimate}
\end{equation}
Also, from Lemma \ref{lemma:standardBrown} and Proposition \ref{proposition:degeneratediscreteHoelder},
\begin{align*}
\sup_m 2^{-(1/2-H-\delta)m}\|\mathcal{I}_\cdot^{(m)}(d(V^{(l)}))\|_{1-H-\delta-\epsilon}<\infty
\end{align*}
So, using an estimate of Stieltjes sum, we have
\begin{align*}
\sup_m 2^{-(1/2-H-\delta)m}\|\mathcal{I}_\cdot^{(m)}(d(Z^{3,(s)}))\|_{1-H-\delta-\epsilon}<\infty
\end{align*}
for $\delta,\epsilon$ such that $\delta+\epsilon<H$.
\end{itemize}

Consequently, we have $\displaystyle\limsup_m 2^{-(1/2-H)m}\|\mathcal{I}_\cdot^{(m)}(J^{-1} f(Y),d(V^{(l)}))\|_{1-H-\epsilon}=0\ a.s$.

\end{proof}

Next, we take a limit for $V^{10*}$.
We require a modification to the discussion for Proposition \ref{proposition:Itointegrallower}, so that we will write our proof in detail.
\begin{lmm}\label{lemma:integral10*}
Given any $f\in C_b^2,0<H<1/2$,
\begin{equation}
2^{-(1/2-H)m}\mathcal{I}_\cdot^{(m)}(J^{-1} f(Y),d(V^{10*}))\to -\frac{1-2H}{4(1+2H)}\int_0^\cdot \Jacobi{t}^{-1}\mathcal{V} f(\exsol{t})dt\ a.s\label{10*limit}
\end{equation}
in the sense of the Skorokhod topology, where $\mathcal{V}:f\mapsto \sigma f'-\sigma'f$ is defined in Section \ref{section:notation}.
Moreover, we can estimate for any $\epsilon>0$ that
\begin{equation*}
\sup_m2^{-(1/2-H)m}\|\mathcal{I}_\cdot^{(m)}(J^{-1} f(Y),d(V^{10*}))\|_{1-H-\epsilon}<\infty\ a.s.
\end{equation*}
\end{lmm}

\begin{proof}
As with above propositions, we split $J_u^{-1}f(Y_u)\fBmitin_{u,v}^{10*}$ to $Z_{u,v}^{1,(s)},Z_{u,v}^{2,(s)}$ and $Z_{u,v}^{3,(s)}$ as follows:
\begin{align*}
J_u^{-1}f(Y_u)\fBmitin_{u,v}^{10*}=&J_s^{-1}f(\exsol{s})\fBmitin_{u,v}^{10*}+J_s^{-1}\mathcal{V}f(\exsol{s})\fBminc{s,u}\fBmitin_{u,v}^{10*}+\hat{R}_{s,u}\fBmitin_{u,v}^{10*}\\
\eqqcolon&Z_{u,v}^{1,(s)}+Z_{u,v}^{2,(s)}+Z_{u,v}^{3,(s)},
\end{align*}
where $\hat{R}_{s,u}$ is defined as (\ref{hatRdefinition}).
With respect to this split, we can decompose $\mathcal{I}_{s,t}^{(m)}(J^{-1}f(y),d(V^{10*}))$ as follows:
\begin{align*}
&2^{-(1/2-H)m}\mathcal{I}_{s,t}^{(m)}(J^{-1}f(y),d(V^{10*}))\\
=&2^{2Hm}(\mathcal{I}_{s,t}^{(m)}(d(Z^{1,(s)}))+\mathcal{I}_{s,t}^{(m)}(d(Z^{2,(s)}))+\mathcal{I}_{s,t}^{(m)}(d(Z^{3,(s)}))).
\end{align*}
Next, we estim the H\"older continuity of  each term.

\begin{itemize}
\item Firstly, we estimate $\mathcal{I}_{s,t}^{(m)}(d(Z^{1,(s)}))$.
From Proposition \ref{proposition:Hoelder_estimate_iteerated_integral} and Corollary \ref{corollary:discreteGRRforwienerchaos}, we have these two statements:
\begin{align*}
2^{2Hm}\mathcal{I}_{s,t}^{(m)}(d(Z^{1,(s)}))\to 0,\ a.s,&&|2^{2Hm}\mathcal{I}_{s,t}^{(m)}(d(Z^{1,(s)}))|\lesssim_m (t-s)^{1-H-\epsilon}.
\end{align*}
\item Secondly, we estimate $\mathcal{I}^{(m)}(d(Z^{2,(s)}))$.
We have already written that we can regard $\fBminc{s,u}$ and $\fBmitin_{u,v}^{10*}$ as elements of first Wiener chaos $\mathscr{H}_1$ in Section \ref{subsection:discrete_Hoelder_norm}.
Additionally, $\mathscr{H}_1$ is isomorphic to Cameron-Martin space $\mathfrak{H}$ by Wiener-It\^{o} integral $I_1:\mathfrak{H}\to \mathscr{H}_1$.
Hence, we set $\delta_{s,t}, \epsilon_{s,t}\in \mathfrak{H}$ such that $\fBminc{s,t}=I_1(\delta_{s,t})$ and $\fBmitin_{s,t}^{10*}=I_1(\epsilon_{s,t})$.
For further details, please refer to \cite{Aida-Naganuma2020}, where we use $\epsilon$ in this paper instead of $\zeta$ used in \cite{Aida-Naganuma2020}.
Now, we decompose $\fBminc{s,u}\fBmitin_{u,v}^{10*}$ with respect to Wiener-chaos as follows:
\begin{equation*}
\fBminc{s,u}\fBmitin_{u,v}^{10*}=I_1(\delta_{s,u})I_1(\epsilon_{u,v})=\langle \delta_{s,u},\epsilon_{u,v}\rangle_\mathscr{H}+I_2(\delta_{s,u}\otimes \epsilon_{u,v}).
\end{equation*}
Following this decomposition, we can split $Z_{u,v}^{2,(s)}$ as below:
\begin{equation*}
Z_{u,v}^{2,(s)} = Z_{u,v}^{2,1,(s)}+Z_{u,v}^{2,2,(s)}+Z_{u,v}^{2,3,(s)}=J_s^{-1}\mathcal{V}f(\exsol{s})(\tilde{Z}_{u,v}^{2,1,(s)}+\tilde{Z}_{u,v}^{2,2,(s)}+\tilde{Z}_{u,v}^{2,3,(s)}),
\end{equation*}
where
\begin{align*}
\tilde{Z}_{u,v}^{2,1,(s)} = -\frac{H|t-u|^{1+2H}}{2(1+2H)},&&\tilde{Z}_{u,v}^{2,2,(s)}=\langle \delta_{s,u},\epsilon_{u,v}\rangle- \tilde{Z}_{u,v}^{2,1,(s)},&&\tilde{Z}_{u,v}^{2,3,(s)}=I_2(\delta_{s,u} \otimes \epsilon_{u,v}).
\end{align*}

Now, we estimate $\mathcal{I}_{s,t}^{(m)}(d(\tilde{Z}^{2,1,(s)})),\mathcal{I}_{s,t}^{(m)}(d(\tilde{Z}^{2,2,(s)}))$ and $\mathcal{I}_{s,t}^{(m)}(d(\tilde{Z}^{2,3,(s)}))$.

At first, it follows that $2^{2Hm}\mathcal{I}_{s,t}^{(m)}(d(\tilde{Z}^{2,1,(s)}))=-\frac{1-2H}{4(1+2H)}(t-s)$ immediately.

Next, due to Lemma \ref{lemma:estimateexpectedvalue}, we have $\tilde{Z}_{u,v}^{2,2,(s)}\lesssim \frac{|v-u|^3}{|u-s|^{2-2H}}$. Therefore, we can see that
\begin{equation*}
\mathcal{I}_{s,t}^{(m)}(d(\tilde{Z}^{2,2,(s)}))\lesssim \sum_{r= \lfloor 2^m s\rfloor+1}^{ \lfloor 2^m t\rfloor-1} \frac{|\bitime{r+1}-\bitime{r}|^3}{|\bitime{r}-s|^{2-2H}}\leq 2^{-(1+2H)m}(t-s)\sum_{n=1}^\infty n^{-(2-2H)}.
\end{equation*}
Since we have assumed $H<1/2$, immediately $2-2H>1$, so that we say that $\sum_{n=1}^\infty n^{-(2-2H)}<\infty$. Hence, we get $2^{2Hm}\mathcal{I}_{s,t}^{(m)}(d(\tilde{Z}^{2,2,(s)}))\lesssim 2^{-m}$.

At last, we estimate the variance of $\mathcal{I}(d(\tilde{Z}^{2,3,(s)}))$ using Lemma \ref{lemma:estimateexpectedvalue} as follows:
\begin{eqnarray*}
&&E[\mathcal{I}_{s,t}^{(m)}(d(\tilde{Z}^{2,3,(s)}))^2]\\
&=&\sum_{r_1,r_2= \lfloor 2^ms\rfloor}^{ \lfloor 2^m t\rfloor-1}\left(E[\fBminc{s,\bitime{r_1}}\fBmitin_{\bitime{r_2},\bitime{r_2+1}}^{10*}]E[\fBminc{s,\bitime{r_2}},\fBmitin_{\bitime{r_1}\bitime{r_1+1}}^{10*}]\right.\\
&&+\left.E[\fBminc{s,\bitime{r_1}}\fBminc{s,\bitime{r_2}}]E[\fBmitin_{\bitime{r_1},\bitime{r_1+1}}^{10*}\fBmitin_{\bitime{r_2},\bitime{r_2+1}}^{10*}]\right)\\
&\lesssim& 2^{-(2+4H)m}\sum_{r_1,r_2= \lfloor 2^ms\rfloor}^{ \lfloor 2^m t\rfloor-1}\left(\frac{1}{1\vee(r_1-2^ms)^{2-2H}}+\frac{1}{1\vee|r_2-r_1|^{2-2H}}\right)\\
&&\left(\frac{1}{1\vee(r_2-2^ms)^{2-2H}}+\frac{1}{1\vee |r_2-r_1|^{2-2H}}\right)\\
&&+(t-s)^{2H}2^{-(3+2H)m}\sum_{r_1,r_2= \lfloor 2^ms\rfloor}^{ \lfloor 2^m t\rfloor-1}\frac{1}{1\vee|r_2-r_1|^{3-2H}}\\
&\lesssim& 2^{-(2+4H)m}\left(\sum_{n=1}^\infty \frac{1}{n^{2-2H}}\right)^2+(t-s)^{1+2H}2^{-(2+2H)m}\sum_{n=1}^\infty \frac{1}{n^{3-2H}}\\
&\lesssim& (t-s)^{1+2H}2^{-(2+2H)m}.
\end{eqnarray*}
From Corollary \ref{corollary:discreteGRRforwienerchaos}, for any $\epsilon>0$, we get
\begin{equation*}
\mathcal{I}_{s,t}^{(m)}(d(\tilde{Z}^{2,3,(s)}))\lesssim (t-s)^{H+1/2-\epsilon}2^{-(1+H-\epsilon)m}\ a.s.
\end{equation*}

Combining these inequalities, we conclude that $2^{2Hm}\mathcal{I}_{s,t}^{(m)}(d(Z^{2,(s)}))\lesssim (t-s)$.
\item Thirdly, as we seen in (\ref{hatRestimate}), $\hat{R}_{u,v}\lesssim (v-u)^{2H_-}$.
We introduce $\overline{Z}_{u,v}^{3,(s)}$ as $\overline{Z}_{u,v}^{3,(s)}\coloneqq \hat{R}_{s,u}\mathcal{I}_{u,v}^{(m)}(V^{10*})$
such that $\mathcal{I}^{(m)}(d(\overline{Z}^{3,(s)}))=\mathcal{I}^{(m)}(d(Z^{3,(s)}))$.

Now, we will estimate $\delta (\overline{Z}^{3,(s)})_{u,v,t}$ for $s<u<v<t$.
By definition, we get $\delta (\overline{Z}^{3,(s)})_{u,v,t}=-\hat{R}_{u,v}\mathcal{I}_{v,t}^{(m)}(d(B^{10*}))$.
Because of the definition of $J$ (\ref{def_J}) and formula of operations for controlled paths (Proposition \ref{proposition:elementary path as a controlled path} and \ref{proposition:Operation for a controlled path}), $J^{-1}f(Y)$ is a controlled path and $(J^{-1}f(Y))^{(2)}=J^{-1} \mathcal{V}f(Y)$.
Hence, we have $|\hat{R}_{u,v}|\lesssim (v-u)^{2H_-}$.
At the same time, from Lemma \ref{lemma:standardBrown} and Corollary \ref{corollary:discreteGRRforwienerchaos}, we can say for any $\epsilon>0$ that
\begin{equation*}
2^{2Hm}\mathcal{I}_{v,t}^{(m)}(d(B^{10*}))\lesssim_m (t-v)^{1-H-\epsilon}\ \ a.s.
\end{equation*}
Combining the estimates for $\hat{R}_{u,v}$ and $\mathcal{I}_{v,t}^{(m)}(d(B^{10*}))$, for any $\epsilon>0$, we get
\begin{equation*}
2^{2Hm}\delta (\overline{Z}^{3,(s)})_{u,v,t}\lesssim_m (t-u)^{1+H-\epsilon}.
\end{equation*}
Applying the sewing lemma for $\overline{Z}^{3,(s)}$, we obtain \newline$2^{2mH}\mathcal{I}_{s,t}^{(m)}(d(Z^{3,(s)}))\lesssim_m (t-s)^{1+\delta}$ with some $\delta>0$.
\end{itemize}

Next, combining these estimates, we obtain the desired limit and H\"older estimate.
A H\"older estimate of $\mathcal{I}_{s,t}^{(m)}(J^{-1} f(Y),d(V^{10*}))$ follows from that of each term.
At the same time, we can see the following equation for $m>n$ and $s,t\in \mathbb{P}_n$:
\begin{align*}
&2^{-(1/2-H)m}\mathcal{I}_{s,t}^{(m)}(J^{-1} f(Y),d(V^{10*}))\\
=&2^{2Hm}\mathcal{I}_{s,t}^{(n)}(d(\mathcal{I}^{(m)}(d(Z^{1,(\cdot_1)}+Z^{2,1,(\cdot_1)}+Z^{2,2,(\cdot_1)}+Z^{2,3,(\cdot_1)}+Z^{3,(\cdot_1)})))).
\end{align*}

\begin{itemize}
\item Firstly, for $Z^{1,(\cdot_1)}$,
\begin{align*}
2^{(H+1/2)m}\mathcal{I}_{s,t}^{(n)}(d(\mathcal{I}_{\cdot_1,\cdot_2}^{(m)}(d(Z^{1,(\cdot_1)}))))\xrightarrow[d]{m\to\infty}& \mathcal{I}_{s,t}^{(n)}(J^{-1}f(y),d(W^{10*}))\\
\xrightarrow[d]{n\to\infty}&\int_s^t J_u^{-1}f(\exsol{t})dW_u^{10*}.
\end{align*}
So, we have $2^{2Hm}\mathcal{I}_{s,t}^{(n)}(d(\mathcal{I}_{\cdot_1,\cdot_2}^{(m)}(d(Z^{1,(\cdot_1)}))))$ convergent to $0$ a.s.
\item Secondly, for $Z^{2,1,(\cdot_1)}$,
\begin{align*}
&2^{2Hm}\mathcal{I}_{s,t}^{(n)}(d(\mathcal{I}_{\cdot_1,\cdot_2}^{(m)}(d(Z^{2,1,(\cdot_1)}))))=-\frac{1-2H}{4(1+2H)}\mathcal{I}_{s,t}^{(n)}(J^{-1}\mathcal{V}f(y),du)\\
\xrightarrow[a.s.]{n\to\infty}& -\frac{1-2H}{4(1+2H)}\int_s^t J_u^{-1}\mathcal{V} f(Y_u)du.
\end{align*}
\item Thirdly, for $Z^{2,2,(\cdot_1)}$,\newline$2^{2Hm}\mathcal{I}_{s,t}^{(n)}(d(\mathcal{I}_{\cdot_1,\cdot_2}^{(m)}(d(Z^{2,2,(\cdot_1)}))))\lesssim 2^{n-m}\leq 1$ and $2^{n-m}\xrightarrow[a.s.]{m\to\infty}0$.
\item Fourthly, for $Z^{2,3,(\cdot_1)}$,\newline$2^{2Hm}\mathcal{I}_{s,t}^{(n)}(d(\mathcal{I}_{\cdot_1,\cdot_2}^{(m)}(d(Z^{2,3,(\cdot_1)}))))\lesssim 2^{-2Hn}2^{-(1-H-\epsilon)m}\xrightarrow[a.s.]{m\to\infty}0$.
\item Lastly, for $Z^{3,(\cdot_1)}$, $2^{2Hm}\mathcal{I}_{s,t}^{(n)}(d(\mathcal{I}_{\cdot_1,\cdot_2}^{(m)}(d(Z^{3,(\cdot_1)}))))\lesssim_m 2^{-(2H+H^c-1)n}\xrightarrow[a.s.]{n\to\infty}0$.
\end{itemize}

Combining the above limits, we obtain the following two estimates. First, for any $\epsilon>0$, there exists $n_0=n_0(\omega)$ such that
\begin{align*}
\sup_{0<s<t<T}&\left|2^{2Hm}\mathcal{I}_{s,t}^{(n_0)}(d(\mathcal{I}_{\cdot_1,\cdot_2}^{(m)}(d(Z^{2,1,(\cdot_1)}))))\right.\\
&+\limsup_m2^{2Hm}\mathcal{I}_{s,t}^{(n_0)}(d(\mathcal{I}_{\cdot_1,\cdot_2}^{(m)}(d(Z^{3,(\cdot_1)}))))\\
&\left.+\frac{1-2H}{4(1+2H)}\int_s^t J_u^{-1}f(Y_u)du\right|\leq \frac{\epsilon}{2}.
\end{align*}
Next, we take $M_0$ satisfying that for any $m\geq M_0(\omega)$,
\begin{equation*}
\sup_{0<s<t<T}|2^{2Hm}\mathcal{I}_{s,t}^{(n_0)}(d(\mathcal{I}_{\cdot_1,\cdot_2}^{(m)}(d(Z^{1,(\cdot_1)}+Z^{2,2,(\cdot_1)}+Z^{2,3,(\cdot_1)}))))|<\frac{\epsilon}{2}.
\end{equation*}
If $m\geq M_0$, adding both sides of these two inequalities, we have
\begin{equation*}
\sup_{0<s<t<T}\left|2^{-(1/2-H)m}\mathcal{I}_{s,t}^{(m)}(J^{-1} f(Y),d(V^{10*}))+\frac{1-2H}{4(1+2H)}\int_s^t J_u^{-1}\mathcal{V} f(Y_u)du\right|<\epsilon.
\end{equation*}

In other words, we can conclude (\ref{10*limit}).

\end{proof}

At last, we take the limit of $\mathcal{I}_\cdot^{(m)}(J^{-1}f(Y),d(V^{\Gamma}))$, where $\Gamma$ is one of $110,101,011$.
\begin{lmm}\label{lemma:integral110}

Given any $f\in C_b^1$ and $0<H<1/2$,
\begin{align*}
\mathcal{I}_\cdot^{(m)}(J^{-1}f(Y),d(V^{110}))&\xrightarrow{a.s.}\frac{1}{2(1+2H)}\int_0^\cdot \Jacobi{t}^{-1}f(\exsol{t})dt,\\
\mathcal{I}_\cdot^{(m)}(J^{-1}f(Y),d(V^{101}))&\xrightarrow{a.s.}-\frac{1-2H}{2(1+2H)}\int_0^\cdot \Jacobi{t}^{-1}f(\exsol{t})dt,\\
\mathcal{I}_\cdot^{(m)}(J^{-1}f(Y),d(V^{011}))&\xrightarrow{a.s.}\frac{1}{2(1+2H)}\int_0^\cdot \Jacobi{t}^{-1}f(\exsol{t})dt
\end{align*}
in the sense of the Skorokhod topology.
Moreover, when $\Gamma$ is one of $110,101$ and $011$, we can see $\sup_m\|\mathcal{I}_\cdot^{(m)}(J^{-1} f(Y),d(V^{\Gamma}))\|_{1}<\infty,\ a.s$.
\end{lmm}

\begin{proof}

At first, by a priori estimate, we have $\sup_m\|\mathcal{I}_\cdot^{(m)}(d(V^{\Gamma}))\|_{1}<\infty\ a.s$. Immediately, the above Lipschitz estimate follows.

Combining Corollary \ref{corollary:estimate110} and estimate of Stieltjes sum, with $\epsilon_1,\epsilon_2$ such that $\epsilon_1+\epsilon_2<H_-$, we have

\begin{align*}
\mathcal{I}_\cdot^{(m)}(J^{-1}f(Y),d(V^{110}))=&\frac{1}{2(1+2H)}\mathcal{I}_\cdot^{(m)}(J^{-1}f(Y),dt)+O(2^{-\epsilon_1 m}),\\
\mathcal{I}_\cdot^{(m)}(J^{-1}f(Y),d(V^{101}))=&-\frac{1-2H}{2(1+2H)}\mathcal{I}_\cdot^{(m)}(J^{-1}f(Y),dt)+O(2^{-\epsilon_1 m}),\\
\mathcal{I}_\cdot^{(m)}(J^{-1}f(Y),d(V^{011}))=&\frac{1}{2(1+2H)}\mathcal{I}_\cdot^{(m)}(J^{-1}f(Y),dt)+O(2^{-\epsilon_1 m}).
\end{align*}
When $m$ goes to infinity, second terms of right-hand side vanish. Hence, we complete the proof.
\end{proof}

\subsection{Calculation of asymptotic error}
\label{subsection:calculation_of_distribution}
In this subsection, we will determine the asymptotic error of the ($k$)-Milstein scheme when $1/( 2\lfloor k/2\rfloor+2)<H< 1/2$ and of the Crank-Nicolson scheme when $1/4<H< 1/2$, respectively.

As we have seen in Section \ref{section:maintheorem}, when $1/(k+1)<H<1$, the ($k$)-Milstein scheme satisfies Condition ($\mathfrak{A}$).
Also, when $k$ is even and $1/(k+2)<H\leq 1/(k+1)$, it satisfies Condition ($\mathfrak{B}$).
Similarly, when $H>1/3$, the Crank-Nicolson scheme satisfies Condition ($\mathfrak{A}$) and when $1/4<H\leq 1/3$, it satisfies Condition ($\mathfrak{B}$).

When each scheme satisfies Condition ($\mathfrak{A}$), firstly we have to decompose $\mathcal{I}^{(m)}(\Jacobi{\cdot}^{-1},d(\epsilon^{1,*,(m,A)}))$, secondly to take the limit of each term, and  thirdly to sum the limits.
When schemes satisfy only Condition ($\mathfrak{B}$), other than the three steps, we have to check H\"{o}lder condition.

However, we have already seen the sufficient statements in Section \ref{subsection:decompose_epsilon}, \ref{subsection:limit_of_individual_sum}, except terms which form $\fBminc{s,t}^k$.
So, in this subsection, we need only to take the limits of

\begin{align*}
&\mathcal{I}_\cdot^{(m)}(J^{-1},d((\cdot_2-\cdot_1)^{-((2l+2)H-1)}(\tilde{g}_{2l+1,1}^{M}(Y_{\cdot_1}) \fBmdotunit^{2l+1}+\tilde{g}_{2l+2,2}^{M}(Y_{\cdot_1})\fBmdotunit^{2l+2}))),\\
&\mathcal{I}_\cdot^{(m)}(J^{-1},d((\cdot_2-\cdot_1)^{-(3H-1/2)}(\tilde{g}_3^{CN}(Y_{\cdot_1}) \fBmdotunit^3+\tilde{g}_4^{CN}(Y_{\cdot_1})\fBmdotunit^4))),\\
&\mathcal{I}_\cdot^{(m)}(J^{-1} f(Y),d((\sigma b'-b\sigma')(Y_{\cdot_1})\fBmitin_{\cdot_1,\cdot_2}^{10*}-b(\sigma\sigma')'(Y_{\cdot_1})\fBmitin_{\cdot_1,\cdot_2}^{110}\\
&-\sigma(b\sigma')'(Y_{\cdot_1})\fBmitin_{\cdot_1,\cdot_2}^{101}-\sigma'(\sigma b')'(Y_{\cdot_1})\fBmitin_{\cdot_1,\cdot_2}^{011})).
\end{align*}

\begin{lmm}\label{lemma:hermite_variation_of_polynomial}
Let $\tilde{g}_{k,1}^{M},\tilde{g}_{k,2}^{M},\ \tilde{g}_{3}^{CN},\ \tilde{g}_4^{CN}$ be defined in Section \ref{subsection:decompose_epsilon}. Also, we set
\begin{equation*}
\overline{g}_{2l+1}^M = \tilde{g}_{2l+2,2}^{M}-\frac{1}{2}\mathcal{V}\tilde{g}_{2l+1,1}^M,
\end{equation*}
where $\mathcal{V}:f\mapsto \sigma f'-\sigma'f$ defined in Section \ref{section:notation}.
Suppose each of them belongs to $C_b^2$.

Then, the following limits hold:

\begin{itemize}
\item when $1/(2l)<H<1/2$ and $l\geq 1$,
\begin{equation*}
\mathcal{I}_\cdot^{(m)}(\Jacobi{\cdot}^{-1},d(2^{(2lH-1)m}(\tilde{g}_{2l,1}^{M}(Y_{\cdot_1}) \fBmdotunit^{2l})))\xrightarrow{a.s.} \moment{2l} \int_0^\cdot \tilde{g}_{2l,1}^{M}(\exsol{t})dt,
\end{equation*}

\item when $1/(2l+2)<H<1/2$ and $l\geq 1$,

\begin{align}
&\mathcal{I}_\cdot^{(m)}(\Jacobi{\cdot}^{-1},d((\cdot_2-\cdot_1)^{-((2l+2)H-1)}(\tilde{g}_{2l+1,1}^{M}(Y_{\cdot_1}) \fBmdotunit^{2l+1}+\tilde{g}_{2l+2,2}^{M}(Y_{\cdot_1})\fBmdotunit^{2l+2})))\nonumber\\
\xrightarrow{a.s.}& \moment{2l+2}\int_0^\cdot\overline{g}_{2l+1,1}^{M}dt,\label{MB2l+1}
\end{align}

\item when $1/4<H<1/2$,

\begin{align}
&\mathcal{I}_\cdot^{(m)}(\Jacobi{\cdot}^{-1},d((\cdot_2-\cdot_1)^{-(3H-1/2)}(\tilde{g}_3^{CN}(Y_{\cdot_1}) \fBmdotunit^3+\tilde{g}_4^{CN}(Y_{\cdot_1})\fBmdotunit^4)))\nonumber\\
\xrightarrow{d}&C_{(3)}\int_0^\cdot \tilde{g}_3^{CN}(\exsol{t})dW_t^{(3)}\label{CNB3}
\end{align}

\end{itemize}
in the sense of the Skorokhod topology, where $W^{(3)}$ is standard Brownian motion independent of $B$ and independent of each other defined in Theorem \ref{theorem:Theorem_of_distribution_CN}.
Furthermore, discrete $H^c$-H\"older norm of left-hand side of (\ref{MB2l+1}) is uniformly bounded with respect to m.
Also, for arbitrary $\lambda>0$, discrete $(1/2-\epsilon)$-H\"older norm of left-hand side or (\ref{CNB3}) multiplied with $2^{-\lambda m}$ is uniformly bounded with respect to m.

\end{lmm}

\begin{proof}
From (\ref{decompose_power_B}), we can decompose of the left-hand side of limits. For example, for (\ref{MB2l+1}), we can see that

\begin{align*}
&\tilde{g}_{2l+1,1}^{M}(\exsoldis{r})\fBmunit^{2l+1}+\tilde{g}_{2l+2,2}^{M}(\exsoldis{r})\fBmunit^{2l+2}\\
=&\timeunit^{(2l+1)H-1/2}\tilde{g}_{2l+1,1}^{M}(\exsoldis{r})\sum_{j=0}^{l-1}\frac{(2l+1)!}{2^j(2l+1-2j)!j!}V_{\bitime{r},t}^{(2l+1-2j)}\\
&+\timeunit^{(2l+2)H-1/2}\tilde{g}_{2l+2,1}^{M}(\exsoldis{r})\sum_{j=0}^{l}\frac{(2l+2)!}{2^j(2l+2-2j)!j!}V_{\bitime{r},t}^{(2l+2-2j)}\\
&+\timeunit^{2lH}\moment{2l+1}(\tilde{g}_{2l+1,1}^M(\exsoldis{r})\fBmunit +\frac{1}{2}\mathcal{V}\tilde{g}_{2l+1,1}^M(\exsoldis{r})\fBmunit^2)\\
&+\timeunit^{(2l+2)H}\moment{2l+1}(\tilde{g}_{2l+2,1}^M-\frac{1}{2}\mathcal{V}\tilde{g}_{2l+1}^M)(\exsoldis{r})\\
&-\timeunit^{(2l+1)H-1/2}\moment{2l+1}\frac{1}{2}\mathcal{V}\tilde{g}_{2l+1,1}^M(\exsoldis{r})V_{\bitime{r},t}^{(2)}.
\end{align*}

In the same way, we can obtain the decomposition for (\ref{CNB3}).
Then, we remark that $\tilde{g}_4^{CN}-\frac{1}{2}\mathcal{V}\tilde{g}_3^{CN}\equiv 0$.
Limits of each term is obtained in Section \ref{subsection:limit_of_individual_sum}, so that we get the limit immediately.

On the other hand, from Corollary \ref{corollary:BdiscreteHoelder}, we can see that discrete $(H^c)$- H\"older norm of the left-hand side of (\ref{MB2l+1}) and (\ref{CNB3}) is uniformly bounded with respect to $m$.
\end{proof}

\begin{lmm}\label{lmm:compound_coefficient}
For $\sigma,b\in C_b^3$, we have

\begin{align*}
2^{-2Hm}&(\mathcal{I}_\cdot^{(m)}(J^{-1} f(Y),d((\sigma b'-b\sigma')(Y_{\cdot_1})\fBmitin_{\cdot_1,\cdot_2}^{10*}))+\mathcal{I}_\cdot^{(m)}(J^{-1} f(Y),d(b(\sigma\sigma')'(Y_{\cdot_1})\fBmitin_{\cdot_1,\cdot_2}^{110}))\\
&+\mathcal{I}_\cdot^{(m)}(J^{-1} f(Y),d(\sigma(b\sigma')'(Y_{\cdot_1})\fBmitin_{\cdot_1,\cdot_2}^{101})+\mathcal{I}_\cdot^{(m)}(J^{-1} f(Y),d(\sigma'(\sigma b')'(Y_{\cdot_1})\fBmitin_{\cdot_1,\cdot_2}^{011})))\\
\xrightarrow{a.s.}&-\int_0^\cdot \Jacobi{t}^{-1}\left(\frac{6H-1}{4(1+2H)}(\sigma\sigma''b+\sigma\sigma'b')+\frac{3-2H}{4(1+2H)}((\sigma')^2b+\sigma^2b'')\right)dt
\end{align*}
\end{lmm}

\begin{proof}
From limit of each term, we have

\begin{align*}
2^{-2Hm}&(\mathcal{I}_\cdot^{(m)}(J^{-1} f(Y),d((\sigma b'-b\sigma')(Y_{\cdot_1})\fBmitin_{\cdot_1,\cdot_2}^{10*}))+\mathcal{I}_\cdot^{(m)}(J^{-1} f(Y),d(b(\sigma\sigma')'(Y_{\cdot_1})\fBmitin_{\cdot_1,\cdot_2}^{110}))\\
&+\mathcal{I}_\cdot^{(m)}(J^{-1} f(Y),d(\sigma(b\sigma')'(Y_{\cdot_1})\fBmitin_{\cdot_1,\cdot_2}^{101})+\mathcal{I}_\cdot^{(m)}(J^{-1} f(Y),d(\sigma'(\sigma b')'(Y_{\cdot_1})\fBmitin_{\cdot_1,\cdot_2}^{011})))\\
\to& \int_0^\cdot \Jacobi{t}^{-1}(f_{10}+f_{110}+f_{101}+f_{011})(\exsol{t})dt,
\end{align*}
where

\begin{align*}
f_{10}=&-\frac{1-2H}{4(1+2H)}\mathcal{V}(\sigma b'-b\sigma'),&f_{110}=&-\frac{1}{2(1+2H)}b(\sigma\sigma')',\\
f_{101}=&\frac{1-2H}{2(1+2H)}\sigma(b\sigma')',&f_{011}=&-\frac{1}{2(1+2H)}\sigma(\sigma b')'.
\end{align*}
In this case, we can expand each term as
\begin{align*}
\mathcal{V}(\sigma b'-\sigma'b)=&(\sigma \partial-\sigma')(\sigma b'-\sigma'b)\\
=&\sigma \sigma' b'+\sigma^2 b''-\sigma \sigma' b'-\sigma\sigma''b-\sigma \sigma'b'+(\sigma')^2b\\
=&\sigma^2 b''-\sigma\sigma''b-\sigma \sigma'b'+(\sigma')^2b,
\end{align*}
\begin{align*}
b(\sigma\sigma')'=\sigma\sigma''b+(\sigma')^2b,&&\sigma(b\sigma')'=\sigma \sigma''b+\sigma \sigma'b',&&\sigma(\sigma b')'=\sigma^2b''+\sigma\sigma'b'.
\end{align*}

Hence, we obtain
\begin{align*}
&f_{10}+f_{110}+f_{101}+f_{011}\\
=&-\frac{6H-1}{4(1+2H)}\sigma\sigma''b-\frac{3-2H}{4(1+2H)}(\sigma')^2b-\frac{6H-1}{4(1+2H)}\sigma\sigma'b'-\frac{3-2H}{4(1+2H)}\sigma^2b''.
\end{align*}

\end{proof}

\begin{proof}[Proof of Theorem \ref{theorem:Theorem_of_distribution_M} and \ref{theorem:Theorem_of_distribution_CN}]

For each case classified in Section \ref{subsection:decompose_epsilon}, we can check that

\begin{itemize}
\item the limit of $R(m)^{-1}\mathcal{I}^{(m)}(J^{-1},d(\epsilon^{1,*,(m,A)}))$ is shown per terms in Section \ref{subsection:limit_of_individual_sum} and Lemma \ref{lemma:hermite_variation_of_polynomial},
\item the limit multiplied with $\Jacobi{\cdot}$ coincides the right-hand side of corresponding cases stated in Theorem \ref{theorem:Theorem_of_distribution_M} and \ref{theorem:Theorem_of_distribution_CN},
\item $R(m)$ and $A$ satisfy that $2^{-m(A-1)}=O(R(m))$,
\item $\sigma$ and $b$ have sufficient smoothness,
\item and the H\"older estimate holds in the case of Condition ($\mathfrak{B}$).
\end{itemize}

Now, by Lemma \ref{lemma:integral10*}, \ref{lemma:integral110} we remark that if $B^{10*},B^{110},B^{101},B^{011}$ appear in main term, we have to require from $b$ the condition that it is in $C_b^3$.
By definition of $\epsilon^{1,*,(m,A)}$, each statements holds when we replace $\epsilon_{\cdot_1,\cdot_2}^{1,*,(m,A)}$ with $J_{\cdot_1}J_{\cdot_2}^{-1}\epsilon_{\cdot_1,\cdot_2}^{1,(m,A)}$.
From the fact that $R(m)\mathcal{I}^{(m)}(J^{-1},d(J_{\cdot_1}J_{\cdot_2}^{-1}\epsilon_{\cdot_1,\cdot_2}^{1,(m,A)}))=R(m)Z^{M,(A)}$, $\nusol{}$ satisfies the assumption of Theorem \ref{theorem:maintheorem}.
Substituting the limit in Lemma \ref{lemma:hermite_variation_of_polynomial} and \ref{lmm:compound_coefficient} with Theorem \ref{theorem:maintheorem}, we complete to prove.
\end{proof}

\bibliographystyle{plain}
\bibliography{C:/texlive/texmf-local/bibtex/bib/local/test}

\end{document}